\documentclass[11pt,reqno,final]{amsart}

\usepackage{amsfonts,amsmath,amssymb,amsthm}
\usepackage{mathrsfs,mathtools}
\usepackage{enumerate,enumitem}
\usepackage{xspace}
\usepackage{url}
\usepackage{hyperref}

\usepackage{graphicx}
\usepackage{epstopdf,epsfig,subfigure}
\usepackage{tikz}
\usepackage{pgfplots}
\usepackage{proofread}
\pgfplotsset{compat=newest}
\usepackage{bm}

\usepackage{xcolor}
\usepackage{curve2e}
\usepackage{booktabs}
\usepackage{arydshln}
\usepackage{verbatim}
\usepackage{cite}
\usepackage{scalefnt}
\usepackage{lineno}
\usetikzlibrary{plotmarks,external}
\usepackage{booktabs}

\theoremstyle{plain}

\newtheorem{theorem}{Theorem}[section]
\newtheorem{corollary}[theorem]{Corollary}

\newtheorem{lemma}[theorem]{Lemma}
\newtheorem*{mainresult*}{Main Result}
\newtheorem*{maincorollary*}{Main Corollary}
\newtheorem{assumption}[theorem]{Assumption}

\theoremstyle{definition}
\newtheorem{definition}[theorem]{Definition}
\newtheorem{remark}[theorem]{Remark}

\newtheorem{example}[theorem]{Example}

\numberwithin{equation}{section}

\newcommand{\linspan}{\mathop{\rm span}\nolimits}
\newcommand{\Ker}{\mathop{\rm Ker}\nolimits}

\newcommand{\rest}{\left.\kern-2\nulldelimiterspace\right|_}
\newcommand{\norm}[2]{\left|#1\right|_{#2}}

\newcommand{\Id}{{\mathbf1}}
\newcommand{\indf}{1}

\newcommand{\ex}{\mathrm{e}}
\newcommand{\p}{\partial}
\newcommand{\ed}{\mathrm d}

\newcommand*{\Bigcdot}{\raisebox{-.25ex}{\scalebox{1.25}{$\cdot$}}}

\newcommand{\clB}{{\mathcal B}}
\newcommand{\clC}{{\mathcal C}}
\newcommand{\clD}{{\mathcal D}}
\newcommand{\clE}{{\mathcal E}}

\newcommand{\clK}{{\mathcal K}}
\newcommand{\clL}{{\mathcal L}}

\newcommand{\clN}{{\mathcal N}}

\newcommand{\clP}{{\mathcal P}}

\newcommand{\clU}{{\mathcal U}}

\newcommand{\bbN}{{\mathbb N}}

\newcommand{\bbR}{{\mathbb R}}

\newcommand{\fkB}{{\mathfrak B}}

\newcommand{\rmD}{{\mathrm D}}

\newcommand{\bfn}{{\mathbf n}}

\newcommand{\rmd}{{\mathrm d}}

\newcommand{\fki}{{\mathfrak i}}

\newcommand{\fkm}{{\mathfrak m}}

\newcommand{\ovlineC}[1]{\overline C_{\left[#1\right]}}

\definecolor{DarkBlue}{rgb}{0,0.08,0.45}
\definecolor{DarkRed}{rgb}{.65,0,0}
\definecolor{applegreen}{rgb}{0.55, 0.71, 0.0}

\newcounter{mymac@matlab}
\newcommand{\matlab}{MATLAB%
   \ifnum\value{mymac@matlab}<1%
   \textregistered%
   \setcounter{mymac@matlab}{1}%
   \fi%
  }

\newcommand{\black}{ \color{black} } 
\newcommand{\blue}{ \color{blue} }

\begin{document}
\title{ Learning an optimal feedback operator semiglobally stabilizing semilinear parabolic equations}
\author{Karl Kunisch$^{\tt1,2}$}
\author{S\'ergio S.~Rodrigues$^{\tt1}$}
\author{Daniel Walter$^{\tt1}$}
\thanks{
\vspace{-1em}\newline\noindent
\newline\noindent
$^{\tt1}$ Johann Radon Institute for Computational and Applied Mathematics,
   Altenbergerstr. 69, 4040 Linz, Austria.
\newline
$^{\tt2}$ Institute for Mathematics and Scientic Computing,
 Karl-Franzens-Universit\"at,
 Heinrichstr. 36, 8010 Graz, Austria.
\newline
{\sc Emails}:
{\small\tt  karl.kunisch@uni-graz.at},\quad ({\small\tt sergio.rodrigues,daniel.walter)@oeaw.ac.at}
 }

\begin{abstract}
 Stabilizing feedback operators are presented which depend only on the orthogonal projection of
the state onto the finite-dimensional control space.\black A class of monotone feedback operators mapping the finite-dimensional
control space into itself is considered.  The special case of the scaled identity operator\black is included.
Conditions are given on the set of actuators and on the magnitude of the monotonicity,
which guarantee the semiglobal stabilizing property of the feedback for a class semilinear parabolic-like equations.
 Subsequently an optimal feedback control minimizing the  quadratic energy cost
is computed by a deep neural network,  exploiting the fact that the feedback depends only on a finite dimensional
component of the state. Numerical simulations demonstrate the stabilizing
performance of explicitly scaled orthogonal projection feedbacks, and of  deep neural network feedbacks.
\end{abstract}

\maketitle

\pagestyle{myheadings} \thispagestyle{plain} \markboth{\sc K. Kunisch, S. S.
Rodrigues, and D. Walter}{\sc learning optimal stabilizing feedback for semilinear parabolic equations}

\section{Introduction}

We consider evolutionary semilinear parabolic-like equations, for time~$t\ge0$, as
 \begin{align}\label{sys-y-intro}
 \dot y +Ay+A_{\rm rc}y +\clN(y) =0,\qquad y(0)= y_0,
\end{align}
evolving in a Hilbert space~$H$,  where~$A$ and~$A_{\rm rc}=A_{\rm rc}(t)$ are
a time-independent linear diffusion-like operator and a
time-dependent linear reaction-convection-like operator. Further, $\clN=\clN(t)$
is a time-dependent nonlinear operator for which conditions will be chosen so that
existence and uniqueness of weak solutions for system~\eqref{sys-y-intro} hold true
for sufficiently small time $t\in[0,\tau)$, $0<\tau\le+\infty$.

System~\eqref{sys-y-intro} can be unstable, meaning  that the norm~$\norm{y(t)}{H}$
of its solution may diverge exponentially to~$+\infty$
as~$t\to\tau^-$.  This  motivates the investigation of  feedback stabilization
of~\eqref{sys-y-intro}. We shall concentrate on the case where this can be achieved by finitely many actuators.

Our  first goal consists in  semiglobal exponential stabilization  of the system.  More precisely,
for arbitrary  $R>0$ and~$\mu>0$, we will prove that for a suitable~$\lambda>0$,  a set of actuators
\[
{U_{M}}\coloneqq\{\Phi_j\mid 1\le j\le M_\sigma\}\subset H,
\qquad \clU_M\coloneqq\linspan U_M,\qquad\dim{\clU_{M}}=M_\sigma,
\]
with $M_\sigma \in \mathbb{N}$, and a suitable (possibly, time-dependent)\black
feedback operator~$\clK_M(t,\Bigcdot)\colon\clU_M\to\clU_M$, the  solution of
\begin{align}\label{sys-y-intro-K}
 \dot y +Ay+A_{\rm rc}y +\clN(y) =\clK_M(t,P_{\clU_{M}} y),\black\qquad y(0)= y_0,
\end{align}
satisfies
\begin{align}\label{goal-intro}
 \norm{y(t)}{H}\le \ex^{-\mu (t-s)}\norm{y(s)}{H},\quad\mbox{for all}\quad t\ge s\ge0\quad
 \mbox{and all}\quad y_0\in H
 \quad\mbox{with}\quad\norm{y_0}{H}<R.
\end{align}
Above,~$M_\sigma=\sigma(M)$ is a
 strictly increasing sequence of positive integers.
Note that at each instant of time the control input is a linear combination of the actuators in the form
\begin{equation}\label{Feed-o-lc-intro}
\clK_M(t,P_{\clU_{M}} y(t))
=\textstyle\sum\limits_{j=1}^{M_\sigma}u_j(t)\Phi_j,
\end{equation}
for ~$(u_1(t),u_2(t),\cdots,u_{M_\sigma}(t))\in \bbR^{M_\sigma}$.

\begin{remark}
 For the moment the reader\black may think of~$M_\sigma=M$. The reason we consider a general
 subsequence~$M_\sigma=\sigma(M)$ is due to the fact that
 in applications of the theoretical results to concrete examples
 of parabolic equations it is more convenient to have a result
 stated for a general subsequence, as we will see later in Section~\ref{S:parabolic-OK}
 where we will take the subsequence~$M_\sigma=\sigma(M)=M^d$,
 for a parabolic equation evolving in a rectangular spatial domain~$\Omega\subset\bbR^d$.
 \end{remark}

\begin{definition}\label{D:monot}
For a given time-dependent operator
$\clK_M(t,\Bigcdot)\colon\clU_{M}\to \clU_{M}$,
we write $\clK_M\preceq-\overline\lambda\Id$, if the
inequality~$(\clK_M(t, p),p)_H\le -\overline\lambda\norm{p}{H}^2$ holds
for  every~$p\in\clU_M$, and almost every~$t>0$.
\end{definition}

We consider a
 general class of feedback  laws in the form~$\clK_{M}\circ P_{\clU_M}$, satisfying
\begin{subequations}\label{FeedK-intro}
\begin{align}
 &\clK_{M}:[0,\infty)\times \clU_{M}\to  \clU_{M},\quad&&\clK_{M}(t,\Bigcdot)\in\clC(\clU_{M},\clU_{M}),\\
&\clK_{M}(t,0)=0,\quad&&\clK_{M}(t,\cdot)\preceq-\overline\lambda\Id,\quad\overline\lambda>0.
\end{align}
\end{subequations}

 Under suitable assumptions on~$(A,A_{\rm rc},\clN,\clU_{M})$ and
 an extra regularity assumption on~$\clK_M$,
 which will be precised
 in Section~\ref{S:assumptions} we will have the following result.
\begin{mainresult*}
 Let~$R>0$, and~$\mu>0$ be given. Let~$\clK_M\le-\overline\lambda\Id$.
 Then, if~$M\in\bbN_0$ and~$\overline\lambda>0$  are large enough,
the solution of~\eqref{sys-y-intro-K}
  satisfies~\eqref{goal-intro}
In particular, $t\mapsto \norm{y(t)}{H}$ is strictly decreasing at time~$t=s$, if~$\norm{y(s)}{H}\ne0$.
\end{mainresult*}
Note also that the operator~$\clK_{M}(t,\cdot)$ in~\eqref{FeedK-intro} may be nonlinear. We give now examples
of such operators which are linear. The simplest class of such feedbacks consists  of the scaled
identity operator:
for each~$\lambda>0$, the operator
\begin{equation}\label{Feed-orth-K-intro}
\breve\clK_M^{\lambda}\coloneqq-\lambda \Id,
\end{equation}
is in the class~\eqref{FeedK-intro} above with~$\overline\lambda=\lambda$, since
$(\breve\clK_M^{\lambda}(t, p),p)_H= -\lambda\norm{p}{H}^2$.

A second  class of examples consists, as we shall see in Section~\ref{sS:RmksExplicitFeed}, of the operators
\begin{equation}\label{Feed-rho-K-intro}
\clK_M^{\lambda,\rho}\coloneqq-\lambda P_{\clU_{M}}^{\widetilde \clU_{M}^\perp}
A^\rho P_{\widetilde \clU_{M}}^{\clU_{M}^\perp}\Bigr|_{\clU_M}
\end{equation}
with~$\rho\le1$, where $P_{\clU_{M}}^{\widetilde \clU_{M}^\perp}\colon H\to\clU_{M}$
denotes (a suitable extension of) the oblique projection in~$H$ onto~$\clU_{M}\coloneqq\linspan{U_{M}}$ along
the orthogonal~$\widetilde \clU_{M}^\perp$ to the
space~$\widetilde \clU_{M}\coloneqq\linspan{\widetilde U_{M}}$, where~${\widetilde U_{M}}$ is a set of
auxiliary functions
\[
{\widetilde U_{M}}\coloneqq\{\widetilde \Phi_j\mid 1\le j\le M_\sigma\}\subset V,
\qquad \widetilde\clU_M\coloneqq\linspan\widetilde U_M,\qquad\dim{\widetilde\clU_{M}}=M_\sigma,
\]
where~$V\subset H$ is a Hilbert space.
It is assumed  that~$H=\clU_{M}\oplus \widetilde\clU_{M}^\perp$, and  $P_{\widetilde \clU_{M}}^{\clU_{M}^\perp}\colon H\to\clU_{M}$ stands for
the oblique projection in~$H$ onto~$\widetilde \clU_{M}$ along~$\clU_{M}^\perp$.

Here
we just consider briefly the case~$\rho=0$. Recalling that~$(P_{\clU_{M}}^{\widetilde \clU_{M}^\perp})^*=P_{\widetilde \clU_{M}}^{\clU_{M}^\perp}$,
for~$p\in\clU_M$ we find\black
\begin{align*}
(\clK_M^{\lambda,\rho} p,p)_H =
(-\lambda P_{\clU_{M}}^{\widetilde \clU_{M}^\perp} P_{\widetilde \clU_{M}}^{\clU_{M}^\perp} p,p)_H
=-\lambda\norm{P_{\widetilde \clU_{M}}^{\clU_{M}^\perp}p}{H}^2
\le -\lambda\norm{P_{\clU_{M}}P_{\widetilde \clU_{M}}^{\clU_{M}^\perp}p}{H}^2
=-\lambda\norm{P_{\clU_{M}}p}{H}^2
=-\lambda\norm{p}{H}^2,
\end{align*}
 thus\black~$\clK_M^{\lambda,0}$ is monotone with~$\overline\lambda=\lambda$.

As a consequence of the  Main Result above and of the following identities
\begin{align*}
 \breve\clK_M^{\lambda}(P_{\clU_M}y)&=-\lambda P_{\clU_{M}}y,\\
 \clK_M^{\lambda,\rho}(P_{\clU_M}y)&=-\lambda P_{\clU_{M}}^{\widetilde \clU_{M}^\perp} A^\rho
 P_{\widetilde \clU_{M}}^{\clU_{M}^\perp}P_{\clU_M}y
 =-\lambda P_{\clU_{M}}^{\widetilde \clU_{M}^\perp} A^\rho  P_{\widetilde \clU_{M}}^{\clU_{M}^\perp}y,
\end{align*}
we have  the following. \black
\begin{maincorollary*}
 Let~$\rho\le1$, ~$R>0$, and~$\mu>0$ be given. Then,  if~$M\in\bbN_0$ and~$\lambda>0$  are large enough,
the solutions of
\begin{align}\label{sys-y-intro-KId}
 \dot y +Ay+A_{\rm rc}y +\clN(y) =-\lambda P_{\clU_{M}} y,\qquad y(0)= y_0,
\end{align}
and
\begin{align}\label{sys-y-intro-KId-rho}
 \dot y +Ay+A_{\rm rc}y +\clN(y) =-\lambda P_{\clU_{M}}^{\widetilde \clU_{M}^\perp}
 A^\rho P_{\widetilde \clU_{M}}^{\clU_{M}^\perp} y,\qquad y(0)= y_0,
\end{align}
 both satisfy
 \[
  \norm{y(t)}{H}\le\ex^{-\mu (t-s)}\norm{y(s)}{H},\quad\mbox{for all}\quad t\ge s\ge0
  \quad\mbox{and all}\quad y_0\in H\quad\mbox{with}\quad \norm{y_0}{H}<R.
 \]
In particular, $t\mapsto \norm{y(t)}{H}$ is strictly decreasing at time~$t=s$, if~$\norm{y(s)}{H}\ne0$.
\end{maincorollary*}

 \begin{remark}\label{R:orth-new}
Note that Main Result and Main Corollary are semiglobal stabilizability results, that is, we can stabilize the system for any
given initial condition, but the number~$M_\sigma$ of needed actuators and
the magnitude~$\overline\lambda$ of the monotonicity
of the feedback operator
may depend  on the norm of the initial condition. Main Result and Main Corollary will be made precise  in
 Section~\ref{S:stab}, in Theorem~\ref{T:main} and Corollary~\ref{C:main-expli-1}, respectively, including the
 assumptions on~$(A,A_{\rm rc},\clN,\clU_{M},\widetilde\clU_{M})$.
 Main Corollary is novel for nonlinear systems.
 With the feedback as in~\eqref{sys-y-intro-KId} it  is novel also for
 linear systems.
 Main Result with the feedback as in~\eqref{sys-y-intro-KId-rho} has been shown  in~\cite{Rod_pp20-CL}
 for linear systems for the cases~$\rho\in\{0,1\}$.
\end{remark}

Our second goal consists in designing a computational approach which leads to an
optimal feedback law with a  feedback operator $\clK_{M}$ of the form ~\eqref{FeedK-intro}.
Since optimizing with respect to the nonlinear operator $\clK_{M}$ is numerically unfeasible,
a parametrization  $\clK_{M}^{(\theta)}$ based on neural network techniques is chosen.
 It will be constructed by a  variational problem and
 will have the property that\black\!\!
\begin{align}\label{sys-y-intro-Ktheta}
 \dot y +Ay+A_{\rm rc}(t)y +\clN(t,y) = \clK_M^{(\theta)}\left(t,P_{\clU_{M}} y\right)
,\qquad y(0)= y_0,
\end{align}
is exponentially stable.

\subsection{Towards taming the curse of dimensionality}\label{sS:taming-intro}
{ The achievement of} stabilization by finite-dimensional rather than infinite-dimensional
 controls is a significant first step towards simplification of feedback controls both
 for practical as well as for computational considerations. There is, however,
 still a second feature for the feedback operators we are considering. It  becomes
 evident when comparing our
  operators
 \begin{equation}\label{FeedtPA-intro}
y(t)\mapsto \clK_M\left(t,P_{\clU_{M}} y\right),
\end{equation}
to those
 in the more general form
\begin{equation}\label{FeedK-intro-H}
\widehat\clK_{M}(t,\Bigcdot)\colon H\to \clU_{M},\qquad y(t)\mapsto\widehat\clK_{M}(t,y(t)).
\end{equation}
Namely our control  input~$\clK_M\left(t,P_{\clU_{M}} y(t)\right)$, at time~$t\ge0$,
only depends on the projection~$P_{\clU_M}y(t)$ of the state, rather than on all of $y(t)$.
Hence, it is  completely defined by the values it takes at the vectors in the
finite dimensional space~$\clU_M$.\black This
has significant consequences for the computational complexity, as we shall point out below.

A feedback law in the form~\eqref{FeedK-intro-H} is presented in~\cite{Rod20}, which stabilizes
semiglobally nonautonomous parabolic equations of the form~\eqref{sys-y-intro}.
The computation of the optimal feedback operator in the general  form~\eqref{FeedK-intro-H}, for
the classical  quadratic cost functional, leads us to the  Hamilton--Jacobi--Belmann equations.
After spatial discretization of the state equation into a  system of a $N$ ordinary differential
equations, this results in a  partial differential equation with dimension $N$, which is extremely
challenging or impossible to solve in practice, see e.g. \cite{B19, DKK19}, and for this reason,
 alternative approaches are important.

To briefly reflect on the computational complexity, we consider the autonomous case,
where the optimal  feedback operator is known to be independent of the  time variable.
Thus, we are looking for an optimal control vector in~$\bbR^{M_\sigma}$
for each state vector in~$\bbR^{N}$.
Taking a bounded grid box $[-L,L]^N$ centered at zero with~$G_x+1$ points in each spatial
direction results in ~$(G_x+1)^N$ points.  Computing and saving the feedback control means
to compute and save an array of dimension ~${M_\sigma\times (G_x+1)^N}$ making the
computation unfeasible for  finite element approximations where N is large.
This issue is usually refereed to as {\em the curse of dimensionality}.
Notice, for example, if we only take the  vertices and center  of the
box so that~$G_x=2$, and~$N=60$ we find that $(G_x+1)^N=3^{60}>(27)^{20}$ which
is  already a large number (e.g., larger than the inverse of the MATLAB (accuracy)
epsilon $\tt eps\approx 10^{-16}$).

 This dependence on the dimension~$N$, can be avoided by using the structure of the
feedback we propose,~$\clK_M\left(t,P_{\clU_{M}}\Bigcdot\right)$,
which allow us to focus on a
finite-dimensional vector space~$\clU_{M}$.  Note that~$\clU_{M}$ is independent of the discretization.
After we have proven that feedbacks of the form~\eqref{FeedK-intro},
are stabilizing system~\eqref{sys-y-intro-K},
we look for an optimal feedback~$\clK_{M}^\bullet\colon \clU_{M}\to \clU_{M}$ in the form~\eqref{FeedK-intro},
which we need to compute, for vectors in the finite dimensional space~$\clU_M$.
Note that this implies that, after discretization of (a box $[-L,L]^{M_\sigma}$ in) the subspace~$\clU_M$,
we shall be looking for an array of dimension ~$ M_\sigma\times (G_x+1)^{M_\sigma}$.
The size of the arrays increases exponentially with respect to  the number of actuators,
but it is independent of the number~$N$
of degrees of freedom of the finite element discretization.
Therefore, if the number of actuators is relatively small, the computation of the optimal
control becomes feasible\black for finite-element discretizations of the state-equation.
Thus the classical {\em curse of dimensionality} depending on the accuracy of
the discretization has been tamed, by considering the subclass~$\clK_{M}\colon\clU_{M}\to \clU_{M}$,
as  in~\eqref{FeedK-intro}.
 But, if~$ M_\sigma\times (G_x+1)^{M_\sigma}$ becomes prohibitively large,\black then further steps have to be taken.

Concerning the computation of (sub-)optimal  feedbacks in the subspace~$\clU_{M}$
one approach is to interpolate the approximations computed at each single point $p$ of the grid,
by solving the corresponding open loop optimal control problems for ``large'' finite-time horizons.
Here we follow a different approach by computing the feedback through
training of the neural network, which  is used to approximate $\clK_{M}$
following a technique introduced in ~\cite{KunischWalter_arx20}.

 \subsection{Illustrating example.}\label{sS:illust-parabolic}
The stabilizability results will be guaranteed  under general assumptions on the system
operators~$A$, $A_{\rm rc}$, and~$\clN$.  They are satisfied, in particular,
for the following semilinear scalar parabolic equation, under either Dirichlet or Neumann boundary conditions.
It will follow that for arbitrary ~$R>0$,  the system
\begin{subequations}\label{sys-y-parab-intro}
\begin{align}
 &\tfrac{\p}{\p t} y +(-\nu\Delta+\Id) y+ay +b\cdot\nabla y - \norm{y}{\bbR}y
 =-\lambda P_{\clU_{M}} y,\\
  &\fkB y\rest{\Gamma}=0,\quad
y(0)=y_0,\qquad \norm{y_0}{V}\le R,
       \end{align}
 \end{subequations}
is exponentially stable, for large enough~$\lambda$ and~$M$, depending on~$R$.
Above we can also take more general polynomial-like nonlinearities, as we shall see later on.
Here the state~$y$ is assumed to be defined in a bounded connected  spatial domain
~$\Omega\in\bbR^d$ assumed to be either
smooth or a convex polygon, $d\in\{1,2,3\}$.\black The state
is a function $y=y(x,t)$, defined for~$(x,t)\in  \Omega\times(0,+\infty) $.\black
The operator~$\fkB$ imposes the conditions
at the boundary~$\Gamma=\p\Omega$ of~$\Omega$, where
\begin{align}
  \fkB &=\Id,&&\quad\mbox{for Dirichlet boundary conditions},\\
  \fkB &=\bfn\cdot\nabla=\tfrac{\p}{\p\bfn},&&\quad\mbox{for Neumann boundary conditions,}
\end{align}
and~$\bfn=\bfn(\overline x)$ stands for the outward unit normal vector to~$\Gamma$,
at~$\overline x\in\Gamma$. The functions $a$ and $b$ depend on space and time, and are assumed to satisfy
\begin{align}
 & a\in L^\infty(\Omega\times(0,+\infty))\quad\mbox{and}\quad b\in L^\infty(\Omega\times(0,+\infty))^d.\label{assum.abf.parab}
\end{align}
By defining,  for Dirichlet, respectively Neumann boundary conditions, the spaces
\begin{align*}
H^2_\fkB(\Omega)&\coloneqq\{h\in  H^2(\Omega)\mid \fkB h\rest\Gamma=0\},\mbox{ for }
\fkB\in\{\Id,\tfrac{\p}{\p\bfn}\},
\intertext{and}
V_\Id(\Omega)&\coloneqq \{h\in  H^1(\Omega)\mid h\rest\Gamma=0\},\qquad
V_{\frac{\p}{\p\bfn}}(\Omega)\coloneqq H^1(\Omega),
\end{align*}
system ~\eqref{sys-y-parab-intro} can be expressed in the form~\eqref{sys-y-intro-K}.
For this purpose, we set
\[
 H\coloneqq L^2(\Omega),\quad V\coloneqq V_{\fkB},\quad\mbox{and}\quad \rmD(A)\coloneqq H^2_\fkB(\Omega),
\]
and the operators~$ A\colon\rmD(A)\to H$, ~$ A_{\rm rc}\colon V\to H$,   and~$\clN\colon\rmD(A)\to H$ as
\[
 A\coloneqq -\nu\Delta+\Id,\quad A_{\rm rc}\coloneqq a\Id +b\cdot\nabla,
 \quad\mbox{and}\quad \clN(y)\coloneqq - \norm{y}{\bbR}y.
\]

 It remains to choose the set of actuators~$U_{M}$. To guarantee stability,
 they will need  to satisfy an appropriate ``richness" condition. We stress  that without control system~\eqref{sys-y-parab-intro} may have solutions which blow up
in finite time; see for example~\cite{Levine73,Ball77,MerleZaag98}.

Our motivation to consider the signed nonlinearity~$-\norm{y}{\bbR}y$ is mainly academic and is
 due to the fact that it is actively pushing towards  destabilization of the system. It will always be competing with the stabilizing
 feedback operator, which will be actively pushing towards the stabilization of the system. Indeed,
 by multiplying~\eqref{sys-y-parab-intro} by~$2y$, we find that
 \[
 \tfrac{\rmd}{\rmd t}\norm{y}{H}^2=-2\langle(-\nu\Delta+\Id) y+ay +b\cdot\nabla y,2y\rangle_{V',V}
 -2\langle-\norm{y}{\bbR}y,y\rangle_{V',V}
 -2\langle\lambda P_{\clU_{M}} y,y\rangle_{V',V}
 \]
 and since~$-2\langle\lambda P_{\clU_{M}} y,y\rangle_{V',V}=-2\lambda \norm{P_{\clU_{M}} y}{H}^2<0$, we
 see that the feedback operator contributes towards a decrease of the norm~$\norm{y}{H}^2$, while
 the nonlinearity contributes towards an increase of the same norm,
 since~$-2\langle-\norm{y}{\bbR}y,y\rangle_{V',V}=2\int_{\Omega}\norm{y}{\bbR}^3\rmd\Omega>0$.
Hence, this  nonlinearity can be seen as
 a ``good test'' for the stabilizing feedback operator.

 Our results will cover polynomial   nonlinearities, up to order 2 for $d$-dimensional spatial domains, with $d\in\{1,2,3\}$. Such nonlinearities appear in  population dynamics models, as the Fisher
 model~\cite{Fisher37,OlmosShizgal06} for biology
 gene propagation with the nonlinearity~$y(y-1)$
  and the Shigesada--Kawasaki--Teramoto model~\cite{ShigesadaKawaTera79,Le20}
  for two interacting biological species, whose nonlinearity includes a (vector) polynomial  of degree~$2$. We can also mention the Schl\"ogl
 model for chemical reactions with a cubic nonlinearity~$y(y-c_1)(y-c_2)$~\cite{Schloegl72,GugatTroeltzsch15}. Nonlinearities involving the absolute value of the state appear in the Ginzburg-Landau model
for the hydrodynamics of vortex
fluids in superconductors, with nonlinearity having the leading term~$\norm{y}{\bbR}^2y$; see~\cite{E94,GrishakovDegtDegtElesinKruglov12}.
Such cubic nonlinearities are covered by our results
for~$1$-dimensional spatial domains.

\subsection{On the literature}
We discuss a list of works on the stabilization problem we investigate.
The list is not exhaustive and thus we refer the reader also to the references within the listed works.

In the case of autonomous systems the spectral properties of the time-independent
operator~$A+A_{\rm rc}$ can be used to derive stabilizability results by means of a finite number of internal
actuators as, see ~\cite{BarbuTri04} for example. Though we do not address,
in the present manuscript the case of boundary controls, again spectral properties can be used
to derive stabilizability results by means of a finite number
of boundary actuators, see~\cite{BarbuLasTri06,Barbu12,BadTakah11,Raymond19}.

In the case of nonautonomous systems the spectral properties are not (or seem not to be)
an appropriate tool to
investigate stabilizability properties as suggested by the examples in~\cite{Wu74}.
In~\cite{BarRodShi11} suitable truncated observability inequalities are used to derive
open-loop stabilizability results for
nonautonomous systems, and subsequently it is shown that a stabilizing feedback operator
can be obtained through the solution of
an operator differential Riccati equation.
Since the computation of such Riccati equations can be a difficult numerical task,
 the feedback operators in~\eqref{sys-y-intro-KId}, being explicitly given, are attractive for applications.\black
 Explicit feedbacks for stabilization of nonautonomous parabolic equations involving oblique
 projections were introduced in~\cite{KunRod19-cocv} for linear systems. The proposed feedback control is given by
\begin{align}\label{FeedKunRod}
 \overline\clK_{M}=P_{\clU_{M}}^{\widetilde\clU_{M}^\perp}
 \Bigl(Ay+A_{\rm rc}y-\lambda y\Bigr),
 \end{align}
and the auxiliary space~$\widetilde\clU_{M}=\clE_{M}$ is spanned by a  set of~$M_\sigma$
eigenfunctions of the diffusion operator~$A$.
An analogous linear feedback operator is used in~\cite{KunRod19-dcds} to stabilize coupled
parabolic-{\sc ode} systems, and
in~\cite{AzmiRod20} to stabilize damped wave equations.
In~\cite{Rod20}, the analogous feedback
\begin{align}\label{FeedRod-eect}
 \overline\clK_{M}=P_{\clU_{M}}^{\widetilde\clU_{M}^\perp}
 \Bigl(Ay+A_{\rm rc}y+\clN(y)-\lambda y\Bigr),
 \end{align}
is proven to stabilize semilinear parabolic equations, for ~$M$ large enough, depending on the norm~$\norm{y_0}{V}$ of the initial condition.
In~\cite{Rod_pp20-CL} it is proven that  the feedbacks as in~\eqref{sys-y-intro-KId-rho},\black
\begin{align}\label{FeedRod-CL}
\clK_M^{\lambda,\rho}(P_{\clU_M}y)=-\lambda P_{\clU_{M}}^{\widetilde \clU_{M}^\perp}
A^\rho P_{\widetilde \clU_{M}}^{\clU_{M}^\perp} y,\qquad\rho\le1,
\end{align}
are able to stabilize the corresponding linear  systems.
Actually in~\cite{Rod_pp20-CL}, only the cases~$\rho\in\{0,1\}$ are considered.
Since it comes naturally here we consider the general case $\rho\le 1$.
Furthermore,  for feedbacks as in~\eqref{FeedRod-CL}, in~\cite{Rod_pp20-CL} it is shown  that we
can choose the auxiliary spaces~$\widetilde \clU_{M}$ in a more ad-hoc explicit way. In particular, we do not
need to know/compute the eigenfunctions of the diffusion operator.\black

In the present work we prove that  a class of feedbacks as~\eqref{FeedK-intro}, in the form~$ \clK_M \circ P_{\clU_{M}}$,
which includes~\eqref{FeedRod-CL},
are also able to stabilize a class of  nonlinear systems, for which the
Cauchy problem  is well posed for the uncontrolled system in the sense of weak solutions.

\begin{remark}
We do not know
whether  any of these feedbacks in~\eqref{FeedRod-CL} is able to stabilize
also general nonlinear systems for which weak solutions do not necessarily exist.
However, the feedback in~\eqref{FeedRod-eect} is able to stabilize a larger class of nonlinearities
in the case where so called strong solutions exist for initial conditions in a suitable subspace~$V\subset H$,
and where
weak solutions may not exist for initial conditions living in~$H$.
However, the stabilization is achieved in a stronger norm, as
\begin{equation}\label{goal-stV}
 \norm{y(t)}{V}\le C_1\ex^{-\mu(t-s)}\norm{y(s)}{V},\quad\mbox{for all}\quad t\ge s\ge0,\quad\mbox{and all}\quad
 y_0\in V\quad\mbox{with}\quad\norm{y_0}{V}\le R.
\end{equation}
In  case that both strong (with~$y_0\in V$) and weak (with~$y_0\in H$) solutions exist,  it is  in general not clear in  whether one of~\eqref{goal-intro}
and~\eqref{goal-stV} implies the other.
\end{remark}

Our results apply  to internal actuators for parabolic equations.
Analogous explicit {\em semiglobally} stabilizing
feedbacks for  boundary actuators are still an open question, even for linear systems.
We refer, however,  to~\cite{Barbu_TAC13} for a different approach for {\em local} stabilization
of semilinear autonomous systems, with related numerical simulations
presented in~\cite{HalanayMureaSafta13}.
We mention also the backstepping approach~\cite{BalKrstic00,CochranVazquezKrstic06,KrsticMagnVazq09,TsubakinoKrsticHara12}
and the approach in~\cite{AzouaniTiti14,LunasinTiti17} which makes a direct use of the existence of
suitable determining parameters (e.g., Fourier modes, volumes).

For applications it can be of interest to not only achieve stabilization but to do
this while minimizing certain cost criteria. Since optimizing with respect to an operator entails difficulties related to computational complexity we propose to approximate the
feedback operators by neural networks.
Its parameters are trained in
learning steps involving
the cost functional to be minimized. The monotonicity requirement of the feedback law
is incorporated to the cost by means of a penalty term.
 A similar approach for optimal feedback stabilization without focus on the
 projection onto finitely many controllers was recently analyzed in
 \cite{KunischWalter_arx20}.

\subsection{Contents}
In Section~\ref{S:assumptions} the assumptions  required for the
operators~$A$, $A_{\rm rc}$, and~$\clN$ and  for the sequence of actuator
sets~$({U_{M}})_{M\in\bbN_0}$ are presented.  The stabilizing property
of the feedback operator~$\clK_M \circ P_{\clU_{M}}$, with~$\clK_M$ as
in~\eqref{FeedK-intro},  is proven in Section~\ref{S:stab}.
In Sections~\ref{S:learn}
the computation of an optimal feedback~$\clK_M^\bullet$ guaranteeing  that
~$\clK_M^\bullet \circ P_{\clU_{M}}$ is stabilizing and that  the corresponding
solution minimizes a suitable cost is presented.
The applicability of our abstract results to
{\em nonautonomous} parabolic equations is demonstrated  in Section~\ref{S:parabolic-OK}.
In  Section~\ref{S:simul} we present the results of numerical simulations and, finally,
the proofs of certain
technical results are gathered in the Appendix.\black

\subsection{Notation}
We write~$\bbR$ and~$\bbN$ for the sets of real numbers and nonnegative
integers, respectively, and we set $\bbR_r\coloneqq(r,\infty)$, $r\in\bbR$,
and $\bbN_0\coloneqq\mathbb N\setminus\{0\}$.

Given two Banach spaces~$X$ and~$Y$, if the inclusion
$X\subseteq Y$ is continuous, we write $X\xhookrightarrow{} Y$. We write
$X\xhookrightarrow{\rm d} Y$, and $X\xhookrightarrow{\rm c} Y$,
if the inclusion is also dense, respectively compact.

Let $X\subseteq Z$ and~$Y\subseteq Z$ be continuous inclusions, where~$Z$
is a Hausdorff topological space.
Then we can define the Banach spaces $X\times Y$, $X\cap Y$, and $X+Y$,
endowed with the  norms
$|(h,g)|_{X\times Y}:=\bigl(|h|_{X}^2+|g|_{Y}^2\bigr)^{\frac{1}{2}}$,
$|\hat h|_{X\cap Y}:=|(\hat h,\hat h)|_{X\times Y}$, and
$|\tilde h|_{X+Y}:=\inf\limits_{(h,g)\in X\times Y}\bigl\{|(h,g)|_{X\times Y}\mid \tilde h=h+g\bigr\}$,
respectively.
In case we know that $X\cap Y=\{0\}$, then $X+Y$ is a direct sum and we write $X\oplus Y$ instead.

The space of continuous linear mappings from~$X$ into~$Y$ is denoted by~$\clL(X,Y)$. In case~$X=Y$ we
write~$\clL(X)\coloneqq\clL(X,X)$.
The continuous dual of~$X$ is denoted~$X'\coloneqq\clL(X,\bbR)$.
The adjoint of an operator $L\in\clL(X,Y)$ will be denoted $L^*\in\clL(Y',X')$.
The kernel of the operator~$L$ is denoted by~$\Ker L=\Ker_X L\coloneqq\{v\in X\mid Lv=0\}$.

The space of continuous functions from~$X$ into~$Y$ is denoted by~$\clC(X,Y)$.
 The space of continuous real valued
 increasing functions, defined on~$\overline{\bbR_0}$ and vanishing at~$0$ is denoted by:
 \begin{equation*}
  \clC_{0,\iota}(\overline{\bbR_0},\bbR) \coloneqq \{\fki\in \clC(\overline{\bbR_0},\bbR)\mid
  \; \mathfrak i(0)=0,\quad\!\!\mbox{and}\quad\!\!
  \mathfrak i(\varkappa_2)\ge\mathfrak i(\varkappa_1)\;\mbox{ if }\; \varkappa_2\ge \varkappa_1\ge0\}.
 \end{equation*}
We also introduce the vector subspace~$\clC_{\rm b, \iota}(X, Y)\subset \clC(X,Y)$ by
 \begin{equation*}
  \clC_{\rm b, \iota}(X, Y)\coloneqq
  \left\{f\in \clC(X,Y) \mid \exists \, \mathfrak i\in \clC_{0,\iota}(\overline{\bbR_0},\bbR)\;\forall x\in X:\;
 \norm{f(x)}{Y}\le \mathfrak i (\norm{x}{X})
 \right\}.
 \end{equation*}

The orthogonal complement to a given subset~$B\subset H$ of a Hilbert space~$H$,
with scalar product~$(\Bigcdot,\Bigcdot)_H$,  is
denoted~$B^\perp\coloneqq\{h\in H\mid (h,s)_H=0\mbox{ for all }s\in B\}$.
We  write~$\overline B$ for the closure of~$B$ in~$H$.
When there is a need to precise the Hilbert space we will write~$B^{\perp,H}$ and~$\overline B^H$  instead.

Given two closed subspaces~$F\subseteq H$ and~$G\subseteq H$ of a Hilbert space
with~$H=F\oplus G$, we write by~$P_F^G\in\clL(H,F)$
the oblique projection in~$H$ onto~$F$ along~$G$. That is, writing $h\in H$ as $h=h_F+h_G$
with~$(h_F,h_G)\in F\times G$, we have~$P_F^Gh\coloneqq h_F$.
The orthogonal projection in~$H$ onto~$F$ is denoted by~$P_F\in\clL(H,F)$. Notice that~$P_F= P_F^{F^\perp}$.

Given a sequence~$(a_j)_{j\in\{1,2,\dots,n\}}$ of real nonnegative constants, $n\in\bbN_0$, $a_i\ge0$, we
denote~$\|a\|\coloneqq\max\limits_{1\le j\le n} a_j$.
By
$\overline C_{\left[a_1,\dots,a_n\right]}$ we denote a nonnegative function that
increases in each of its nonnegative arguments~$a_i$, $1\le i\le n$.
Finally, $C,\,C_i$, $i=0,\,1,\,\dots$, stand for unessential positive constants.

\section{Assumptions}\label{S:assumptions}
The results rely on  assumptions on the
operators~$A$, $A_{\rm rc}$, and~$\clN$  appearing in the system dynamics, and
on the set~$U_M$ of actuators. We summarize them here.
\begin{assumption}\label{A:A0sp}
 $A\in\clL(V,V')$ is symmetric, and such that $(y,z)\mapsto\langle Ay,z\rangle_{V',V}$ is a
 complete scalar product on~$V.$
\end{assumption}

Hereafter, we suppose that~$V$ is endowed with the scalar product~$(y,z)_V\coloneqq\langle Ay,z\rangle_{V',V}$,
which again makes~$V$ a Hilbert space.
Necessarily, $A\colon V\to V'$ is an isometry.
\begin{assumption}\label{A:A0cdc}
The inclusion $V\subseteq H$ is dense, continuous, and compact.
\end{assumption}

Necessarily, we have that
\[
 \langle y,z\rangle_{V',V}=(y,z)_{H},\quad\mbox{for all }(y,z)\in H\times V,
\]
and also that the operator $A$ is densely defined in~$H$, with domain $\rmD(A)$ satisfying
\[
\rmD(A)\xhookrightarrow{\rm d,\,c} V\xhookrightarrow{\rm d,\,c} H\xhookrightarrow{\rm d,\,c} V'\xhookrightarrow{\rm d,\,c}\rmD(A)'.
\]
Further,~$A$ has a compact inverse~$A^{-1}\colon H\to \rmD(A)$, and we can find a nondecreasing
system of (repeated accordingly to their multiplicity) eigenvalues $(\alpha_n)_{n\in\bbN_0}$ and a corresponding complete basis of
eigenfunctions $(e_n)_{n\in\bbN_0}$:
\begin{equation}\label{eigfeigv}
0<\alpha_1\le\alpha_2\le\dots\le\alpha_n\to+\infty, \quad Ae_n=\alpha_n e_n.
\end{equation}

 We can define, for every $\zeta\in\bbR$, the fractional powers~$A^\zeta$, of $A$, by
 \[
  y=\sum_{n=1}^{+\infty}y_ne_n,\quad A^\zeta y=A^\zeta \sum_{n=1}^{+\infty}y_ne_n\coloneqq\sum_{n=1}^{+\infty}\alpha_n^\zeta y_n e_n,
 \]
 and the corresponding domains~$\rmD(A^{|\zeta|})\coloneqq\{y\in H\mid A^{|\zeta|} y\in H\}$, and
 $\rmD(A^{-|\zeta|})\coloneqq \rmD(A^{|\zeta|})'$.
 We have that~$\rmD(A^{\zeta})\xhookrightarrow{\rm d,\,c}\rmD(A^{\zeta_1})$, for all $\zeta>\zeta_1$,
 and we can see that~$\rmD(A^{0})=H$, $\rmD(A^{1})=\rmD(A)$, $\rmD(A^{\frac{1}{2}})=V$.

\begin{assumption}\label{A:A1}
For almost every~$t>0$ we have~$A_{\rm rc}(t)\in\clL(H, V')+\clL(V, H)$,
and we have a uniform bound, that is, $\norm{A_{\rm rc}}{L^\infty(\bbR_0,\clL(H,V')+\clL(V, H))}\eqqcolon C_{\rm rc}<+\infty.$
\end{assumption}

\begin{assumption}\label{A:NN}
 For almost every~$t>0$, we have~$\clN(t,\Bigcdot)\in\clC_{\rm b,\iota}(V,V')$
and
there exist constants $C_\clN\ge 0$, $n\in\bbN_0$,
$\zeta_{1j}\ge0$, $\zeta_{2j}\ge0$,
 $\delta_{1j}\ge 0$, ~$\delta_{2j}\ge 0$,
 with~$j\in\{1,2,\dots,n\}$, such that
 for  all~$t>0$ and
 all~$(y_1,y_2)\in V\times V$, we have
\begin{align*}
&\norm{\clN(t,y_1)-\clN(t,y_2)}{V'}\le C_\clN\textstyle\sum\limits_{j=1}^{n}
  \left( \norm{y_1}{H}^{\zeta_{1j}}\norm{y_1}{V}^{\zeta_{2j}}+\norm{y_2}{H}^{\zeta_{1j}}\norm{y_2}{V}^{\zeta_{2j}}\right)
   \norm{y_1-y_2}{H}^{\delta_{1j}}\norm{y_1-y_2}{V}^{\delta_{2j}},
\end{align*}
with~$\zeta_{2j}+\delta_{2j}<1$ and~$\delta_{1j}+\delta_{2j}\ge1$.
Further,
$\clN(\Bigcdot,y)\in L^1_{\rm loc}(\bbR_0,V')$, for every $y\in V$.
\end{assumption}

\begin{assumption}\label{A:poincare}
The set~${U_{M}}\coloneqq\{\Phi_j\mid 1\le j\le M_\sigma\}\subset V$ of actuators satisfy the following.
We have~$M_\sigma\coloneqq\sigma(M)$, where~$\sigma\colon\bbN_0\to\bbN_0$
is a strictly increasing function~$\sigma\colon\bbN_0\to\bbN_0$.
Further, with~$ \clU_{M}\coloneqq\linspan{U_{M}}$ and defining for each~$M\in\bbN_0$,
the  Poincar\'e-like constant
\begin{align}\label{Poinc_const}
\xi_{M_+}\coloneqq\inf_{\varTheta\in (V\bigcap\clU_{M}^\perp)\setminus\{0\}}
\tfrac{\norm{\varTheta}{V}^2}{\norm{\varTheta}{H}^2},
\end{align}
we have that
$
\lim\limits_{M\to+\infty}\xi_{M_+}=+\infty.
$
\end{assumption}

\begin{assumption}\label{A:monotone}
The operator~$(t,y)\mapsto \widehat\clN(t,y)\coloneqq  \clK_M(t,P_{\clU_M}y)\in \clU_M$ satisfies Assumption~\ref{A:NN} and,
 for a constant~$\overline\lambda>0$, the nonlinear  operator~$(t,p)\mapsto\clK_M(t,p)$, $p\in\clU_M$, satisfies the monotonicity condition
~$\clK_M(t,\Bigcdot)\preceq-\overline\lambda\Id$, for all $t\ge0$:
\end{assumption}

\begin{remark}
We shall show that Assumptions~\ref{A:A0sp}--\ref{A:monotone} are satisfiable for parabolic
systems as in Section~\ref{sS:illust-parabolic}, evolving
in rectangular domains~$\Omega\subset\bbR^d$.  An analogous argument can be used to show
 the assumptions are satisfiable for
parabolic systems evolving in general convex polygonal spatial domains; see the discussion
in~\cite[Sect.~8.2]{Rod21-sicon}.
\end{remark}

\section{The main stabilizability result}\label{S:stab}
Our main stabilizability result is as follows, which is a more precise statement of Main Result,
in the Introduction.
\begin{theorem}\label{T:main}
 Let Assumptions~\ref{A:A0sp}--\ref{A:monotone} hold true, and let ~$R>0$ and~$\mu>0$ be given. Then if~$M\in\bbN_0$ and~$\overline\lambda>0$  are large enough,
the solution of system~\eqref{sys-y-intro-K}
 satisfies
 \[
  \norm{y(t)}{H}\le\ex^{-\mu (t-s)}\norm{y(s)}{H},\quad\mbox{for all}\quad t\ge s\ge0\quad\mbox{and all}\quad y_0\in H\quad\mbox{with}\quad \norm{y_0}{H}<R.
 \]
In particular, $t\mapsto \norm{y(t)}{H}$ is strictly decreasing at time~$t=s$, if~$\norm{y(s)}{H}\ne0$.
Furthermore,~$M$ and~$\overline\lambda$ can be chosen as~$M=\ovlineC{R,\mu,C_{\rm rc},C_\clN}$
and~$\overline\lambda=\ovlineC{R,\mu,C_{\rm rc},,C_\clN}$.
\end{theorem}

This section is mainly dedicated to the proof of this result.

\subsection{Auxiliary results}
We will use the following estimate for the reaction-convection and nonlinear operators.

\begin{lemma}\label{L:Arc-estim}
With~$A_{\rm rc}\in L^\infty(\bbR_0,\clL(H,V')+\clL(V, H))$  and~$C_{\rm rc}$ as in Assumption~\ref{A:A1}, for every~$(h,g)\in V\times V$, we have the estimate
\begin{align}\label{est.Arc}
 2\norm{\langle A_{\rm rc}(t)h,g\rangle_{V',V}}{\bbR}\le \gamma(\norm{g}{V}^2+\norm{h}{V}^2)
 +\gamma^{-1}C_{\rm rc}^2(\norm{h}{H}^2+\norm{g}{H}^2).
\end{align}
\end{lemma}
The proof of Lemma~\ref{L:Arc-estim} can be found in~\cite[Sect.~3.1]{Rod21-sicon}.

\begin{lemma}\label{L:NN1}
 If Assumptions~\ref{A:A0sp}, \ref{A:A0cdc}, and~\ref{A:NN} hold true, then
there is a constant $\overline C_{\clN1}>0$
such that:
  for all~$\widehat\gamma_0>0$, all~$t>0$,
  all~$(y_1,y_2)\in V\times V$,  we have
\begin{align}
 &2\bigl\langle \clN(t,y_1)-\clN(t,y_2),y_1-y_2\bigr\rangle_{V',V}
 \le \widehat\gamma_0 \norm{y_1- y_2}{V}^{2}\label{NNyAy}\\
  &\hspace*{3em}
  +\left(1+\widehat\gamma_0^{-\frac{1+\|\delta_2\|}{1-\|\delta_2\|} }\right)
 \overline C_{\clN1}\sum\limits_{j=1}^n \left( \norm{y_1}{H}^\frac{2\zeta_{1j}}{1-\delta_{2j}} \norm{y_1}{V}^\frac{2\zeta_{2j}}{1-\delta_{2j}}
 +\norm{y_2}{H}^\frac{2\zeta_{1j}}{1-\delta_{2j}} \norm{y_2}{V}^\frac{2\zeta_{2j}}{1-\delta_{2j}}
 \right)\norm{y_1-y_2}{H}^{\frac{2\delta_{1j}}{1-\delta_{2j}}},\notag
\end{align}
 with~$\|\delta_2\|\coloneqq\max\limits_{1\le j\le n}\delta_{2j}$. Further, the constant~$\overline C_{\clN1}$ is of the form
 $ \overline C_{\clN1}=\ovlineC{n,\frac{1}{1-\|\delta_{2}\|},C_\clN}$.
 \end{lemma}

\begin{proof}
By Assumption~\ref{A:NN}, with
\[
\overline\clN\coloneqq A^{-\frac12}\clN A^\frac12,\quad (z_1,z_2)\coloneqq(A^{-\frac12}y_1,A^{-\frac12}y_2),\quad\mbox{and}\quad
\overline d\coloneqq A^{-\frac12}d=A^{-\frac12}(y_1-y_2),
\]
it follows that
\begin{align*}
&\norm{\overline\clN(t,z_1)-\overline\clN(t,z_2)}{H}\le C_\clN\textstyle\sum\limits_{j=1}^{n}
  \left( \norm{z_1}{V}^{\zeta_{1j}}\norm{z_1}{\rmD(A)}^{\zeta_{2j}}+\norm{z_2}{V}^{\zeta_{1j}}\norm{z_2}{\rmD(A)}^{\zeta_{2j}}\right)
   \norm{\overline d}{V}^{\delta_{1j}}\norm{\overline d}{\rmD(A)}^{\delta_{2j}}.
\end{align*}
Then, by~\cite[Prop.~3.5]{Rod20} we have that
\begin{align*}
 &2\bigl\langle \clN(t,y_1)-\clN(t,y_2),y_1-y_2\bigr\rangle_{V',V}
=2\bigl( \overline\clN(t,z_1)-\overline\clN(t,z_2),A(z_1-z_2)\bigr)_H
 \le \widehat\gamma_0 \norm{z_1- z_2}{\rmD(A)}^{2}\\
  &\hspace*{3em}
  +\left(1+\widehat\gamma_0^{-\frac{1+\|\delta_2\|}{1-\|\delta_2\|} }\right)
 \overline C_{\clN1}\sum\limits_{j=1}^n \left( \norm{z_1}{V}^\frac{2\zeta_{1j}}{1-\delta_{2j}} \norm{z_1}{\rmD(A)}^\frac{2\zeta_{2j}}{1-\delta_{2j}}
 +\norm{z_2}{V}^\frac{2\zeta_{1j}}{1-\delta_{2j}} \norm{y_2}{\rmD(A)}^\frac{2\zeta_{2j}}{1-\delta_{2j}}
 \right)\norm{z_1-z_2}{V}^{\frac{2\delta_{1j}}{1-\delta_{2j}}},
\end{align*}
which gives us the desired estimate~\eqref{NNyAy}.
\end{proof}

\begin{remark}
Lemma~\ref{L:NN1} is appropriate to deal with weak solutions investigated in this manuscript. It is the analogous of~\cite[Prop.~3.5]{Rod20}, which have been used to deal with strong solutions. Actually, the result presented in~\cite[Prop.~3.5]{Rod20} is an estimate for $2(\clP(\overline\clN(t,z_1)-\overline\clN(t,z_2)),A(z_1-z_2)_{H}$, where~$\clP$ is a suitable projection. However, the steps of the proof can be repeated for an arbitrary given linear operator~$\clP\in\clL(H)$, and in particular for the identity operator~$\clP=\Id$.
\end{remark}

\begin{lemma}\label{L:NN2}
 If Assumptions~\ref{A:A0sp}, \ref{A:A0cdc}, and~\ref{A:NN} hold true, then
there exits a constant~$p>0$
such that:
  for all~$\widehat\gamma_0>0$, all~$t>0$,
  all~$y\in V$,  we have
\begin{align}\label{NNyAy2}
2\bigl\langle \clN(y),y\bigr\rangle_{V',V}
&\le \widehat\gamma_0 \norm{y}{V}^{2}
+\overline C_{\clN2} (1+\norm{y}{H}^{p})\norm{y}{H}^2, \quad\mbox{for all}\quad
y\in V,
\end{align}
with~$\overline C_{\clN2}=\ovlineC{\widehat\gamma_0^{-1},C_\clN}$.
 \end{lemma}

The proof of Lemma~\ref{L:NN2} follows by direct computations using arguments from~\cite[Proof of Prop.~3.5]{Rod20}.
Since the computations are long, we give them in the Appendix, Section~\ref{Apx:proofL:NN2}.

Below, we will assume that the span~$\clU_M$ of our actuators and the span of auxiliary
functions~$\widetilde\clU_M$
satisfy the relations
\begin{equation}\label{eqA:dir+}
 \clU_M\subset H,\quad\widetilde\clU_M\subset V,\quad\dim\widetilde\clU_M=\dim\clU_M=M_\sigma,
 \quad\mbox{and}\quad H=\widetilde\clU_M\oplus\clU_M^\perp.
\end{equation}

\begin{lemma}\label{L:MlamPoinc}
 Let Assumptions~\ref{A:A0sp}, \ref{A:A0cdc}, and~\ref{A:poincare} hold true. Let the
 pair~$(\clU_M,\widetilde\clU_M)$ satisfy~\eqref{eqA:dir+}
and let~$(\Xi_M)_{M\in\bbN_0}$ be a sequence of positive real numbers. Then for every constant~$\zeta>0$ we can find~$M$ and~$\overline\lambda$ large enough such that
\[
\norm{y}{V}^2+2\overline\lambda \Xi_M\norm{P_{\widetilde\clU_{M}}^{\clU_{M}^\perp}y}{H}^2
\ge\zeta \norm{y}{H}^2, \quad\mbox{for all}\quad
y\in V.
\]
Furthermore~$M=\ovlineC{\zeta}$ and~$\overline\lambda=\ovlineC{\zeta,\xi_{M},\Xi_M^{-1}}$, where
$
\xi_{M}^2\coloneqq\sup\limits_{\vartheta\,\in\, \widetilde\clU_{M}\setminus\{0\}}
\tfrac{\norm{\vartheta}{V}^2}{\norm{\vartheta}{H}^2}.
$
 \end{lemma}
\begin{proof}
Since~$(\clU_M,\widetilde\clU_M)$ satisfies~\eqref{eqA:dir+}, we can write
\[
y=\varTheta+\vartheta,\quad\mbox{with}\quad
\varTheta\coloneqq P_{\clU_{M}^\perp}^{\widetilde\clU_{M}}y\quad\mbox{and}\quad
\vartheta\coloneqq P_{\widetilde\clU_{M}}^{\clU_{M}^\perp}y.
\]
We find that
\begin{align*}
\norm{y}{V}^2+2\overline\lambda\Xi_M \norm{P_{\widetilde\clU_{M}}^{\clU_{M}^\perp}y}{H}^2
&=\norm{\varTheta+\vartheta}{V}^2+2\overline\lambda\Xi_M \norm{\vartheta}{H}^2
=\norm{\varTheta}{V}^2+2(\varTheta,\vartheta)_V+\norm{\vartheta}{V}^2+2\overline\lambda\Xi_M \norm{\vartheta}{H}^2\\
&\ge\tfrac{1}{2}\norm{\varTheta}{V}^2-\norm{\vartheta}{V}^2+2\overline\lambda \Xi_M\norm{\vartheta}{H}^2
\ge\tfrac{1}{2}\norm{\varTheta}{V}^2+(2\overline\lambda\Xi_M\xi_{M}^{-2}-1)\norm{\vartheta}{V}^2
\end{align*}
and by choosing~$\overline\lambda>\frac{\xi_{M}^{2}}{2\Xi_M}$ it follows that
\begin{align*}
\norm{y}{V}^2+2\overline\lambda\Xi_M \norm{P_{\widetilde\clU_{M}}^{\clU_{M}^\perp}y}{H}^2
&\ge\tfrac{1}{2}\xi_{M_+}\norm{\varTheta}{H}^2
+\norm{\Id}{\clL(V,H)}^{-2}(2\overline\lambda\Xi_M\xi_{M}^{-2}-1)\norm{\vartheta}{H}^2.
\end{align*}
Hence, for given~$\zeta>0$, by choosing
\[
\xi_{M_+}\ge 4\zeta\quad\mbox{and}\quad\overline\lambda\ge(2\zeta\norm{\Id}{\clL(V,H)}^{2}+1)\tfrac{\xi_{M}^{2}}{2\Xi_M}>\tfrac{\xi_{M}^{2}}{2\Xi_M},
\]
we arrive at
\begin{align*}
\norm{y}{V}^2+2\overline\lambda\Xi_M \norm{P_{\widetilde\clU_{M}}^{\clU_{M}^\perp}y}{H}^2
&\ge2\zeta\left(\norm{\varTheta}{H}^2+ \norm{\vartheta}{H}^2\right)
\ge \zeta\norm{\varTheta+\vartheta}{H}^2=\zeta\norm{y}{H}^2,
\end{align*}
which ends the proof.
\end{proof}

\begin{lemma}\label{L:monot-obli}
Let Assumption~\ref{A:monotone} hold true for~$\overline\lambda>0$,  and let the
 pair~$(\clU_M,\widetilde\clU_M)$ satisfy~\eqref{eqA:dir+}.
Then, we have that
$
(\clK_M(t,p),p)_H\le-\overline\lambda\norm{P_{\widetilde \clU_M}^{\clU_M^\perp}}{\clL(H)}^{-2}\norm{P_{\widetilde \clU_M}^{\clU_M^\perp}p}{H}^2$, for all $p\in\clU_M$.
\end{lemma}
\begin{proof}
Let~$p\in\clU_M$ be arbitrary. Using the relations
\[
P_{\widetilde \clU_M}^{\clU_M^\perp}=P_{\widetilde \clU_M}^{\clU_M^\perp}P_{\clU_M}\qquad\mbox{and}\qquad P_{\clU_M}=P_{\clU_M}P_{\widetilde \clU_M}^{\clU_M^\perp}
\]
we find
\begin{align*}
\norm{P_{\widetilde \clU_M}^{\clU_M^\perp}p}{H}^2&=\norm{P_{\widetilde \clU_M}^{\clU_M^\perp}P_{\clU_M}P_{\widetilde \clU_M}^{\clU_M^\perp}p}{H}^2
\le\norm{P_{\widetilde \clU_M}^{\clU_M^\perp}}{\clL(H)}^2 \norm{P_{\clU_M}P_{\widetilde \clU_M}^{\clU_M^\perp}p}{H}^2=\norm{P_{\widetilde \clU_M}^{\clU_M^\perp}}{\clL(H)}^2 \norm{p}{H}^2
\end{align*}
and, since by the hypothesis we have that $(\clK_M(t,p),p)_H\le-\overline\lambda\norm{p}{H}^2$, we arrive at the desired inequality
$
(\clK_M(t,p),p)_H\le-\overline\lambda\norm{P_{\widetilde \clU_M}^{\clU_M^\perp}}{\clL(H)}^{-2}\norm{P_{\widetilde \clU_M}^{\clU_M^\perp}p}{H}^2
$.
\end{proof}

\subsection{Proof of Theorem~\ref{T:main}}\label{sS:proofT:main}
Let us fix an auxiliary subspace~$\widetilde U_M$ as in  Lemma~\ref{L:monot-obli}. Multiplying the dynamics   in~\eqref{sys-y-intro-K} by~$y$,
 we find\black
\begin{align}
\tfrac12\tfrac{\ed}{\ed t}\norm{y}{H}^2 &=-\langle Ay+A_{\rm rc}y+\clN(y),y\rangle_{V',V}
+\bigl(\clK_M( P_{\clU_{M}}y),P_{\clU_{M}}y\bigr)_H\notag\\
&\le-\norm{y}{V}^2-\langle A_{\rm rc}y+\clN(y),y\rangle_{V',V}
-\overline\lambda \norm{P_{\widetilde \clU_M}^{\clU_M^\perp}}{\clL(H)}^{-2} \norm{P_{\widetilde \clU_M}^{\clU_M^\perp} y}{H}^2.\label{main-1test}
\end{align}

Now, using Lemmas~\ref{L:Arc-estim} and~\ref{L:NN2}, we obtain
\begin{align*}
 2\langle A_{\rm rc}y,y\rangle_{V',V}&\le2\gamma_1\norm{y}{V}^2 +2\gamma_1^{-1}C_{\rm rc}^2\norm{y}{H}^2,\quad\mbox{for all}\quad \gamma_1>0,\\
 2\langle \clN(y),Ay\rangle_{V',V}&\le\gamma_2\norm{y}{V}^2 +\ovlineC{\gamma_2^{-1},C_\clN}(1+\norm{y}{H}^p)\norm{y}{H}^2,\quad\mbox{for all}\quad \gamma_2>0,
\end{align*}
which lead us to
\begin{subequations}\label{main-2test}
\begin{align}
\tfrac{\ed}{\ed t}\norm{y}{H}^2
&\le-(2-2\gamma_1-\gamma_2)\norm{y}{V}^2
+\left(2\gamma_1^{-1}C_{\rm rc}^2+\ovlineC{\gamma_2^{-1},C_\clN}(1+\norm{y}{H}^p)\right)\norm{y}{H}^2
-2\overline\lambda \norm{P_{\widetilde \clU_M}^{\clU_M^\perp}}{\clL(H)}^{-2} \norm{P_{\widetilde\clU_{M}}^{\clU_{M}^\perp}y}{H}^2\notag\\
&\le-(2-2\gamma_1-\gamma_2)\norm{y}{V}^2
+\left(\clD_1+\clD_2\norm{y}{H}^p\right)\norm{y}{H}^2
-2\overline\lambda \norm{P_{\widetilde \clU_M}^{\clU_M^\perp}}{\clL(H)}^{-2} \norm{ P_{\widetilde\clU_{M}}^{\clU_{M}^\perp}y}{H}^2,
\intertext{with}
\clD_1&\coloneqq2\gamma_1^{-1}C_{\rm rc}^2+\ovlineC{\gamma_2^{-1},C_\clN}
\quad\mbox{and}\quad \clD_2\coloneqq\ovlineC{\gamma_2^{-1},C_\clN}.
\end{align}
\end{subequations}
Now, we set~$(\gamma_1,\gamma_2)=(\frac14,\frac12)$, and obtain
\begin{equation}\label{dtny.DAV}
\tfrac{\ed}{\ed t}\norm{y}{H}^2
\le-\norm{y}{V}^2 -2\overline\lambda\norm{P_{\widetilde \clU_M}^{\clU_M^\perp}}{\clL(H)}^{-2} \norm{ P_{\widetilde\clU_{M}}^{\clU_{M}^\perp}y}{H}^2
+\left(\clD_1+\clD_2\norm{y}{H}^p\right)\norm{y}{H}^2.
\end{equation}

For arbitrary given~$\mu>0$ and~$R>0$  we set
\begin{equation}\label{clD0}
\clD_0=2\mu +\left(\clD_1+\clD_2R^p\right)
\end{equation}
and we use Lemma~\ref{L:MlamPoinc}, with~$\zeta=\clD_0$ and~$\Xi_M=\norm{P_{\widetilde \clU_M}^{\clU_M^\perp}}{\clL(H)}^{-2}$, to conclude that for~$M$ and~$\overline\lambda$ large enough we have
\begin{equation}\label{dtny.V}
\tfrac{\ed}{\ed t}\norm{y}{H}^2
\le-\left(\clD_0-
\clD_1-\clD_2\norm{y}{H}^p\right)\norm{y}{H}^2=-\left( 2\mu+\clD_2R^p\black-\clD_2\norm{y}{H}^p\right)\norm{y}{H}^2
\end{equation}
and, by using~\cite[Prop.~4.3]{Rod20}, it follows that
\begin{equation}\label{ny.V}
\norm{y(t)}{H}^2\le\ex^{-2\mu(t-s)}\norm{y(s)}{H}^2,\quad\mbox{for all}\quad t\ge s\ge0, \quad\mbox{provided}\quad \norm{y(0)}{H}\le R.
\end{equation}

Finally, by~\eqref{ny.V} it follows that the  norm~$\norm{y(t)}{H}^2$ strictly decreases at time~$t=s$ if~$\norm{y(s)}{H}^2\ne 0$. See,
for example,~\cite[Lem.~3.3]{Rod_pp20-CL}.
\qed

\begin{remark}\label{R:depMlam1}
Observe that after fixing~$(\gamma_1,\gamma_2)=(\frac14,\frac12)$ in Lemma~\ref{L:MlamPoinc}, it is sufficient to take
$M=\ovlineC{\zeta}$ and~$\overline\lambda=\ovlineC{\zeta,\xi_M^2\Xi_M^{-1}}$ large enough.
So, it is enough to take large enough $M=\ovlineC{R,\mu,C_{\rm rc},C_\clN}$
and large enough~$\overline\lambda =\ovlineC{R,\mu,C_{\rm rc},C_\clN,\xi_M^2\Xi_M^{-1}}$.
Defining the constants~$\varsigma_M=\max\limits_{1\le m\le M}\xi_m^2\Xi_m^{-1}$ and observing that~$\varsigma_M=\ovlineC{M}=\ovlineC{\zeta}$ we can conclude that it is enough to take
$\overline\lambda =\ovlineC{R,\mu,C_{\rm rc},C_\clN}$.
\end{remark}

\subsection{On the existence and uniqueness of the solutions for the controlled system}\label{sS:exist-solut}
Note that, in Section~\ref{sS:proofT:main}, we have proven the stability of
system~\eqref{sys-y-intro-K} for large enough~$M$ and~$\overline\lambda$, where we have implicitly
assumed that weak solutions do exist. The existence and uniqueness of weak solutions can be proven
by following similar arguments as in~\cite[Sect.~4.3]{Rod20}. Indeed,
inequalities~\eqref{dtny.DAV}, \eqref{dtny.V}, and~\eqref{ny.V} will also hold for
Galerkin approximations of~\eqref{sys-y-intro-K} given by
\begin{align}\label{sys-y-intro-K-GalN}
 \dot y^N +Ay^N+\widehat P_{\clE_N^{\rm f}}A_{\rm rc}(t)y^N +\widehat P_{\clE_N^{\rm f}}\clN(t,y^N) =\widehat P_{\clE_N^{\rm f}}\clK_{M}\left(t,P_{\clU_{M}} y^N\right),\qquad y^N(0)= \widehat P_{\clE_N^{\rm f}}y_0,
\end{align}
where~$\widehat P_{\clE_N^{\rm f}}\in\clL(V')$ stands  for the orthogonal projection in~$V'$ onto the subspace~$\clE_N^{\rm f}=\linspan\{e_i\mid 1\le i\le N\}$ spanned by the first eigenfuntions of~$A$. It is not difficult to observe (e.g., from the Fourier expansion of an element in~$V'$) that~$\widehat P_{\clE_N^{\rm f}}\in\clL(V')$ is an extension of  the orthogonal projection~$P_{\clE_N^{\rm f}}\in\clL(H)$ in~$H$ onto the subspace~$\clE_N^{\rm f}$.

Therefore, for arbitrary given~$T>0$ and~$R>0$, by the analogue to~\eqref{ny.V} we will have that for large enough~$M$ and~$\lambda$, if
$\norm{y_0}{H}\le R$, then
\begin{equation}\label{dtny.V-Gal}
\norm{y^N}{L^\infty((0,T),H)}^2\le \norm{y^N(0)}{H}^2\le  \norm{y(0)}{H}^2.
\end{equation}
Then, by the analogous of~\eqref{dtny.DAV}, we will also find that
\begin{equation}\label{dtny.DAV-Gal}
\norm{y}{L^2((0,T),V)}^2\le C_2\norm{y(0)}{H}^2,
\end{equation}
with~$C_2$ independent of~$N$. Next, we can use~\eqref{sys-y-intro-K-GalN}
together with Assumptions~\ref{A:A1}  and~\ref{A:NN} to obtain
\begin{equation}\label{dtny.H-Gal}
\norm{\dot y}{L^2((0,T),V')}^2\le C_3\norm{y(0)}{H}^2.
\end{equation}
Notice, in particular, that
\[
\norm{\widehat P_{\clE_N^{\rm f}}\black\clN(t,y)}{V'}^2\le\norm{\clN(t,y)}{V'}^2.
\]
 Hence\black from Assumption~\ref{A:NN} (with~$(y,0)$ in the role of~$(y_1,y_2)$) it follows that
\begin{align*}
\norm{\widehat P_{\clE_N^{\rm f}}\black\clN(t,y)}{V'}^2&\le C_\clN^2 \left(\textstyle\sum\limits_{j=1}^{n}
   \norm{y}{H}^{\zeta_{1j}+\delta_{1j}}\norm{y}{V}^{\zeta_{2j}+\delta_{2j}}\right)^2
\le n C_\clN^2\textstyle\sum\limits_{j=1}^{n}
   \norm{y}{H}^{2(\zeta_{1j}+\delta_{1j})}(1+\norm{y}{V}^{2}).
\end{align*}

Next, we can show that a weak limit of a suitable subsequence of such Galerkin approximations is a
weak solution for system~\eqref{sys-y-intro-K} defined on the time interval~$I_T\coloneqq(0,T)$. Indeed, from~\eqref{dtny.DAV-Gal} and~\eqref{dtny.H-Gal} we can conclude that,
there exists a subsequence~$y^{N_s}$ of~$y^{N}$,   converging weakly as
\begin{align*}
y^{N_s}\xrightharpoonup[L^2(I_T,V)]{} y^{\infty}\quad\mbox{and}\quad \dot y^{N_s}\xrightharpoonup[L^2(I_T,V')]{} \dot y^{\infty}
\end{align*}
for some~$y^\infty\in W(I_T,V,V')\coloneqq \{z\in L^2(I_T,V)\mid \dot z\in L^2(I_T,V')\}$. This implies that the linear terms, for such subsequence, satisfy
\begin{align*}
Ay^{N_s}\xrightharpoonup[L^2(I_T,V')]{} Ay^{\infty}\quad\mbox{and}\quad  A_{\rm rc}y^{N_s}\xrightharpoonup[L^2(I_T,V')]{} A_{\rm rc} y^{\infty}.
\end{align*}

Now, recalling that we have the compact inclusion~$W(I_T,V,V')\xhookrightarrow{\rm c}L^2(I_T,H)$, we can also assume the strong limit
\begin{align*}
y^{N_s}\xrightarrow[L^2(I_T,H)]{} y^{\infty}.
\end{align*}

For the nonlinear term~$\clN(t,y^{N_s})$,
using Assumption~\ref{A:NN}, we find
\begin{align*}
&\norm{\clN(t,y^{N_s})-\clN(t,y^\infty)}{V'}\le C_\clN\textstyle\sum\limits_{j=1}^{n}
  \left( \norm{y^{N_s}}{H}^{\zeta_{1j}}\norm{y^{N_s}}{V}^{\zeta_{2j}}+\norm{y^\infty}{H}^{\zeta_{1j}}\norm{y^\infty}{V}^{\zeta_{2j}}\right)
   \norm{d^{N_s}}{H}^{\delta_{1j}}\norm{d^{N_s}}{V}^{\delta_{2j}},
\end{align*}
with~$d^{N_s}=y^{N_s}-y^\infty$.
Recalling that~$\delta_{2j}+\zeta_{2j}<1$, by the H\"older inequality we can write
\begin{align*}
 &\norm{\clN(t,y^{N_s})-\clN(t,y^\infty)}{L^2(I_T,V')} \notag\\
&\hspace*{2em}\le C_{\clN}\textstyle\sum\limits_{j=1}^n\norm{\left( \sum\limits_{w\in\{y^{N_s},y^\infty\}}\norm{w}{H}^{\zeta_{1j}}\norm{w}{V}^{\zeta_{2j}}
 \right)
 \norm{d^{N_s}}{V}^{\delta_{2j}}}{L^\frac{2}{\zeta_{2j}+\delta_{2j}}\!(I_T,\bbR)}
 \norm{\norm{d^{N_s}}{H}^{\delta_{1j}}}{L^\frac{2}{1-\zeta_{2j}-\delta_{2j}}\!(I_T,\bbR)} \\
 &\hspace*{2em}\le C_8\textstyle\sum\limits_{j=1}^n\textstyle\sum\limits_{w\in\{y^{N_s},y^\infty\}}
\norm{\norm{w}{V}^{\zeta_{2j}}
  \norm{d^{N_s}}{V}^{\delta_{2j}}}{L^\frac{2}{\zeta_{2j}+\delta_{2j}}(I_T,\bbR)}
 \norm{d^{N_s}}{L^{2}(I_T,H)}^{1-\zeta_{2j}-\delta_{2j}}.
\end{align*}
In the last inequality we have used the fact that, from $\delta_{1j}+\delta_{2j}\ge1$,
it follows~$\frac{2\delta_{1j}}{1-\zeta_{2j}-\delta_{2j}}\ge 2$,
and thus~$\norm{\norm{d^{N_s}}{H}^{\delta_{1j}}}{L^\frac{2}{1-\zeta_{2j}-\delta_{2j}}(I_T,\bbR)}
\le C_9\norm{d^{N_s}}{L^{2}(I_T,H)}^{1-\zeta_{2j}-\delta_{2j}}$,
because~$d^{N_s}$ is uniformly bounded in~$L^{\infty}(I_T,H)$.
Observe also that by the Young inequality
\[
 \norm{w}{V}^{\frac{2\zeta_{2j}}{\zeta_{2j}+\delta_{2j}}}
  \norm{d^{N_s}}{V}^{\frac{2\delta_{2j}}{\zeta_{2j}+\delta_{2j}}}
  \le
  \norm{w}{V}^{2}
  +\norm{d^{N_s}}{V}^{2},
\]
which leads us to
\begin{align*}
\norm{\norm{w}{V}^{\zeta_{2j}}
  \norm{d^{N_s}}{V}^{\delta_{2j}}}{L^\frac{2}{\zeta_{2j}+\delta_{2j}}(I_T,\bbR)}
  &\le \norm{\norm{w}{V}^{\frac{2\zeta_{2j}}{\zeta_{2j}+\delta_{2j}}}
  \norm{d^{N_s}}{V}^{\frac{2\delta_{2j}}{\zeta_{2j}+\delta_{2j}}}}{L^1(I_T,\bbR)}^\frac{\zeta_{2j}+\delta_{2j}}{2}\\
  &\le\left(\norm{w}{L^2(I_T,V)}^{2}+\norm{d^{N_s}}{L^2(I_T,V)}^2\right)^{\frac{\zeta_{2j}+\delta_{2j}}{2}}
  \end{align*}
and consequently, since~$1-\zeta_{2j}-\delta_{2j}>0$, we arrive at
\begin{align*}
 &\norm{\clN(t,y^{N_s})-\clN(t,y^\infty)}{L^2(I_T,V')} \le C_{10}\norm{d^{N_s}}{L^{2}(I_T,H)}^{1-\zeta_{2j}-\delta_{2j}}
 \xrightarrow[N\to+\infty]{} 0.
\end{align*}

Analogously, since by Assumption~\ref{A:monotone} the nonlinear input feedback operator  is in the class of
nonlinearities as~$\clN$ in Assumption~\ref{A:NN}, we can repeat the arguments above to conclude that
\begin{align*}
 &\norm{\clK_M(t, P_{\clU_{M}}y^{N_s})-\clK_M(t,P_{\clU_{M}}y^\infty)}{L^2(I_T,V')} \le C_{11}\norm{P_{\clU_{M}}d^{N_s}}{L^{2}(I_T,H)}^{1-\zeta_{2j}-\delta_{2j}}
 \xrightarrow[N\to+\infty]{} 0.
\end{align*}

The above limits allow us to conclude that~$y^\infty$ is a weak solution for our system.

The uniqueness of the such weak solution can be concluded by using
standard arguments for the linear terms, and by using the estimate~\eqref{NNyAy} for the nonlinear terms~$\clN(t,y)$ and~$\clK_M(t,P_{\clU_{M}} y)$.
 We skip the details and refer the reader to the corresponding arguments in~\cite[Sect.~4.3]{Rod20}.

Finally, note that since~$T>0$ is arbitrary, we can conclude the existence and uniqueness of a weak solution defined for all time~$t>0$.

\subsection{On the explicit linear feedback operators}\label{sS:RmksExplicitFeed}

We show that the family
of linear feedback input operators in system~\eqref{sys-y-intro-KId} are indeed well defined and stabilizing for large enough~$M$ and ~$\lambda$.

We start with two results on the properties of the oblique projection~$P_{\widetilde\clU_{M}}^{\clU_{M}^\perp}\in\clL(H)$.

\begin{lemma}\label{L:restPUM}
We have that $P_{\widetilde\clU_{M}}^{\clU_{M}^\perp}\in\clL(H)$ restricted to~$V\subset H$ is an operator~$\widetilde P_{\widetilde\clU_{M}}^{\clU_{M}^\perp}\in \clL(V)$.
\end{lemma}
\begin{proof}
Since~$\widetilde\clU_{M}\subset V$ is a finite-dimensional space, for a suitable constant~$C_M>0$ we have
\[
\norm{\widetilde P_{\widetilde\clU_{M}}^{\clU_{M}^\perp}\phi}{V}=\norm{P_{\widetilde\clU_{M}}^{\clU_{M}^\perp}\phi}{V}\le C_M\norm{P_{\widetilde\clU_{M}}^{\clU_{M}^\perp}\phi}{H}\le C_M \norm{P_{\widetilde\clU_{M}}^{\clU_{M}^\perp}}{\clL(H)}\norm{\phi}{H},\quad\mbox{for all}\quad\phi\in V,
\]
which implies $\norm{P_{\widetilde\clU_{M}}^{\clU_{M}^\perp}}{\clL(V)}\le C_M \norm{P_{\widetilde\clU_{M}}^{\clU_{M}^\perp}}{\clL(H)}\norm{\Id}{\clL(V,H)}$.
\end{proof}
\begin{lemma}\label{L:extPUM}
We have that $P_{\clU_{M}}^{\widetilde\clU_{M}^\perp}\in\clL(H)$ can be extended to an operator~$\widehat P_{\clU_{M}}^{\widetilde\clU_{M}^\perp}\in \clL(V')$, by setting
\[
 \widehat P_{\clU_{M}}^{\widetilde\clU_{M}^\perp}f\in\clU_M\quad\mbox{such that}\quad\langle  \widehat P_{\clU_{M}}^{\widetilde\clU_{M}^\perp}f,\phi\rangle_{V',V}\coloneqq
\langle  f, P_{\widetilde\clU_{M}}^{\clU_{M}^\perp}\phi\rangle_{V',V},\quad\mbox{for all}\quad(f,\phi)\in V'\times V.
\]
\end{lemma}
\begin{proof}
Note that for all~$g\in H\xhookrightarrow{\rm d}V'$, we have that
\[
\langle  \widehat P_{\clU_{M}}^{\widetilde\clU_{M}^\perp}g,\phi\rangle_{V',V}=
\langle  g, P_{\widetilde\clU_{M}}^{\clU_{M}^\perp}\phi\rangle_{V',V}=(  g, P_{\widetilde\clU_{M}}^{\clU_{M}}\phi)_H=(P_{\clU_{M}}^{\widetilde\clU_{M}^\perp}g,\phi)_H,\quad\mbox{for all}\quad\phi\in V\xhookrightarrow{\rm d}H.
\]
Hence, $\widehat P_{\clU_{M}}^{\widetilde\clU_{M}^\perp}$ is a linear extension of $P_{\clU_{M}}^{\widetilde\clU_{M}^\perp}$ to~$V'\supset H$. The continuity follows from
\[
\norm{\widehat P_{\clU_{M}}^{\widetilde\clU_{M}^\perp}f}{V'}
=\sup_{\phi\in V\setminus\{0\}}\tfrac{ \norm{\langle  f, P_{\widetilde\clU_{M}}^{\clU_{M}^\perp}\phi\rangle_{V',V}}{\bbR} }{ \norm{\phi}{V} }
\le\sup_{\phi\in V\setminus\{0\}}\tfrac{ \norm{f}{V'}\norm{\widetilde P_{\widetilde\clU_{M}}^{\clU_{M}^\perp}\phi}{V} }{ \norm{\phi}{V} },
\]
which gives us $\norm{\widehat P_{\clU_{M}}^{\widetilde\clU_{M}^\perp}f}{V'}
\le  \norm{\widetilde P_{\widetilde\clU_{M}}^{\clU_{M}^\perp}}{\clL(V)} \norm{f}{V'}$, with~$\widetilde P_{\widetilde\clU_{M}}^{\clU_{M}^\perp}$ as in Lemma~\ref{L:restPUM}.
\end{proof}

For simplicity,  we still denote the restriction and extension of~$P_{\widetilde\clU_{M}}^{\clU_{M}^\perp}$, in Lemmas~\ref{L:restPUM} and~\ref{L:extPUM},  by~$P_{\widetilde\clU_{M}}^{\clU_{M}^\perp}$, that is,
\[
\widetilde P_{\clU_M}^{\widetilde\clU_{M}^\perp}=P_{\clU_M}^{\widetilde\clU_{M}^\perp}
\quad\mbox{and}\quad \widehat P_{\widetilde\clU_{M}}^{\clU_{M}^\perp}=P_{\widetilde\clU_{M}}^{\clU_{M}^\perp}.
\]

In particular, for this notation, we can see that the input feedback operator in~\eqref{sys-y-intro-KId} is well defined
\[
\clK_M^{\lambda,\rho} y=-\lambda P_{\clU_M}^{\widetilde\clU_{M}^\perp}A^\rho P_{\widetilde\clU_M}^{\clU_{M}^\perp}y
=-\lambda \widehat P_{\clU_M}^{\widetilde\clU_{M}^\perp}A^\rho P_{\widetilde\clU_M}^{\clU_{M}^\perp}y,\quad\mbox{for every}\quad y\in H, \quad\rho\le1.
\]
Indeed ~$\vartheta\coloneqq P_{\widetilde\clU_M}^{\clU_{M}^\perp}y\in V$ and~$A^\rho\vartheta\in\rmD(A^\frac{1-2\rho}2)\xhookrightarrow{}V'$. Furthermore, ~$\clK_M^{\lambda,\rho}\in\clL(H)$ with
\[
\norm{\clK_M^{\lambda,\rho}}{\clL(H)}=\lambda \norm{P_{\clU_M}^{\widetilde\clU_{M}^\perp}A^\rho P_{\widetilde\clU_M}^{\clU_{M}^\perp}}{\clL(H)}
\le \lambda \norm{\Id\rest{\widetilde\clU_M}}{\clL(V',H)}\black\norm{P_{\clU_M}^{\widetilde\clU_{M}^\perp}}{\clL(V')}\norm{A^\rho}{\clL(V,V')}\norm{\Id\rest{\widetilde\clU_M}}{\clL(H,V)}\norm{P_{\widetilde\clU_M}^{\clU_{M}^\perp}}{\clL(H)}.
\]

The operator~$\clK_M^{\lambda,\rho}(t,\Bigcdot)\coloneqq-\lambda A^\rho p$, $p\in\widetilde\clU_M$ is also
monotone as in Definition~\ref{D:monot},
\begin{equation}\label{monot-expli.barlam}
(\clK_M^{\lambda,\rho}p, p)_H=-\lambda\norm{p}{\rmD(A^\frac{\rho}2)}^2
\le-\lambda \norm{\Id\rest{\widetilde \clU_M}}{\clL(\rmD(A^\frac{\rho}2),H)}^{-2}\norm{p}{H}^2=-\overline\lambda\norm{p}{H}^2,
\end{equation}
with the monotonicity constant~$\overline\lambda=\lambda \norm{\Id\rest{\widetilde \clU_M}}{\clL(\rmD(A^\frac{\rho}2),H)}^{-2}$.

Next, we give the proof of Main Corollary, in the Introduction, which we
now write more precisely as follows.
\begin{corollary}\label{C:main-expli-1}
 Let Assumptions~\ref{A:A0sp}--\ref{A:poincare} hold true, let\black $\rho\le1$, ~$R>0$, and~$\mu>0$ be given,
 and let~$H=\clU_{M}\oplus \widetilde\clU_{M}^\perp$. Then,  if~$M\in\bbN_0$ and~$\lambda>0$  are large enough,
the solution of systems~\eqref{sys-y-intro-KId} and~\eqref{sys-y-intro-KId-rho}
 satisfy
 \[
  \norm{y(t)}{H}\le\ex^{-\mu (t-s)}\norm{y(s)}{H},\quad\mbox{for all}\quad t\ge s\ge0\quad\mbox{and all}\quad y_0\in H\quad\mbox{with}\quad \norm{y_0}{H}<R.
 \]
In particular, $t\mapsto \norm{y(t)}{H}$ is strictly decreasing at time~$t=s$, if~$\norm{y(s)}{H}\ne0$.
Furthermore,~$M$ and~$\overline\lambda$ can be chosen
as~$M=\ovlineC{R,\mu,C_{\rm rc},C_\clN}$ and $\lambda=\ovlineC{R,\mu,C_{\rm rc},,C_\clN}$.
\end{corollary}

\begin{proof}
We have seen in the Introduction that the
feedback~$\breve\clK_M^\lambda=-\lambda\Id\rest{\clU_M}$ corresponding
to system~\eqref{sys-y-intro-KId} is monotone with~$\overline\lambda=\lambda$.
Thus Assumption~\ref{A:monotone} is satisfied and we conclude,
by Theorem~\ref{T:main}, that system~\eqref{sys-y-intro-KId} is stable for large enough~$M$ and~$\lambda$.

By~\eqref{monot-expli.barlam} we can conclude that Assumption~\ref{A:monotone} holds for the
feedback~$\clK_M^{\lambda,\rho}=-\lambda P_{\clU_{M}}^{\widetilde\clU_{M}^\perp}A^\rho
P_{\widetilde\clU_{M}}^{\clU_{M}^\perp}\rest{\clU_M}$ corresponding to system~\eqref{sys-y-intro-KId-rho},
with~$\overline\lambda=\lambda \norm{\Id\rest{\widetilde \clU_M}}{\clL(\rmD(A^\frac{\rho}2),H)}^{-2}$.
Note that we can make $\overline\lambda$ arbitrarily large by increasing~$\lambda$.
Therefore, by Theorem~\ref{T:main}, system~\eqref{sys-y-intro-KId-rho} is stable for large enough~$M$ and~$\lambda$.
\end{proof}

\begin{remark}
Observe that in Lemma~\ref{L:MlamPoinc} (see also Remark~\ref{R:depMlam1}) it is sufficient to take
$M=\ovlineC{R,\mu,C_{\rm rc},C_\clN}$ and~$\overline\lambda=\ovlineC{R,\mu,C_{\rm rc},C_\clN}$ large enough.
In particular, for~$\lambda$ it is enough to take large
enough~$\lambda =\ovlineC{R,\mu,C_{\rm rc},C_\clN}
\norm{\Id\rest{\widetilde \clU_M}}{\clL(\rmD(A^\frac{\rho}2),H)}^2$.
Defining the constants~$\widehat\varsigma_M\coloneqq\max\limits_{1\le m\le M}
\norm{\Id\rest{\widetilde \clU_M}}{\clL(\rmD(A^\frac{\rho}2),H)}^2$ and observing
that~$\widehat\varsigma_M=\ovlineC{M}$, we can conclude that it is enough to take
$M=\ovlineC{R,\mu,C_{\rm rc},C_\clN}$ and~$\lambda=\ovlineC{R,\mu,C_{\rm rc},C_\clN}$.
\end{remark}

\section{Learning an optimal feedback control}\label{S:learn}
From Theorem~\ref{T:main} we know that for arbitrary $R>0$ any feedback operator in the
form~$\clK_M \circ P_{\clU_M}$ with~$\clK_M\colon\clU_M\to\clU_M$, see~\eqref{FeedK-intro},
allows us to stabilize the solution of system~\eqref{sys-y-intro-K}
from every initial condition~$y_0$ in the closed ball~
$ \overline\clB_{R}^H\coloneqq\{v\in H\mid \norm{v}{H}\le R\}$, provided that~$\clK_{M}$
is monotone in the sense of Definition~\ref{D:monot},
and~$M$ as well as~$\overline\lambda$ are sufficiently large. In particular, these
assumptions cover the linear operator~$\clK_M=-\lambda\Id$. From the application perspective
a next natural question is to find a feedback of this form which is also optimal with
respect to some cost functional. For example, we may seek to find a balance between
the energies of the trajectory and the feedback control. This is the aim of the present section.
Here we restrict ourselves to autonomous functions~$\clK_{M}$, however,
the presented framework can be extended to the nonautonomous case.

\subsection{Feedback control as learning problem}\label{subsec:formlearn}
Define the cost functional
\begin{align*}
J(y,u)=\frac{1}{2}\int^\infty_0 \lbrack |y(t)|^2_V+ \beta |u(t)|^2_H \rbrack~\mathrm{d}t
\end{align*}
where~$\beta>0$ balances the trade-off between the energy of the
trajectory~$y \in L^2((0,\infty),V)$ and the magnitude of the control~$u \in L^2((0,\infty),\clU_M)$.
For the rest of this section we fix the radius~$R$. Moreover for abbreviation we
set~$Y\coloneqq L^2((0,\infty),V)$.

A first idea for defining an ``optimal" low-dimensional feedback law could be to seek for a pair~$(\overline{\mathbf{y}}, \overline{\clK}_M)$ with
\begin{align*}
J(\overline{y},\overline{\clK}_M(P_{\clU_M}\overline{\mathbf{y}}(y_0)))
\le J(\mathbf{y}(y_0),{\clK}_M(P_{\clU_M}\mathbf{y}(y_0)))  \quad \forall y_0 \in \overline{\clB}^H_R,
\end{align*}
amongst all pairs~$(\mathbf{y}, \clK_M)$ of ensemble states~$\mathbf{y}\colon \overline\clB_{R}^H \to Y $ such that~$\mathbf{y}(y_0)$ satisfies~\eqref{sys-y-intro-K} and feedback laws~$\clK_M$ which are monotone in the sense of Definition~\ref{D:monot}.
This is unfeasible, but it guides the way to successful approaches.

Indeed, we shall  pursue an approach similar to~\cite{KunischWalter_arx20} and compute
a feedback minimizing the expected cost of trajectories originating from a ``training set" of initial conditions described by a probability measure~$\fkm$ on~$ \overline\clB_{R}^H$. Moreover we restrict the search for the optimal feedback law to a subset
\begin{align*}
\clC_{\theta} := \operatorname{Im}(\clK^\cdot_M)= \left\{\,\clK^{(\theta)}_M\;|\;\theta \in
\mathbb{R}^N\,\right\} \subset \clC(\clU_M,\clU_M).
\end{align*}
parametrized by a mapping~$\clK^{(\Bigcdot)}_M \black\colon \mathbb{R}^N \to \clC(\clU_M,\clU_M),~N\in \mathbb{N} $.
It is assumed that~$-\lambda \Id \in \clC_{\theta}$ for every~$\lambda\geq 0$. Henceforth we fix~$\overline{\lambda}>0$ and~$\varepsilon>0$. With these specifications we choose~$\overline{K}_{M}=\clK^{\overline{\theta}}_M \in\clC_{\theta}$ as the solution to
\begin{align} \label{def:learningprob}
\inf_{\substack{\mathbf{y} \colon \overline{\clB}^H_R \to Y ,\\ \theta \in \mathbb{R}^N }}
\left \lbrack\int_{\overline{\clB}^H_R} J(\mathbf{y}(y_0),\clK_{M}(P_{\clU_M}
\mathbf{y}(y_0)))\,\rmd\fkm\black(y_0)+ \mathcal{G}(\theta) \right \rbrack, \tag{$\mathcal{P}$}
\end{align}
where~$y=\mathbf{y}(y_0)\in Y$  fulfills
\begin{align} \label{eq:stateeq}
 \dot y +Ay+A_{\rm rc}y +\clN(y) = \clK^{(\theta)}_M( P_{\clU_{M}} y),\qquad y(0)= y_0,
\end{align}
for~$\fkm$-a.e.~$y_0 \in \overline{\clB}^H_R$ and~$\clK^{(\theta)}_M$ satisfies
\begin{align} \label{eq:stateconst}
(\clK^{(\theta)}_M(p), p)_H\le -\overline{\lambda}\norm{p}{H}^2 \quad \forall p \in \overline{\clB}^H_{R+\varepsilon} \cap \clU_M.
\end{align}
Additional penalty terms~$\mathcal{G}\colon \mathbb{R}^N \to \mathbb{R}_+ \cup \{+\infty\}$ can be added to the objective functional to enforce constraints on the parameters~$\theta$ and/or to guarantee the radial unboundedness of the objective functional.
From Theorem \ref{T:main} and~$-\lambda \Id \in \mathcal{C}_\theta$,~$\lambda>0$, we deduce that
for $\overline \lambda$ and~$M \in \mathbb{N}$ large enough, there are feasible points
for problem~\eqref{def:learningprob} under the constraints \eqref{eq:stateeq} and \eqref{eq:stateconst}. Note that, in comparison to Definition~\ref{D:monot}, we only require~\eqref{eq:stateconst} to hold for~$p\in \overline{\clB}^H_{R+\varepsilon} \cap \clU_M$. This change is necessary since enforcing monotonicity on the whole space~$\mathcal{U}_M$ constitutes a numerical burden which is hard to realize in practice. It is readily shown that Theorem~\ref{T:main} still holds for feedback laws~$\mathcal{K}^{(\theta)}_M$ satisfying~\eqref{eq:stateconst}. In particular, if~$\theta\in \mathbb{R}^N$ is admissible, and~$\bar{\lambda}$ as well as~$M$ are large enough then~\eqref{eq:stateeq} is exponentially stable for all~$y_0 \in \overline{\clB}^H_{R}$.
\begin{example} \label{examplsubsets}
In the remainder of this  section we briefly discuss different choices for ~$\clC_{\theta}$. We may simply choose
\begin{align*}
\clC^1_{\theta}=\left\{\, -\lambda \Id\;|\;\lambda \geq 0\,\right\}
\end{align*}
or more generally
\begin{align*}
\clC^2_{\theta}=\left\{\,\mathcal{K}^{(\theta)}_M \in \mathcal{L}(\mathcal{U}_M)\;|\;\left( \mathcal{K}^{(\theta)}_M\right)^*=\mathcal{K}^{(\theta)}_M\,\right\} \simeq \operatorname{Sym}(M_\sigma).
\end{align*}
A feedback law in this form stabilizes initial conditions~$y_0 \in \overline{\clB}^H_R$ at an exponential rate
if its largest eigenvalue is small (negative) enough and if~$M$ is large enough.

Our focus lies on determining low dimensional feedback laws which are induced by realizations of certain neural networks. They will be realized numerically in  Section~\ref{S:simul}.
Such networks have recently received tremendous attention due to their excellent approximation properties in practice. A tuple of parameters
\begin{align*}
\theta=(W_{1,1},W_{2,1},b_1, \cdots, W_{1,L},W_{2,L},b_L) \in \mathcal{R}
\end{align*}
where
\begin{align*}
\mathcal{R}= \bigtimes^{L-1}_{i=1} \left ( \mathbb{R}^{N_{i} \times N_{i-1}} \times \mathbb{R}^{N_{i} \times N_{i-1}} \times \mathbb{R}^{N_i} \right ) \times \mathbb{R}^{N_{L} \times N_{L-1}} \times \mathbb{R}^{N_{L} \times N_0} \times \mathbb{R}^{N_L},~N_i \in \mathbb{N},~N_0=N_L=M_\sigma,
\end{align*}
is called a~\textit{residual network} with~$L\geq 2$ layers.

Fixing the~\textit{activation function}~$\chi \in \clC(\mathbb{R},\mathbb{R})$,
we define the~\textit{realization}~$K^{(\theta)}_M \in \clC(\mathbb{R}^{M_\sigma},
\mathbb{R}^{M_\sigma})$ of~$\theta\in \mathcal{R}$ by
\begin{align}\label{def:realization}
K^{(\theta)}_M(x)= f_{\theta,L} \circ f_{\theta,L-1} \circ \cdots \circ f_{\theta,1}(x)- f_{\theta,L} \circ f_{\theta,L-1} \circ \cdots \circ f_{\theta,1}(0)+W_{2,L}x,
\end{align}
where
\begin{align*}
f_{\theta,i}(x)= \chi(W_{1,i}x+ b_i )+ W_{2,i} x \quad \forall x \in \mathbb{R}^{N_i}
\end{align*}
for~$i=1,\dots,L-1$ and
\begin{align*}
f_{\theta,L}(x)=W_{1,L}x+b_L.
\end{align*}
Here the action of~$\chi$ has to be understood componentwise, and we note that $K^{(\theta)}_M(0)=0$.
Finally set
\begin{align*}
\clC^3_{\theta}= \{\,\mathbb{P}_{\clU_M}^{-1} K^{(\theta)}_M \mathbb{P}_{\clU_M}\;|\;\theta \in \mathcal{R}\,\},
\end{align*}
where~$\mathbb{P}_{\clU_M} \colon \clU_M \to \mathbb{R}^{M_\sigma} $ denotes the coordinate mapping for a basis~$\{\zeta_i\}^{M_\sigma}_{i=1}$ of~$\clU_M$ i.e.
\begin{align*}
\mathbb{P}_{\clU_M}(u)=\mathbb{P}_{\clU_M}\left(\sum^{M_\sigma}_{i=1}\lambda_i \zeta_i\right)=(\lambda_1,\dots,\lambda_{M_\sigma})^\top
\end{align*}
for all~$u \in \clU_M$.
These feedbacks satisfy~$\clK_M^{(\theta)}(0)=0$ by construction and~$-\lambda \Id \in \clC^3_{\theta} $ for every~$\lambda \in\mathbb{R}$.
A suitable penalization term for this type of parametrization is given by
\begin{align*}
\mathcal{G}(\theta)=\frac{\alpha}{2} \|\theta\|_{\mathcal{R}}= \frac{\alpha}{2} \left( \sum^{L}_{i=1} \left  \lbrack \|W_{1,i}\|^2+\|W_{2,i}\|^2+ |b_i|^2 \right \rbrack \right)
\end{align*}
where~$\alpha>0$.
\end{example}

\subsection{Practical realization}
In order to practically compute a stabilizing feedback via problem~\eqref{def:learningprob}
we have to address several discretization aspects. First the state space~$V$ is replaced by a
finite dimensional subspace~$V_h \subset V$. Accordingly we consider discretized diffusion and
reaction operators~$A_h,~A_{\rm rc,h} \colon V_h \to V_h$ as well as a discretization of the
nonlinearity~$\clN_h\colon V_h \to V_h$. Next we cut-off the time integral at some~$T>0$ and
approximate the integral with respect to~$\fkm$ using quadrature
points~$\{y^i_0\}^{N_0}_{i=1} \subset \overline{\clB}^H_{R}$. Finally note that the monotonicity constraint is difficult to implement in practice. Therefore it is replaced by a penalization of the form
\begin{align*}
\mathcal{G}_\gamma^{N_1}(\theta)\coloneqq \frac{\gamma}{(1+\varepsilon_1)\,N_1}
\sum^{N_1}_{j=1}\left((\mathcal{K}^{(\theta)}_M(p_j), p_j)_H
+\overline{\lambda}|p_j|^2\right)^{1+\varepsilon_1}_+
\end{align*}
where~$\{p_j\}^{N_1}_{j=1} \subset \overline{\clB}^H_{R+\varepsilon} \cap \clU_M. $
Moreover~$\gamma >0$ is a penalty parameter,~$\varepsilon_1>0$, and~$(\Bigcdot)_+=\max(\Bigcdot,0)$.
We arrive at the discretized problem
\begin{align} \label{def:discproblem}
\inf_{ \substack{\{y_i\}^{N_0}_{i=1} \subset L^2((0,T),V_h), \\ \theta \in \mathbb{R}^N} }
\left \lbrack \frac{1}{2N_0} \sum^{N_0}_{i=1}  \int^T_0 \left[ |y_i(t)|_V^2
+\beta |\mathcal{K}^{(\theta)}_M(P_{\clU_M} y_i(t))|^2_H]\rmd t\right \rbrack+\mathcal{G}(\theta)+\mathcal{G}^{N_1}_\gamma(\theta) \right \rbrack
\end{align}
subject to
\begin{align} \label{eq:discequation}
\dot y_i +A_hy_i+A_{\rm rc,h}y_i +\clN_h(y_i) =\mathcal{K}^{(\theta)}_M(P_{\clU_M} y_i),\qquad y(0)= y^i_0
\end{align}
for all~$i=1,\dots,N_0$.
For abbreviation denote the objective functional in~\eqref{def:discproblem} by~$j(y_1,\dots,y_{N_0},\theta)$.

From now on we tacitly assume that the nonlinearity~$\mathcal{N}_h$ and the feedback
parametrization~$\clK^\cdot_M$ are such that the induced superposition operators are at
least continuously Fr\'echet differentiable on~$H^1((0,T),V_h)$ and~$\mathbb{R}^N \times H^1((0,T),V_h)$, respectively.
Moreover the penalty term~$\mathcal{G}$ is smooth. Derivatives are denoted by~$``\partial"$ in the following with an
additional subscript if it is a partial one.
In order to actually solve the learning problem algorithmically we will rely on first order methods. Therefore
it remains to argue the differentiability of~$j$ and to give a representation of its gradient. To fix ideas
let~$\{y_i\}^{N_0}_{i=1}$ and~$\overline{\theta} \in \mathbb{R}^N$ be an admissible point, that is, $y_i$
satisfies~\eqref{eq:discequation} given~$\theta=\overline{\theta}$,~$i=1,\dots,N$. In virtue of the
implicit function theorem there is a neighbourhood~$N(\bar{\theta})$ of~$\bar{\theta}$ and a~$\clC^1$ operator
\begin{align*}
S \colon  N(\bar{\theta}) \to   \bigtimes^{N_0}_{i=1} H^1((0,T),V_h), \quad S(\theta)=(y_1(\theta), \dots, y_{N_0}(\theta))
\end{align*}
such that~$y_i(\theta)$ satisfies~\eqref{eq:discequation},~$i=1,\dots,N_0$. This implies that the reduced objective functional
\begin{align*}
\mathcal{J}(\theta) \coloneqq j(y_1(\theta), \dots, y_{N_0}(\theta))
\end{align*}
is smooth around~$\overline{\theta}$. Using adjoint calculus its gradient is given by
\begin{align*}
\partial\mathcal{J}(\theta)
=\frac{1}{N_0}\sum^{N_0}_{i=1} \int^T_0 \left\lbrack \partial_\theta
\mathcal{K}^{(\theta)}_M(P_{\clU_M} y_i(t))^*(p_i(t)
+\beta \mathcal{K}^{(\theta)}_M(P_{\clU_M} y_i(t)))
\right\rbrack\rmd t+ \partial \mathcal{G}(\theta)+\partial\mathcal{G}^{N_1}_\gamma(\theta)
\end{align*}
where the function~$p_i \in H^1((0,T),V_h)$ satisfies~$p_i(T)=0$ and
\begin{align*}
\dot p_i -(A^*_h+A^*_{\rm rc,h} +\partial\clN_h(y_i)^*
- P_{\clU_M}\partial_y \mathcal{K}^\theta_M(P_{\clU_M} y_i)^*)p_i=-y_i-\beta  P_{\clU_M}\partial_y \mathcal{K}^\theta_M(P_{\clU_M} y_i)^* \mathcal{K}^\theta_M(P_{\clU_M} y_i)
\end{align*}
for all~$i=1,\dots,N_0$.

\begin{remark}
Note that the monotonicity enhancing penalty term~$\mathcal{G}^{N_1}_{\gamma}$ in~\eqref{def:discproblem} can be interpreted as a quadrature of
\begin{align*}
\mathcal{G}_\gamma (\theta)\coloneqq  \frac{\gamma}{1+\varepsilon_1}
\int_{\overline{\clB}^H_{R+\varepsilon} \cap\, \clU_M}\left((\mathcal{K}^{(\theta)}_M(p), p)_H
+\overline{\lambda}|p|^2\right)^{1+\varepsilon_1}_+ \,\mathrm{d} \mathcal{L}(p)
\end{align*}
where~$\mathcal{L}$ denotes the normalized Lebesgue measure on~$\overline{\clB}^H_{R+\varepsilon} \cap \clU_M$.
This corresponds to a Moreau--Yosida regularization of the monotonicity constraint in problem~\eqref{def:learningprob}.
Concerning the practical implementation of~\eqref{def:discproblem},
we made good experience with choosing a large number~$N_1$ of randomly sampled functions~$\{p_j\}^{N_1}_{j=1}$. On the other hand we deliberately   chose a significantly smaller number~$N_0$ of  initial conditions~$\{y^i_0\}^{N_1}_{i=1}$ in the training set. In Section~\ref{S:simul}, e.g., we rely on the first few leading, unstable, eigenvectors of the diffusion operator in~\eqref{sys-y-intro-K} which yields satisfactory results. This discrepancy between~$N_0$ and~$N_1$ is mainly motivated by two observations. First, enforcing the monotonicity constraint in a large number of points~$\{p_j\}^{N_1}_{j=1}$ enhances the stabilizing properties of~$\mathcal{K}^{(\theta)}_M$ and eventually ensures the stabilization of initial conditions outside of the training set. Second, computing those parts of the gradient~$\mathcal{J}(\theta)$ which depend on the training set requires $N_0$ solves of the state and adjoint equation, respectively. This totals~$2 N_0$ PDE solves for one gradient evaluation. In
contrast, the gradient of the penalty term is given by
\begin{align*}
\partial \mathcal{G}^{N_1}_\gamma(\theta)= \frac{\gamma}{N_1} \sum^{N_1}_{j=1} \left((\mathcal{K}^{(\theta)}_M(p_j), p_j)_H
+\overline{\lambda}|p_j|^2\right)^{\varepsilon_1}_+ \, \partial_\theta \mathcal{K}^{(\theta)}_M(p_j)^* p_j,
\end{align*}
that is, it can be efficiently computed if the evaluation of~$\mathcal{K}^{(\theta)}_M$ and its derivatives is cheap (which is  the case, e.g., for realizations of neural networks).
\end{remark}

 \section{Example of application}\label{S:parabolic-OK}
We show that the parabolic coupled system~\eqref{sys-y-parab-intro},
evolving in spatial rectangular domains
\begin{equation}\label{Omegaxd}
\Omega=\Omega^\times=(0,L_1)\times (0,L_2)\times \cdots\times (0,L_d)\subset\bbR^d,\qquad d\in\{1,2,3\},
\end{equation}
is stable for large enough~$M$ and~$\lambda$, and for suitable chosen sets of actuators. Here~$M$ and~$\lambda$ (may) depend on the norm of the initial condition. The same arguments
can be extended to parabolic equations evolving
in  general  convex polygonal domains.

It is enough to show that our Assumptions~\ref{A:A0sp}--\ref{A:monotone} are satisfied.
Assumptions~\ref{A:A0sp}--\ref{A:A1} are satisfied with
$A$ and~$A_{\rm rc}$ as in Section~\ref{sS:illust-parabolic},  see~\cite[Sect.~5]{PhanRod18-mcss}.

Assumption~\ref{A:poincare} is satisfied for a suitable placement of the actuators, where as actuators we take  indicator functions~$\Phi_{\omega_{j}^M}=1_{\omega_{j}^M}$ of rectangular subdomains~$\omega_{j}^M$.
Figure~\ref{fig.suppActSens}
illustrates the actuators regions for a planar rectangle~$\Omega^\times\in\bbR^2$,
where the number of actuators is given by~$M_\sigma=\sigma(M)=M^2$, $M\in\bbN_0$.
We take analogous regions in other dimensions, that is for rectangular domains~$\Omega^\times\subset\bbR^d$ as in~\eqref{Omegaxd}, where we will have $M_\sigma=M^d$ actuators.

We denote the subrectagles for each ~$M\in\bbN_0$ by~$\omega^M_j$. Hence, our set of actuators is
\[
{U_{M}}=\{\indf_{\omega^M_j}\mid 1\le j\le M_\sigma\}\subset H=L^2(\Omega).
\]


\setlength{\unitlength}{.0018\textwidth}
\newsavebox{\Rectfw}%
\savebox{\Rectfw}(0,0){%
\linethickness{2pt}
{\color{black}\polygon(0,0)(120,0)(120,80)(0,80)(0,0)}%
}%

\newsavebox{\RectRef}%
\savebox{\RectRef}(0,0){%
\linethickness{1.5pt}
{\color{lightgray}\polygon*(40,30)(80,30)(80,50)(40,50)(40,30)}%
}%

 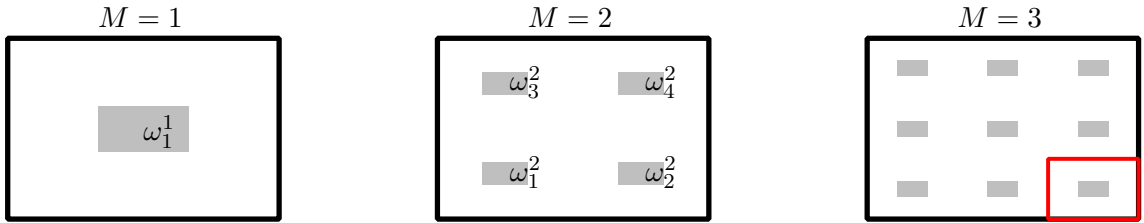
\begin{figure}[h!]
\begin{center}
\begin{picture}(500,100)

 \put(0,0){\usebox{\Rectfw}}
\put(0,0){\usebox{\RectRef}}
  \put(190,0){\usebox{\Rectfw}}
   \put(190,0){\scalebox{.5}[.5]{\usebox{\RectRef}}}
    \put(250,0){\scalebox{.5}[.5]{\usebox{\RectRef}}}
   \put(190,40){\scalebox{.5}[.5]{\usebox{\RectRef}}}
  \put(250,40){\scalebox{.5}[.5]{\usebox{\RectRef}}}
%
 \put(380,0){\usebox{\Rectfw}}
  \put(380,0){\scalebox{.333333}[.333333]{\usebox{\RectRef}}}
 \put(420,0){\scalebox{.333333}[.333333]{\usebox{\RectRef}}}
 \put(460,0){\scalebox{.333333}[.333333]{\usebox{\RectRef}}}
  \put(380,26.66666){\scalebox{.333333}[.333333]{\usebox{\RectRef}}}
 \put(420,26.66666){\scalebox{.333333}[.333333]{\usebox{\RectRef}}}
 \put(460,26.66666){\scalebox{.333333}[.333333]{\usebox{\RectRef}}}
  \put(380,53.33333){\scalebox{.333333}[.333333]{\usebox{\RectRef}}}
 \put(420,53.33333){\scalebox{.333333}[.333333]{\usebox{\RectRef}}}
 \put(460,53.33333){\scalebox{.333333}[.333333]{\usebox{\RectRef}}}

 \put(40,85){$M=1$}
 \put(230,85){$M=2$}
 \put(420,85){$M=3$}
\put(60,35){$\omega_1^1$}
\put(222,17){$\omega_1^2$}
\put(282,17){$\omega_2^2$}
\put(222,57){$\omega_3^2$}
\put(282,57){$\omega_4^2$}

\linethickness{1.5pt}
{\color{red}\polygon(460,0)(500,0)(500,26.66666)(460,26.66666)(460,0)}
\end{picture}
\end{center}
 \caption{Supports of actuators in the rectangle~$\Omega^\times\subset\bbR^2$.} \label{fig.suppActSens}
 \end{figure}

To construct the explicit feedback input, we need  the auxiliary set~$\widetilde\clU_M$.
Observe that, for a suitable lower bound tuple $((p_1)_{j}^M,(p_2)_{j}^M,\dots,,(p_d)_{j}^M)$ we have
\[
 \omega_j^M=((p_1)_{j}^M,(p_1)_{j}^M+\tfrac{l_1}{M})\times((p_2)_{j}^M,(p_2)_{j}^M+\tfrac{l_2}{M})
 \times\dots\times ((p_d)_{j}^M,(p_d)_{j}^M+\tfrac{l_d}{M}),\quad 1\le j\le M_\sigma,
\]
for the interior subrectangles as in Figure~\ref{fig.suppActSens}. Where essentially
we partition the rectangle into $M_\sigma$ similar (rescaled) rectangles and put one
(rescaled) actuator in each subrectange. See one of these copies highlighted in
Figure~\ref{fig.suppActSens}, at the right-bottom corner of the case $(M,d)=(3,2)$.
More details can be found in~\cite[Sect.~4]{Rod21-sicon}.
Above~$p_j^M=((p_1)_{j}^M,(p_2)_{j}^M,\dots,(p_d)_{j}^M)\in\bbR^d$
is the lower vertex of~$\omega_{j}^M$, and $\frac{l}{M}\in\bbR_0^d$ is the vector of sides length.\black

\begin{remark}\label{R:volAct}
By construction, in Figure~\ref{fig.suppActSens}, the total volume (area) covered by the actuators is independent of~$M$.
\end{remark}

Assumption~\ref{A:monotone} will be satisfied by~$\clK_M^{\lambda,\rho}$ as we have seen in Section~\ref{sS:RmksExplicitFeed}.

Therefore, it remains to show that the nonlinearity~$\clN(t,y)=-\norm{y}{\bbR}y$ satisfies Assumption~\ref{A:NN}, for~$d\in\{1,2,3\}$.  Observe that,   for the more general nonlinearity
\[
\clN_r(t,y)=-\norm{y}{\bbR}^{r-1}y,\qquad 1<r\le3,
\]
for $y\in V\subseteq H^1(\Omega)\xhookrightarrow{} L^6(\Omega)$, we obtain
\begin{align*}
&\norm{\clN_r(t,y_1)-\clN_r(t,y_2)}{V'}=\norm{\norm{y_1}{\bbR}^{r-1}y_1 -\norm{y_2}{\bbR}^{r-1}y_2}{V'}=\sup_{v\in V\setminus\{0\}}\frac{\langle\norm{y_1}{\bbR}^{r-1}y_1-\norm{y_2}{\bbR}^{r-1}y_2,v\rangle_{V',V}}{\norm{v}{V}}\\
&\hspace*{2em}
=\sup_{v\in V\setminus\{0\}}\frac{(\norm{y_1}{\bbR}^{r-1}y_1-\norm{y_2}{\bbR}^{r-1}y_2,v)_H}{\norm{v}{V}}\le \sup_{v\in V\setminus\{0\}}\frac{\norm{\norm{y_1}{\bbR}^{r-1}y_1-\norm{y_2}{\bbR}^{r-1}y_2}{L^\frac65}\norm{v}{L^6}}{\norm{v}{V}}\\
&\hspace*{2em}
\le C_1\norm{\norm{y_1}{\bbR}^{r-1}y_1-\norm{y_2}{\bbR}^{r-1}y_2}{L^\frac65}.
\end{align*}
Then, following arguments as in~\cite[Sect.~5.2.1]{Rod20} we obtain that
\begin{align*}
\norm{\clN_r(t,y_1)-\clN_r(t,y_2)}{V'}
&\le C_2\left(\norm{y_1}{L^{\frac{6r}5}}^{r-1}+\norm{y_2}{L^{\frac{6r}5}}^{r-1}\right)\norm{y_1-y_2}{L^{\frac{6r}5}}.
\end{align*}

Recall that for~$\Omega\in\bbR^d$ the Sobolev relations, see~\cite[Thm.~4.57]{DemengelDem12},
\begin{align}\label{sobol-rel}
2s<d\quad\mbox{and}\quad 1\le q\le\tfrac{2d}{d-2s}
\end{align}
give us the embedding~$H^s(\Omega)=W^{s,2}(\Omega)\xhookrightarrow{}L^q(\Omega)$.

Now, in the case~$r=2$, the relations~\eqref{sobol-rel} hold true for~$d\in\{1,2,3\}$ and~$(s,q)=(\frac14,\frac{12}5)$, hence~$H^\frac{1}{4}(\Omega)\xhookrightarrow{}L^\frac{12}{5}(\Omega)$,
and so we find, by an interpolation argument, that
\begin{align*}
\norm{\clN(t,y_1)-\clN(t,y_2)}{V'}&=
\norm{\clN_2(t,y_1)-\clN_2(t,y_2)}{V'}\\
&\le C_3\left(\norm{y_1}{H}^{\frac{3}{4}}\norm{y_1}{V}^{\frac{1}{4}}+\norm{y_2}{H}^{\frac{3}{4}}\norm{y_2}{V}^{\frac{1}{4}}\right)\norm{y_1-y_2}{H}^{\frac{3}{4}}\norm{y_1-y_2}{V}^{\frac{1}{4}},
\end{align*}
from which we can conclude that Assumption~\ref{A:NN} is satisfied by~$\clN=\clN_2$, with~$n=1$ and
\[
\delta_{11}=\zeta_{11}=\tfrac{3}{4},\qquad
\delta_{21}=\zeta_{21}=\tfrac{1}{4}.
\]
Notice that~$\delta_{11}+\delta_{21}\ge1$ and~$\delta_{21}+\zeta_{21}<1$.

\begin{remark}
For simplicity, above we have dealt with the cases~$d\in\{1,2,3\}$ simultaneously. However,
since the Sobolev relations~\eqref{sobol-rel} depend on~$d$ the nonlinearities which satisfy Assumption~\ref{A:NN} also depend on the dimension~$d$. For example, we do not know whether the nonlinearity~$\clN_3=-\norm{y}{\bbR}^2y$ satisfies Assumption~\ref{A:NN}
in the case~$d\in\{2,3\}$, but the arguments above show that it satisfies the same assumption for~$d=1$.\black Indeed, we can find that
\begin{align*}
\norm{\clN_3(t,y_1)-\clN_3(t,y_2)}{V'}
&\le C_1\left(\norm{y_1}{L^{\frac{18}5}}^2+\norm{y_2}{L^{\frac{18}5}}^2\right)\norm{y_1-y_2}{L^{\frac{18}5}},\\
H^\frac{2}{9}(\Omega)&\xhookrightarrow{}L^\frac{18}{5}(\Omega),\qquad d=1.
\end{align*}
Therefore we arrive at
\begin{align*}
\norm{\clN_3(t,y_1)-\clN_3(t,y_2)}{V'}
&\le C_2\left(\norm{y_1}{H}^{\frac{14}{9}}\norm{y_1}{V}^{\frac{4}{9}}+\norm{y_1}{H}^{\frac{14}{9}}\norm{y_1}{V}^{\frac{4}{9}}\right)\norm{y_1-y_2}{H}^{\frac{14}{9}}\norm{y_1-y_2}{V}^{\frac{4}{9}},\qquad d=1,
\end{align*}
from which we can conclude that Assumption~\ref{A:NN} is satisfied by~$\clN=\clN_3$, with~$n=1$ and
\[
\delta_{11}=\zeta_{11}=\tfrac{14}{9},\qquad
\delta_{21}=\zeta_{21}=\tfrac{4}{9}.
\]
Again, notice that~$\delta_{11}+\delta_{21}\ge1$ and~$\delta_{21}+\zeta_{21}<1$.
\end{remark}

\section{Numerical experiments}\label{S:simul}

This section serves to illustrate the main theoretical results  by numerical simulations.
For this purpose consider a particular instance of the parabolic system in
\eqref{sys-y-parab-intro},
\begin{align} \label{eq:numeq}
&\tfrac{\p}{\p t} y +(-0.1\Delta+\Id) y+ay +b\cdot\nabla y - \norm{y}{\bbR}y
 =\mathcal{K}_M ( P_{\clU_{M}} y),\quad \tfrac{\p}{\p\bfn}y\rest{\Gamma}=0,\quad
y(0)=y_0,
\end{align}
where the state~$y$ evolves in~$H=L^2(\Omega)$ on the unit square~$\Omega=(0,1)^2$.
The set of actuators~$\mathcal{U}_M$ is chosen as in Section~\ref{S:parabolic-OK}
and the coefficient functions~$a$ and~$b$ are defined individually for each numerical
example.

\begin{figure}[ht]
\centering
\subfigure
{\includegraphics[width=0.3\textwidth]{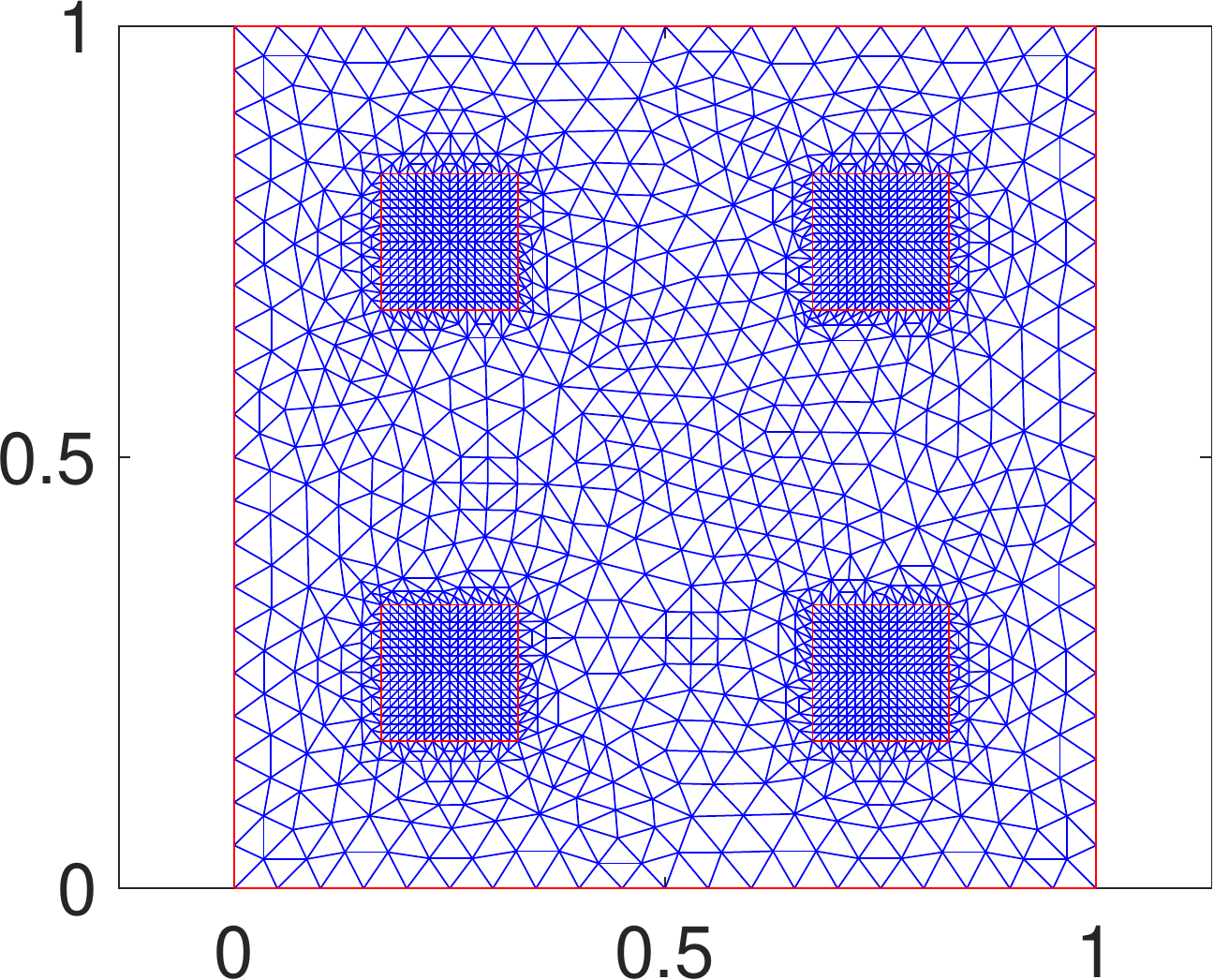}}
\subfigure
{\includegraphics[width=0.3\textwidth]{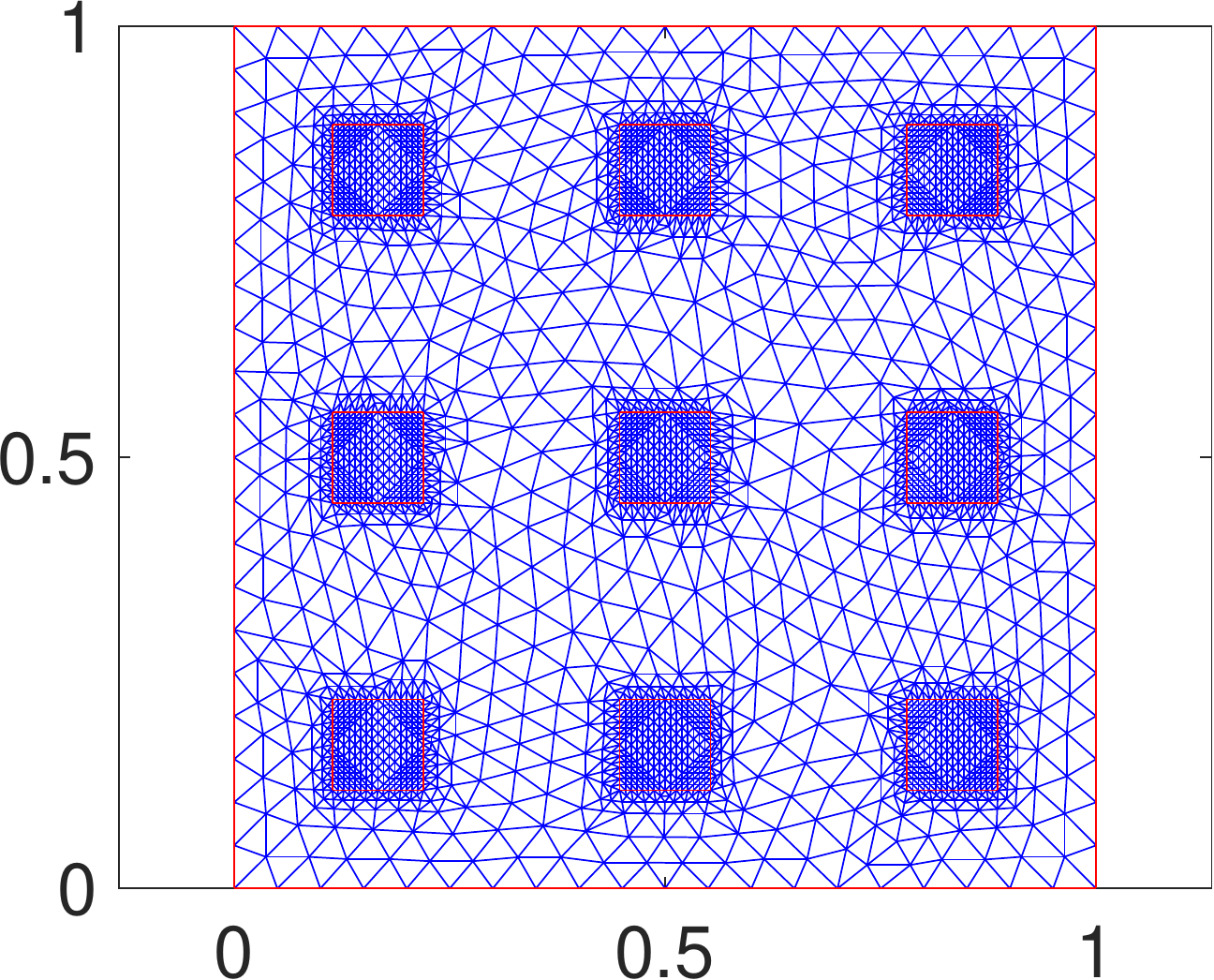}}
\subfigure
{\includegraphics[width=0.3\textwidth]{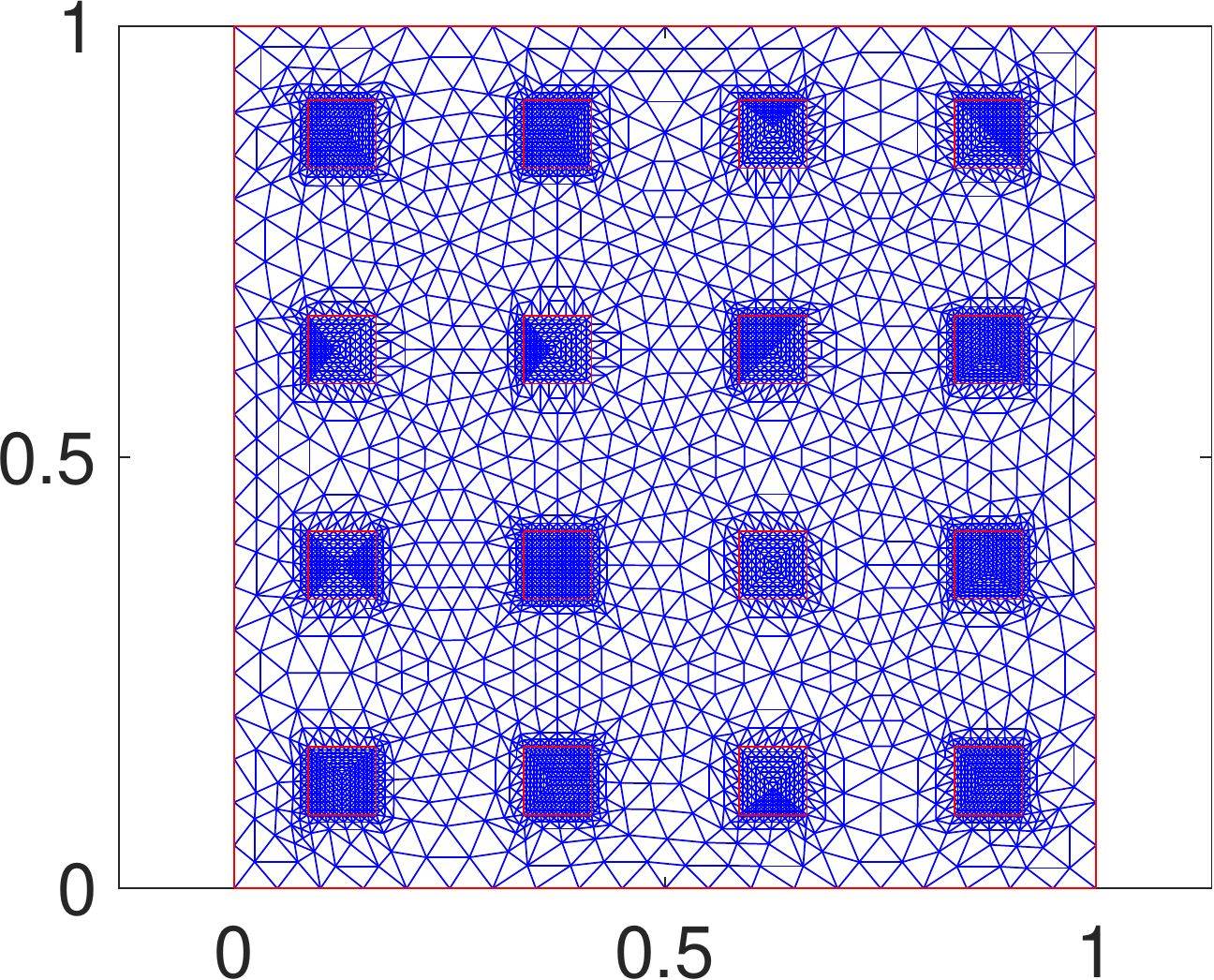}}
\caption{Actuators locations and triangulations of~$\Omega$}
\label{Fig:Mesh}
\end{figure}

The remainder of this section is split into two parts. The first one aims to confirm
the statement of the stabilization result in Corollary \ref{C:main-expli-1}.
For this purpose we consider feedback
controls of the form~$\mathcal{K}^\lambda_M({P}_{\mathcal{U}_M}y)=-\lambda P_{\mathcal{U}_M}y$,~$\lambda>0$,
and show that the closed-loop system is exponentially stable in the~$H$-norm provided the
number~$M_\sigma=\dim\clU_M$
of actuators and the scalar~$\lambda>0$ are both large enough.
In the second part we compare~$\mathcal{K}^\lambda_M({P}_{\mathcal{U}_M}y)$
with neural network induced feedback laws~$\mathcal{K}^{(\theta)}_M({P}_{\mathcal{U}_M}y)$
obtained from solving the learning problem.

For the numerical approximation of~\eqref{eq:numeq} in space we rely on a piecewise
linear finite element ansatz on locally refined triangulations of~$\Omega$ which
well resolve the support of the actuators. The meshes chosen for~$4,~9,~16$ actuators
are depicted in Figure~\ref{Fig:Mesh}. Concerning the temporal
discretization, a semi-implicit time-stepping scheme with stepsize~$k>0$ is applied.
More precisely, the symmetric linear term~$(-\nu \bigtriangleup +\mathbf{1} +a )y$
is treated by a Crank-Nicolson scheme and for the remaining terms
a first order Adams-Bashforth method is  used. All calculations were carried out in
MATLAB.

\subsection{Scaled orthogonal projection feedback laws}
In the following we choose the following parameters in \eqref{sys-y-parab-intro}:
\begin{align*}
 a \coloneqq -2+x_1 -\norm{\sin(t+x_1)}
 {\bbR},&\qquad b\coloneqq
 \begin{bmatrix}
  x_1+x_2\\
  \cos(t)x_1x_2
  \end{bmatrix}, \qquad y_0= \overline y_0\coloneqq\frac{1-2x_1x_2}{\norm{1-2x_1x_2}{H}}, \quad \nu=0.1.
\end{align*}


We briefly address the behavior of~\eqref{eq:numeq} if no control is applied i.e. for~$\lambda=0$. In this case, the norm of the solution~$y$ blows up in finite time, c.f. Figure~\ref{Fig:y_4and16lam00Fon01nIC11FTYOrth}. Note that the simulations were run for two different meshes namely those corresponding to~$4$ and~$16$ actuators, respectively. The computed results suggest that the blow-up time of the uncontrolled dynamics is independent of the spatial discretization.
\begin{figure}[ht]
\centering
\subfigure
{\includegraphics[width=0.45\textwidth,height=0.33\textwidth]{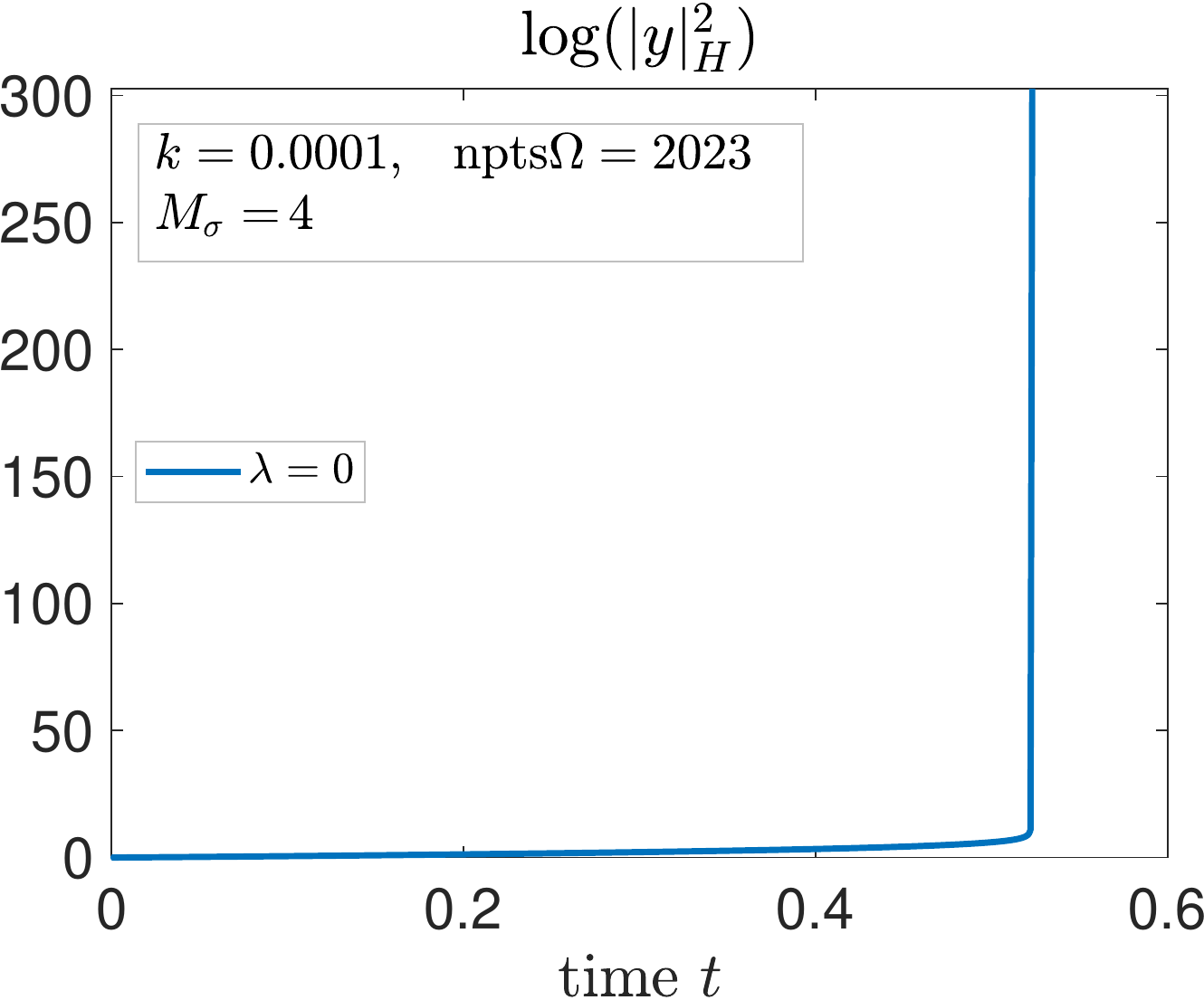}}
\qquad
\subfigure
{\includegraphics[width=0.45\textwidth,height=0.33\textwidth]{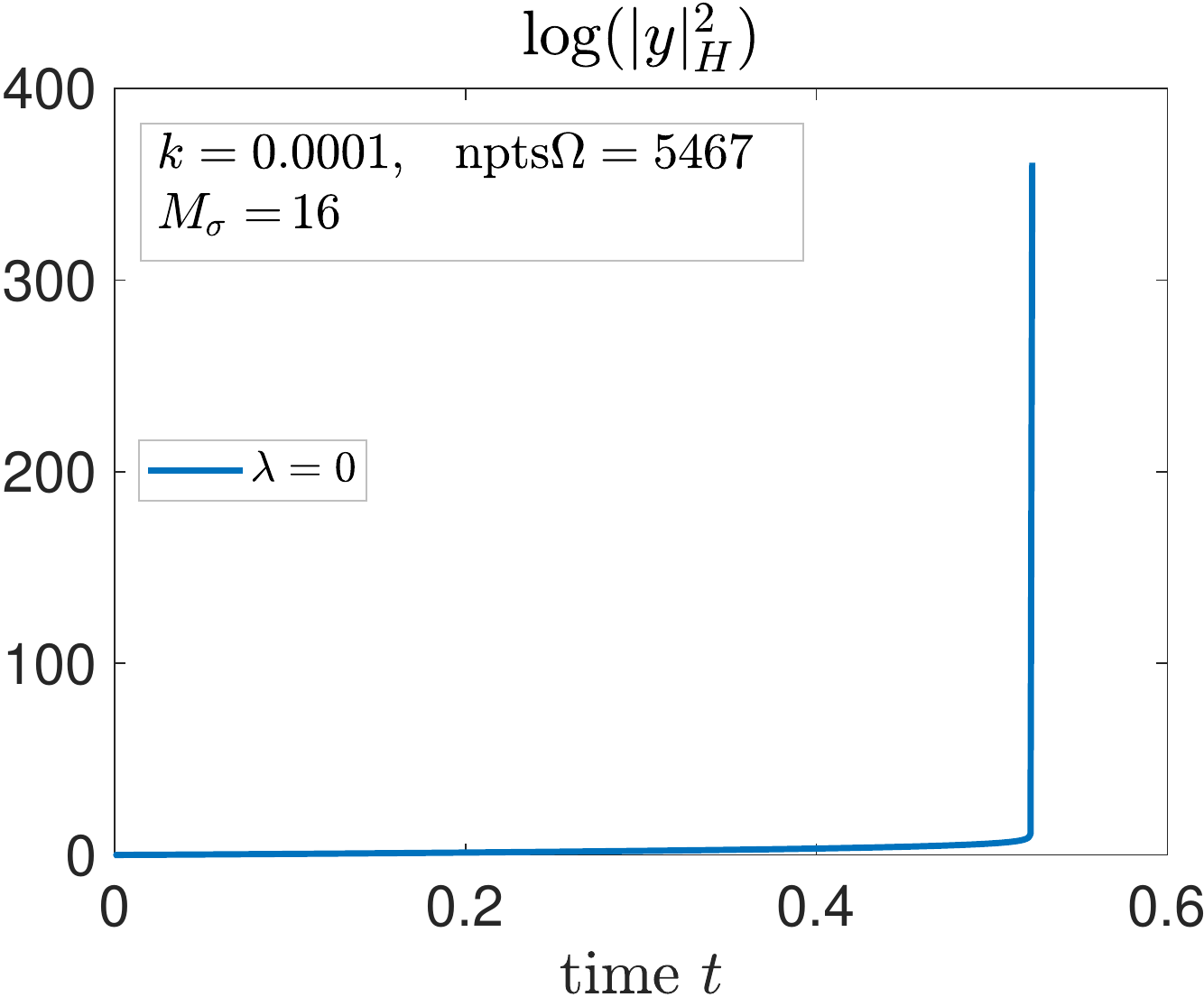}}
\caption{Free dynamics norm behavior.}
\label{Fig:y_4and16lam00Fon01nIC11FTYOrth}
\end{figure}

Now we fix the number of actuators~$M_\sigma=9$ and simulate~\eqref{eq:numeq} for
varying values of~$\lambda \geq 0$. The results are reported in
Figure~\ref{Fig:y_M9lam50to175Fon03nIC11FTYOrth}. Similar to the uncontrolled system,
we observe a finite time blow-up of the solution norm for~$\lambda=50,75,100$.
Note that the blow-up time increases with~$\lambda$.
Finally, for~$\lambda=125,150,175$, the system~\eqref{eq:numeq} is exponentially
stable and the norm of~$y$ is strictly decreasing. However we point out that the
stabilization rate~$\mu$ does not improve from~$\lambda=150$ to~$\lambda=175$.
This backs up the claim of Corollary \ref{C:main-expli-1}, which
states that both~$\lambda>0$~\emph{and}~$M\in \mathbb{N}$ have to be chosen
large enough in order to achieve a certain rate of stabilization.
\begin{figure}[ht]
\centering
{\includegraphics[width=0.45\textwidth,height=0.33\textwidth]{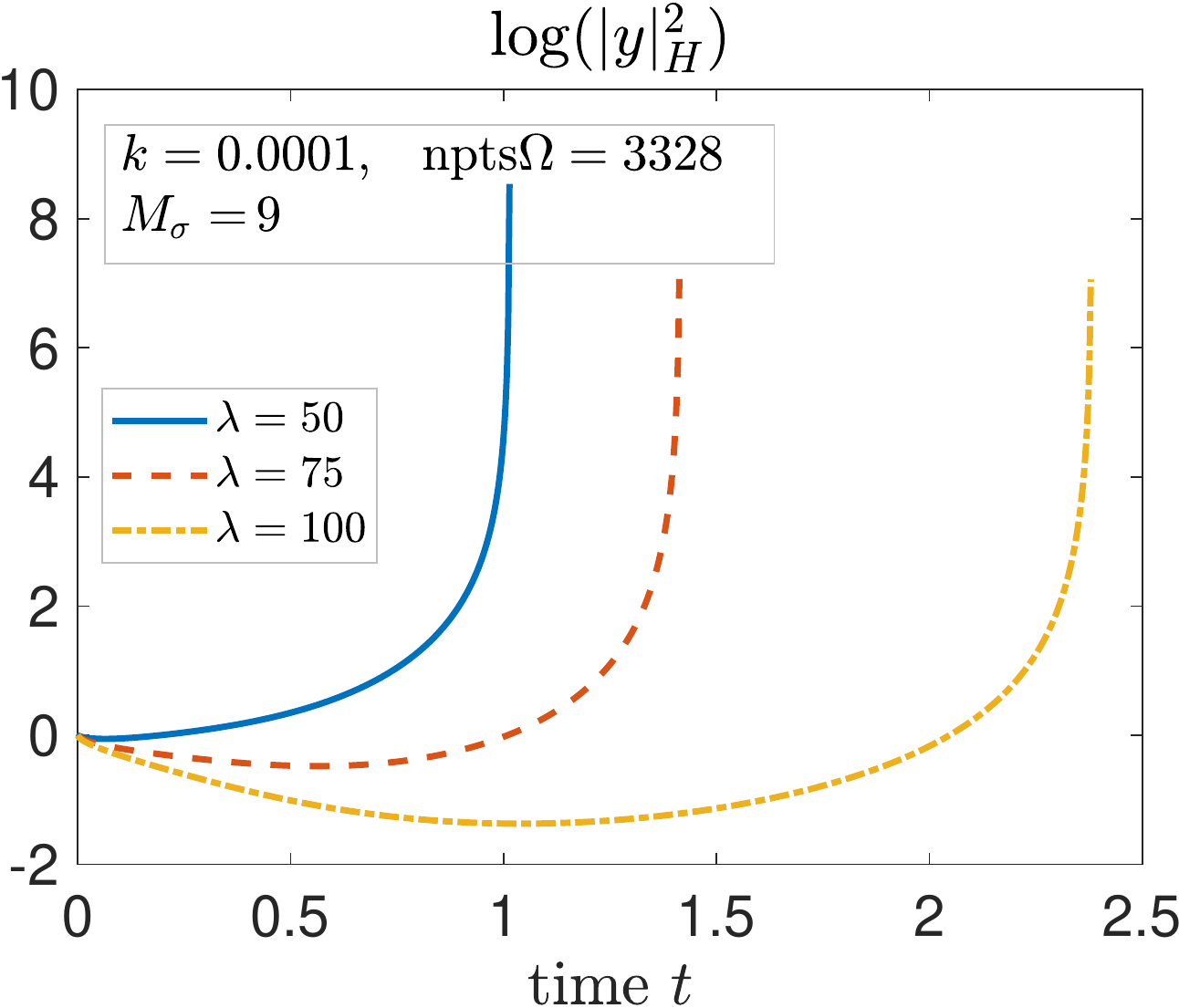}}
\qquad
\subfigure
{\includegraphics[width=0.45\textwidth,height=0.33\textwidth]{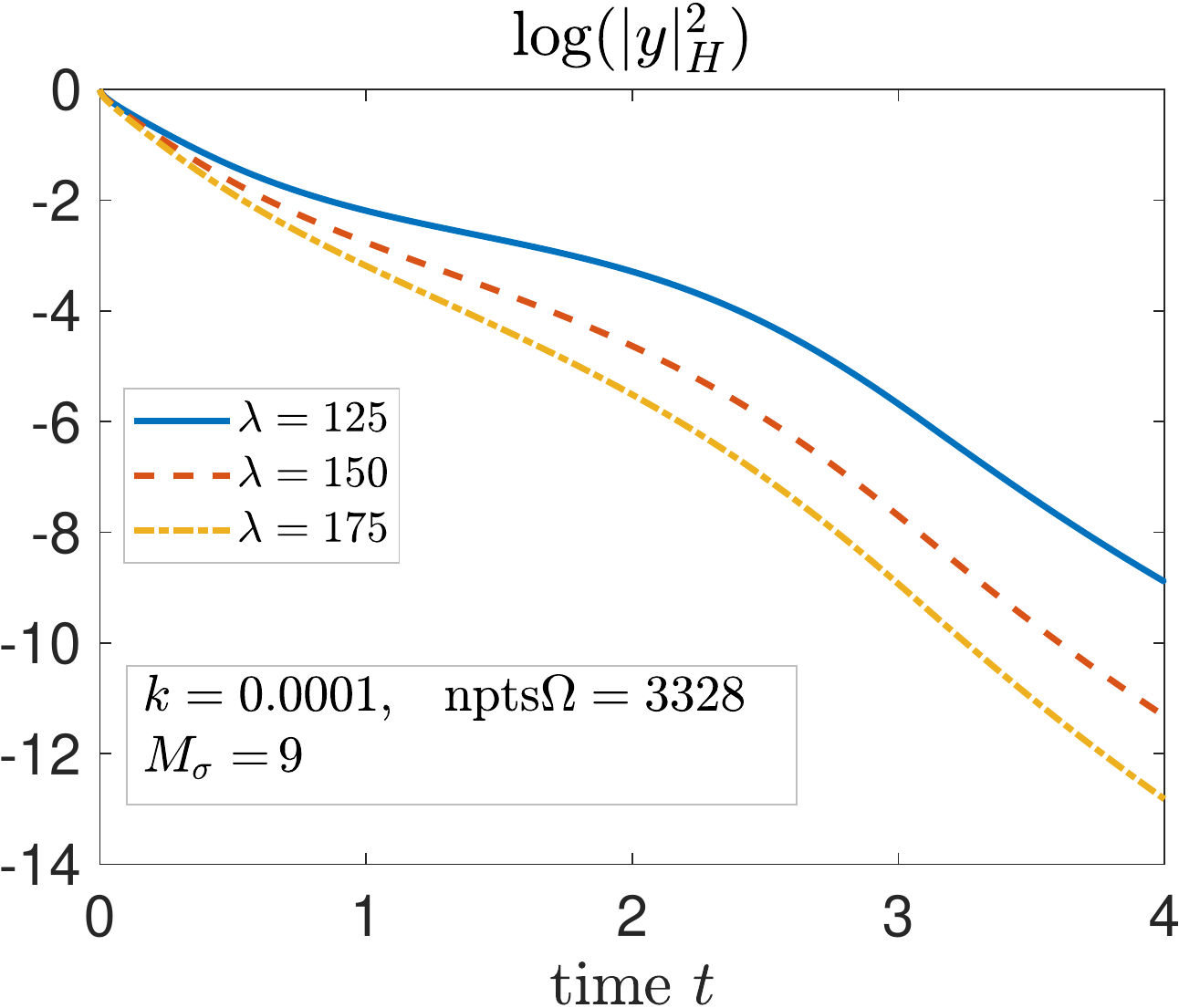}}
\caption{The case of~$9$ actuators.}
\label{Fig:y_M9lam50to175Fon03nIC11FTYOrth}
\end{figure}

This becomes even more evident if we simulate~\eqref{eq:numeq} for the same values of~$\lambda$ but
with~less/more actuators. In Figure~\ref{Fig:y_M4and16lam50to175Fon01-4nIC11FTYOrth} we\black see that~$M_\sigma=4$
actuators are not able to stabilize
the system, even for large~$\lambda$. Moreover, the respective blow up times of for each~$\lambda$ seem
to convergence towards a limit point around~$0.7$ as~$\lambda$ increases. This behavior suggests
that~$4$ actuators are not able to stabilize the system independently of the value of~$\lambda$.
On the contrary with~$M_\sigma=16$ actuators the system is stabilized for all considered values of~$\lambda$.
In particular note the improved rate of stabilization in comparison to the case of~$M_\sigma=9$ in Figure~\ref{Fig:y_M9lam50to175Fon03nIC11FTYOrth}.
\begin{figure}[ht]
\centering
\subfigure
{\includegraphics[width=0.45\textwidth,height=0.33\textwidth]{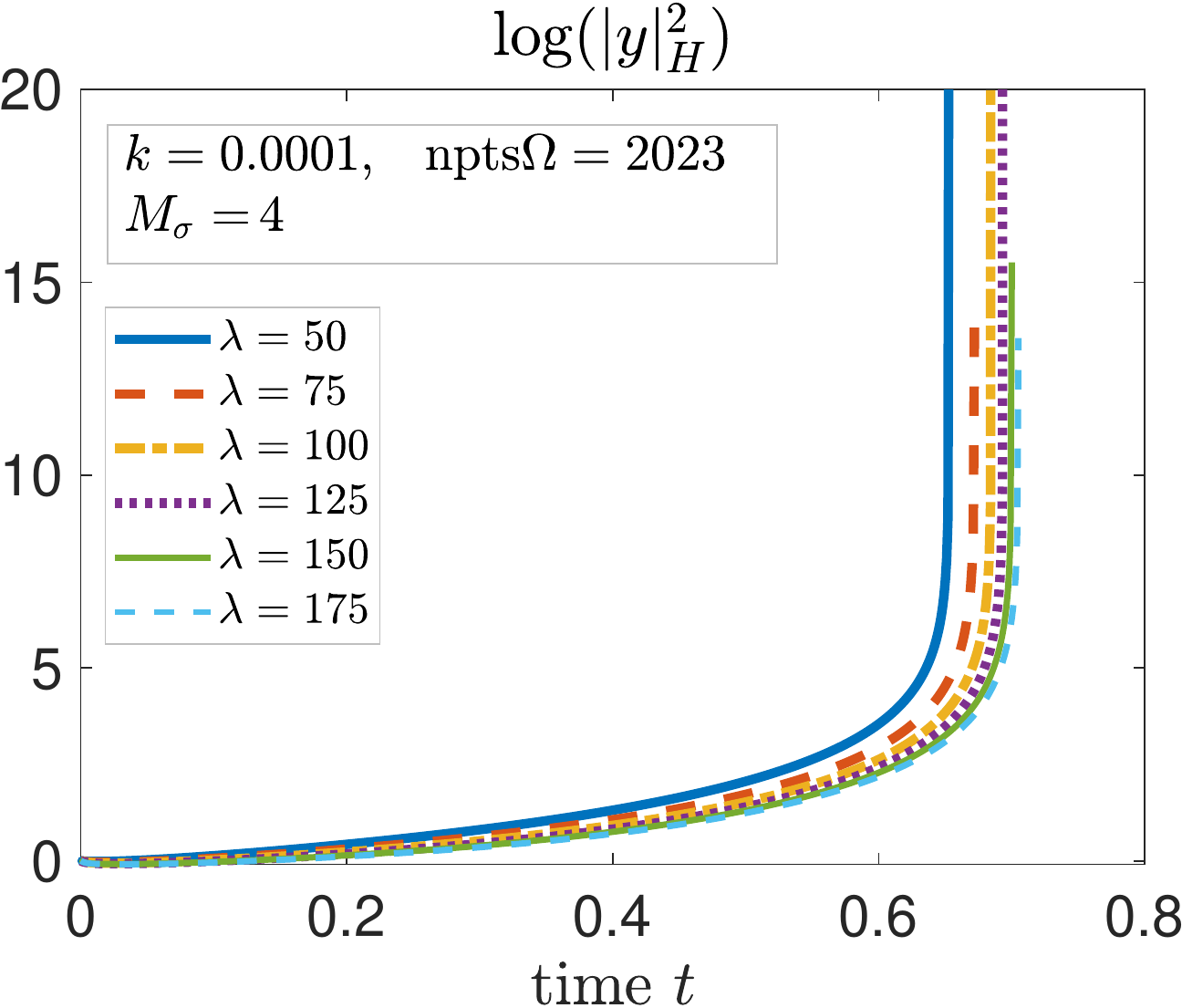}}
\qquad
\subfigure
{\includegraphics[width=0.45\textwidth,height=0.33\textwidth]{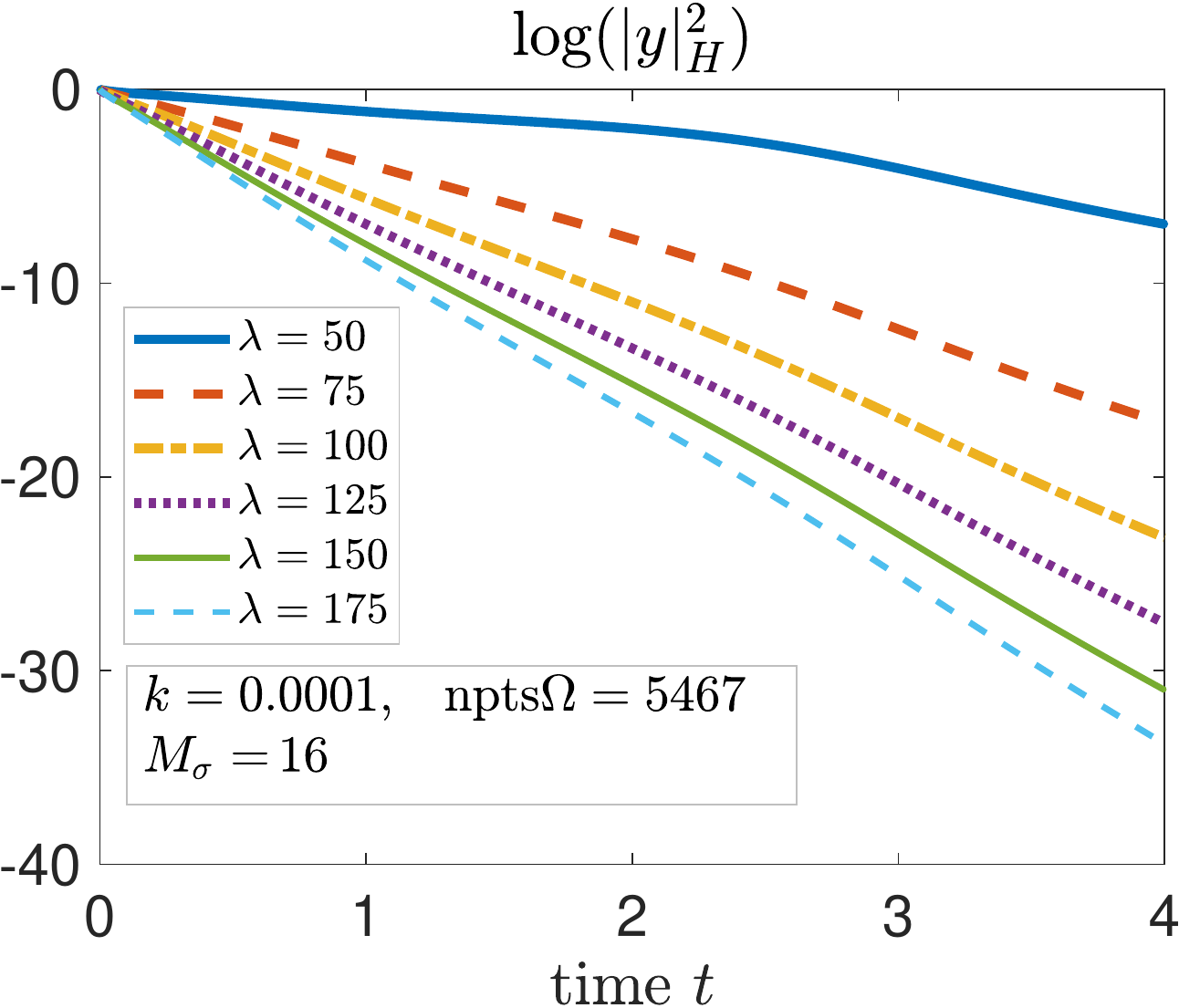}}
\caption{The case of~$4$ and~$16$ actuators.}
\label{Fig:y_M4and16lam50to175Fon01-4nIC11FTYOrth}
\end{figure}

Finally we repeat the previous simulations with~$M_\sigma=16$~actuators for the rescaled
initial condition~$y_0= 2 \bar{y}_0$.\black The results can be found in Figure
\ref{Fig:y_M16lam50to175Fon04nIC11FTYOrth}. After increasing  the norm of the initial
condition we note that~$16$ actuators are no longer able to stabilize the system
for~$\lambda=50,75$. For the remaining larger values of~$\lambda$, the  system is still
exponentially stable, however the rate of stabilization is smaller than in
Figure~\ref{Fig:y_M4and16lam50to175Fon01-4nIC11FTYOrth}.
\begin{figure}[ht]
\centering
\subfigure
{\includegraphics[width=0.45\textwidth,height=0.33\textwidth]{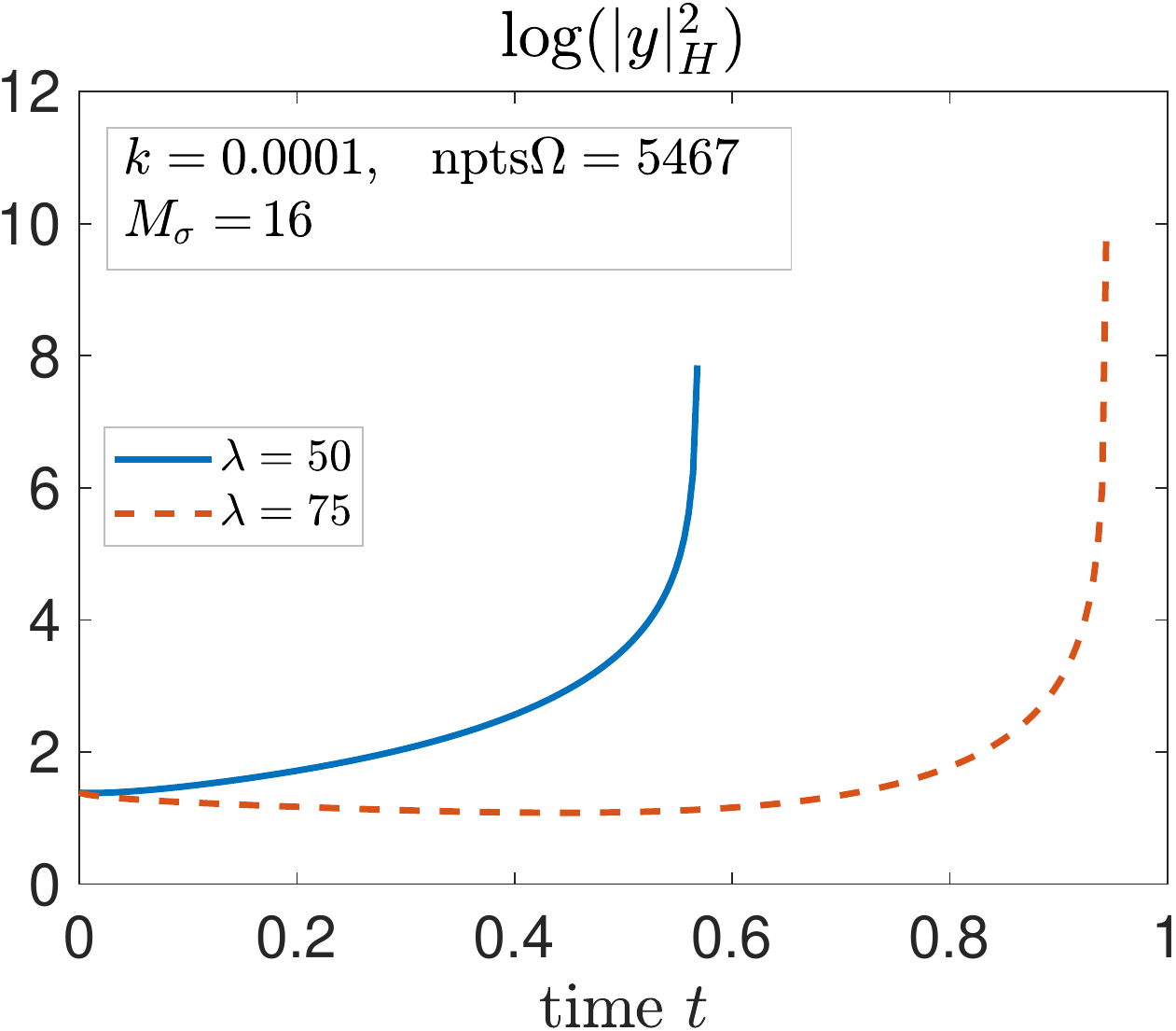}}
\qquad
\subfigure
{\includegraphics[width=0.45\textwidth,height=0.33\textwidth]{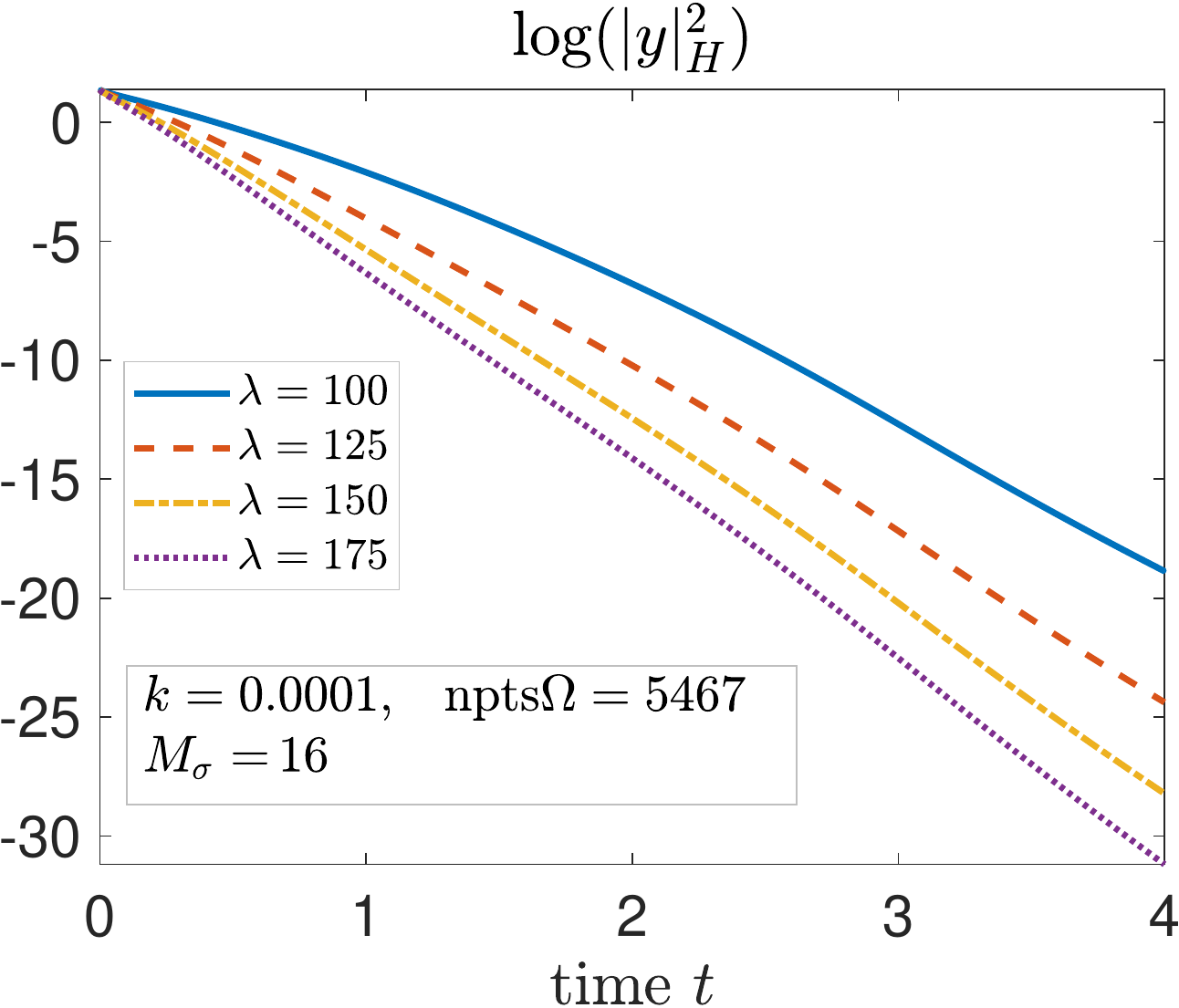}}
\caption{Larger initial condition, $y_0=2\overline y_0$.}
\label{Fig:y_M16lam50to175Fon04nIC11FTYOrth}
\end{figure}

In conclusion, Figures~\ref{Fig:y_4and16lam00Fon01nIC11FTYOrth}--\ref{Fig:y_M16lam50to175Fon04nIC11FTYOrth} backup the findings of Theorem~\ref{T:main}. Indeed,
they confirm the stabilizing property of the orthogonal projection feedback~$\mathcal{K}^{\lambda}_M(P_{\mathcal{U}_M}y)=-\lambda P_{\mathcal{U}_M}y $
provided that both, the number of actuators~$M_\sigma=\dim\clU_M$
and the monotonicity parameter~$\lambda>0$ are chosen large enough. Moreover Figure~\ref{Fig:y_M16lam50to175Fon04nIC11FTYOrth} shows that for larger initial conditions we (may) also need to
increase~$M_\sigma=\dim\clU_M$~and/or~$\lambda>0$ which highlights the semiglobal nature of the considered feedback.

As a last remark to this first part, note that we do not report here on simulations
for the oblique projection feedback (c.f. Corollary~\ref{C:main-expli-1}), since
its qualitative behavior parallels that of the orthogonal
projection. For simulations concerning the linear case we refer the reader
to~\cite{Rod_pp20-CL}.

\subsection{Learning based feedback stabilization}
Next we compare the orthogonal projection
feedback~$\mathcal{K}^\lambda_{M}(\mathcal{P}_{\mathcal{U}_M} \Bigcdot)$ for
fixed~$\lambda=100$ and~$M_\sigma=9$ with neural network feedback
laws~$\mathcal{K}^{(\theta)}_{M}(\mathcal{P}_{\mathcal{U}_M} \Bigcdot)$
obtained from solving~\eqref{def:discproblem} with~$\bar{\lambda}=100$.
In this section the parameters in system~\eqref{eq:numeq} are chosen to be
\begin{align*}
 a\coloneqq-2+x_1 -\norm{\sin(x_1)}{\bbR},&\qquad b\coloneqq\begin{bmatrix}
                                                           x_1+x_2\\
                                                           x_1x_2
                                                          \end{bmatrix}, \qquad \nu=0.1,
\end{align*}
and the initial state~$y_0$ will be specified below. Moreover we denote by~$\Delta_h$,~$\nabla_h$
and~$\mathcal{N}_h$ the finite element approximations of the diffusion and reaction
operators as well as of the nonlinearity in~\eqref{eq:numeq}.
For the definitions and notation concerning neural networks used in the following
we refer to Example~\ref{examplsubsets}. We choose feedback
laws~$\mathcal{K}^{(\theta)}_M$ induced by
realizations~$K^{(\theta)}_M$, c.f.~\eqref{def:realization}, of neural
networks~$\theta$ with~$L=3$ layers and
architecture~$\operatorname{arch}(\mathcal{R})=(9,15,15,9)$.
The activation function is chosen as the softmax~$\chi(x)=\log(1+\exp(x))$.
Moreover we fix the finite time horizon~$T=5$, the cost parameter~$\beta=1$, the timestep~$k=0.001$
and~$\mathcal{G}(\theta)=0$. We point out that~$T$,~$L$
and~$\operatorname{arch}(\mathcal{R})$ are chosen based on numerical experience.
A systematic approach on how to choose the best neural network architecture
for the stabilization problem at hand is beyond of the scope of this manuscript.
Finally we address the choice of the ``training set" in~\eqref{def:discproblem}. For this purpose denote
by~$\lambda_j \in \mathbb{R}$ the eigenvalues of~$(-\nu\Delta_h+\mathbf{1})$
ordered by increasing magnitude and let~$v_j$ be the associated eigenfunctions.
Now define the initial conditions~$y^i_0$,~$i=1,\dots,N_0$,~$N_0=10$, in the training
set by~$y^j_0=v_j$ and~$y^{j+5}_0=-v_j$,~$j=1,\dots,5$. The learning
grid~$\{u_i\}^{N_1}_{i=1}$,~$N_1=50000$, is obtained by uniformly sampling from the set
\begin{align*}
\mathbf{U}_M:= \left \{\,u \in L^2(\Omega) \;|\;u \in  \mathcal{U}_M,~|u|_H\leq 1.1\,\right\}.
\end{align*}
The parameter~$\varepsilon_1$ in the definition of~$\mathcal{G}^{N_1}_\gamma$ is chosen as~$\varepsilon_1=10^{-4}$.
A gradient descent method is applied to all arising learning problems, that is, starting from~$\theta_0 \in \mathcal{R}$ we iterate~$\theta_{k+1}=\theta_k-s_k \, \partial \mathcal{J}(\theta_k)$ for some stepsize~$s_k \geq 0$.

\subsubsection{Computing neural network feedback laws}
We now determine two neural network feedbacks, one for~$\gamma=0$ (i.e.,
 monotonicity is not enforced) and one for~$\gamma=500$.  We stress that,
since the uncontrolled systems might blow-up in finite time, finding an initial
neural network~$\theta_0 \in \mathcal{R}$ is not straightforward. As a remedy we consider
the damped closed loop systems
\begin{align}\label{eq:damped}
&\hspace*{-.5em}\tfrac{\p}{\p t} y_i +(-0.1\Delta_h+\Id) y_i+ay_i +b\cdot\nabla_h y_i - \alpha\,\norm{y_i}{\bbR}y_i
 =\mathcal{K}^{(\theta)}_M ( P_{\clU_{M}} y_i),\quad\! \tfrac{\p}{\p\bfn}y_i\rest{1\Gamma}=0,\quad\!
y(0)=y^i_0,
\end{align}
for~$\alpha \in[0,1]$,~$i=1,\dots,10$. For~$\alpha=0$, the uncontrolled systems are
linear and  thus admit a solution on~$[0,T]$. In virtue of the implicit function
theorem, the same will be true for feedback laws~$\mathcal{K}^{(\theta)}_M$ with small
weights~$\theta$. Now we can solve~\eqref{def:discproblem} subject to~\eqref{eq:damped}
and~$\alpha=0$.\black
Then  path-following with respect to  the damping parameter is carried out, by
increasing~$\alpha$
and solving\eqref{def:discproblem} subjected to~\eqref{eq:damped}\black using the previous
solution as a starting point. This procedure is repeated until we arrive
at~$\alpha=1$.\black
A similar homotopy strategy is applied for the penalty parameter,
from~$\gamma=0$ to~$\gamma=500$.\black
\begin{remark}
We observe  that~$\mathcal{K}^{(\theta_0)}_M(P_{\mathcal{U}_M} \Bigcdot)
=-\lambda P_{\mathcal{U}_M}$
for~$\theta_0=(0,0,0,\cdots,0,-\lambda \Id,0)$.
This observation might suggest to use the orthogonal projection feedback as a
starting point for the learning problem. However, due to the definition
of~$K^{(\theta)}_M$ as a composition, we readily
verify that~$\partial_{W_{1,j}} J(\theta_0)=0$,~$\partial_{b_j} J(\theta_0)=0$,~$j
=1,\dots,L$, as well as~$\partial_{W_{2,j}} J(\theta_0)=0$,~$j=1,\dots,L-1$.
Hence we have~$\theta_k=(0,0,0,\cdots,0,W^k_{2,L},0)$ for some~$W^k_{2,L}\in \mathbb{R}^{M_{\sigma}\times M_\sigma}$, that is, all realizations would result in
linear mappings. Therefore we do not use this particular choice
of~$\theta_0$ as initialization for the algorithm.
\end{remark}

\subsubsection{Comparison for explicit initial conditions} \label{subsubsec:explicit}
To compare the neural network and orthogonal projection feedback laws we simulate the associated closed loop system for two initial conditions given by
\begin{align*}
\bar{y}^1_0\coloneqq \frac{1-2x_1x_2}{\norm{1-2x_1x_2}{H}},\qquad \bar{y}^2_0\coloneqq \frac{-\operatorname{sgn}(x_1-0.5)\,\operatorname{sgn}(x_2-0.5)}{\norm{\operatorname{sgn}(x_1-0.5)\,\operatorname{sgn}(x_2-0.5)}{H}}.
\end{align*}
Note that these initial conditions are neither contained in the training set
nor in its linear span.
The temporal evolution of the $H$-norm of the computed states and feedback
controls are plotted in Figures~\ref{Fig:NNfirstinitial} and~\ref{Fig:NNsecondinitial}.
Note the  logarithmic scale, which is used for the vertical axis, and the additional
zoomed plots showing the short-time behavior of the feedback controls.

\begin{figure}[ht]
\centering
\subfigure
{\includegraphics[width=0.48\textwidth,height=0.4\textwidth]{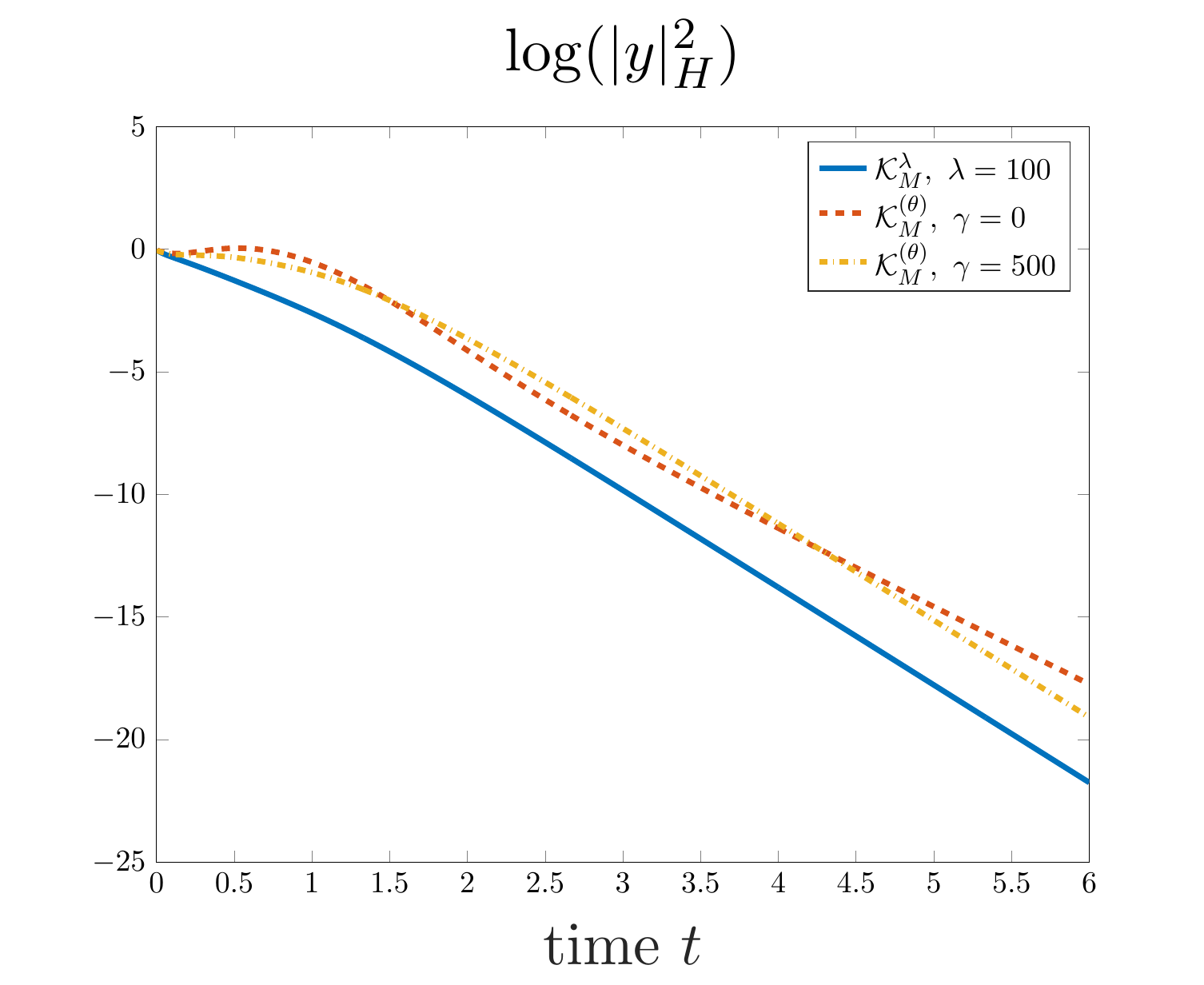}}
\subfigure
{\includegraphics[width=0.48\textwidth,height=0.4\textwidth]{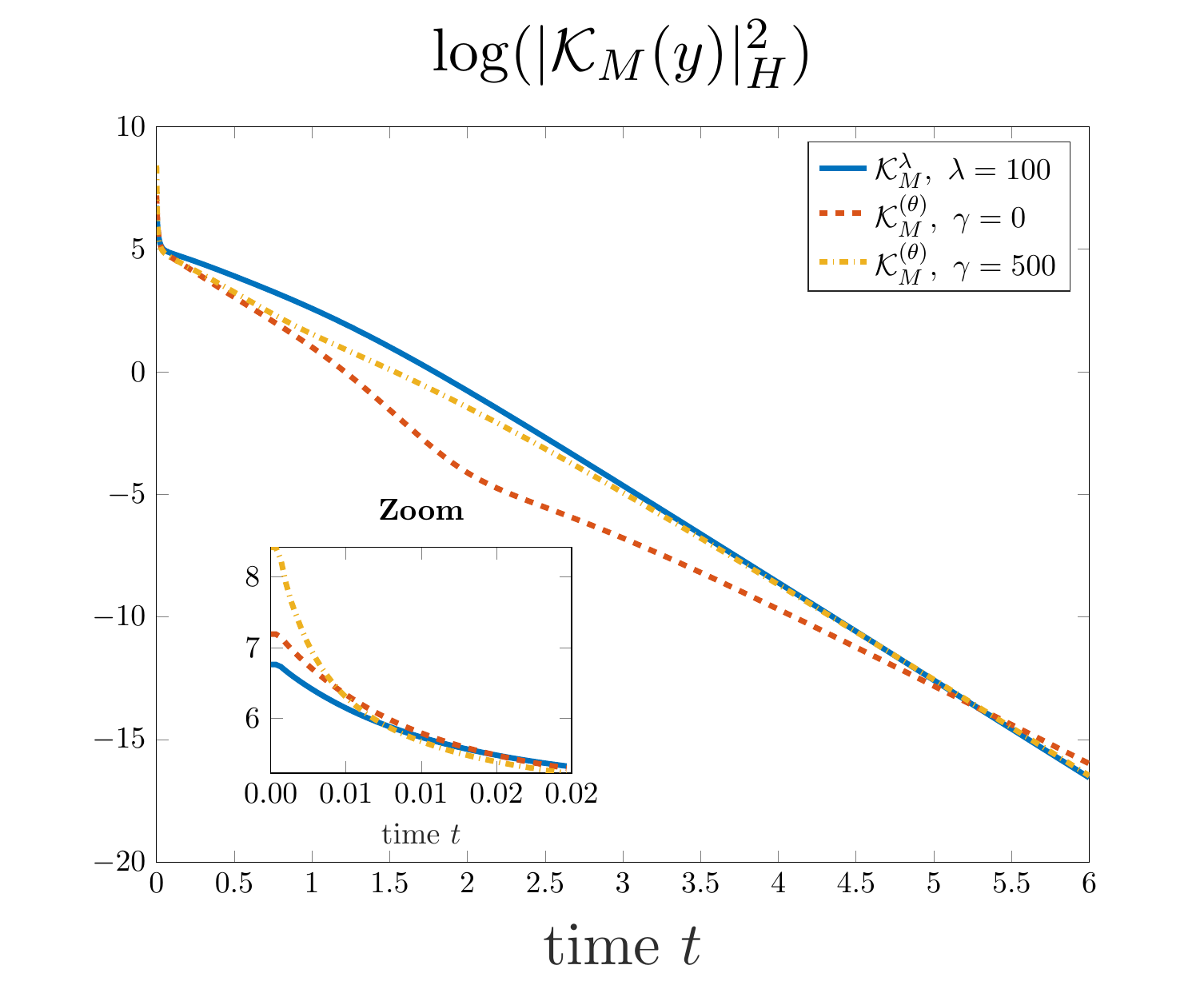}}
\caption{Evolution of state and control norms for~$y_0=\bar{y}^1_0$.}
\label{Fig:NNfirstinitial}
\end{figure}
\begin{figure}[ht]
\centering
\subfigure
{\includegraphics[width=0.41\textwidth,height=0.4\textwidth]{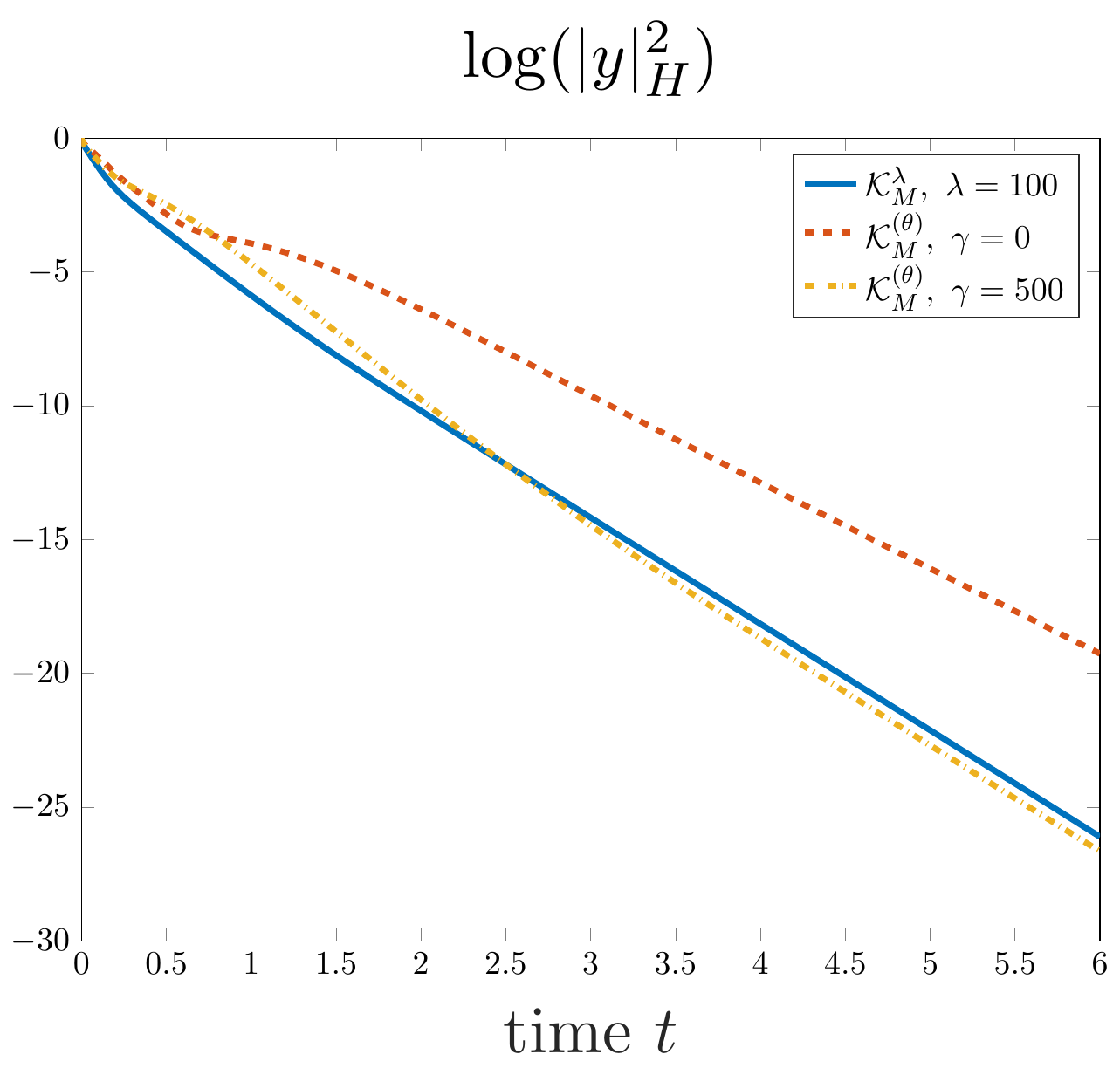}}
\qquad\quad
\subfigure
{\includegraphics[width=0.41\textwidth,height=0.4\textwidth]{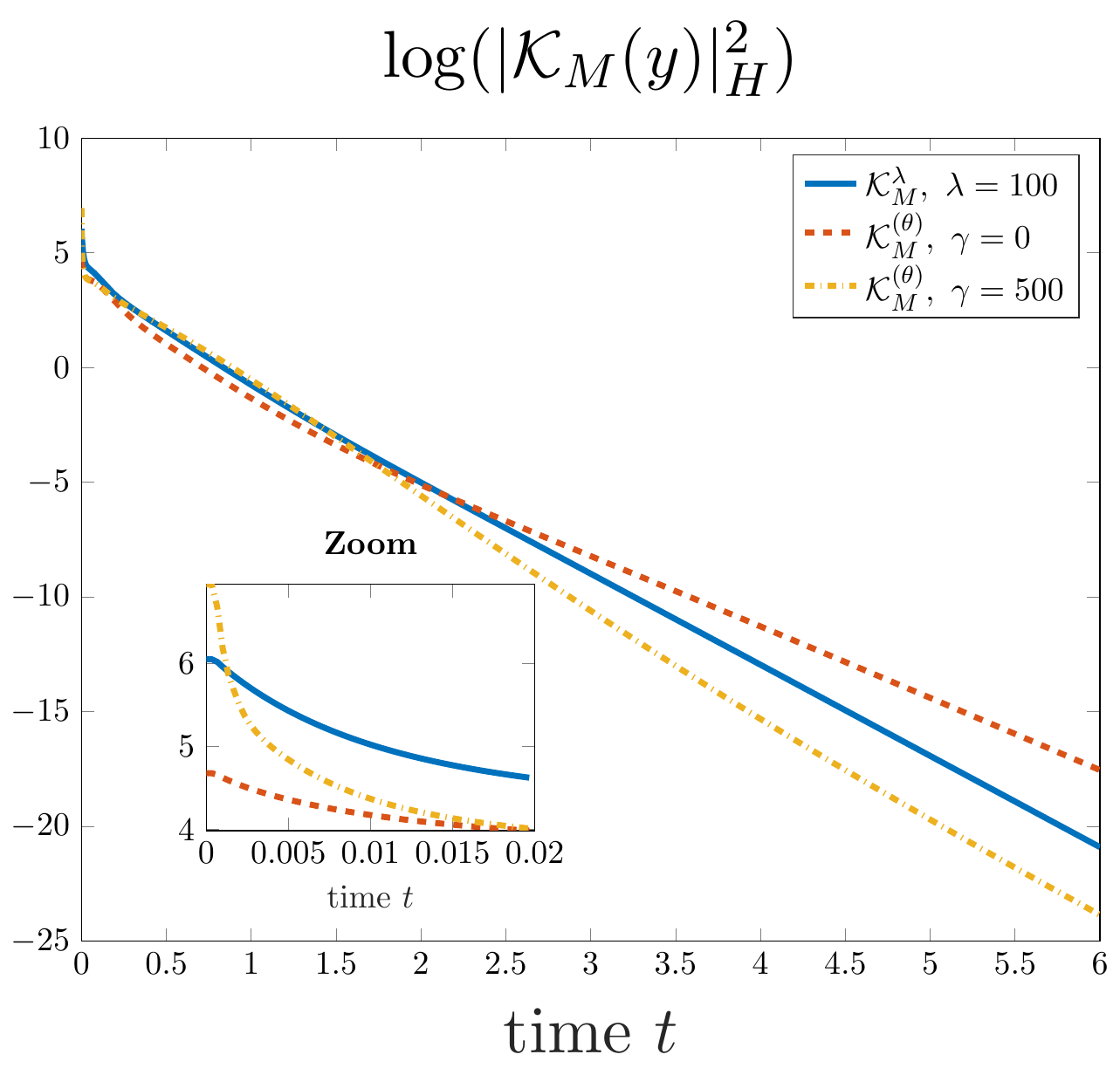}}
\caption{Evolution of state and control norms for~$\bar{y}^2_0$.}
\label{Fig:NNsecondinitial}
\end{figure}

\begin{table}[!htb]
\centering
\begin{minipage}{.48\linewidth}
    \centering

    \medskip
\begin{tabular}{  l c c c c  }

\multicolumn{5}{c}{Initial condition $\bar{y}^1_0$}\\
\toprule
$\qquad \mathcal{K}_M$ & &  $|y|^2_{L^2}$ & $|u|^2_{L^2}$ &  $J(y,u)$  \\
\midrule
$\mathcal{K}^\lambda_M$,~$\lambda=100$ & & 0.38 & 69.86 & 35.12\\
$\mathcal{K}^{(\theta)}_M$,~$\gamma=0$ & & 1.08 &  44.86   & 22.97 \\
$\mathcal{K}^{(\theta)}_M$,~$\gamma=500$  & & 0.84 &  50.3 &  25.57\\
\bottomrule
\end{tabular}
\end{minipage}\hfill
\begin{minipage}{.48\linewidth}
    \centering

    \medskip
\begin{tabular}{ l c c c c }

\multicolumn{5}{c}{Initial condition $\bar{y}^2_0$}\\
\toprule
$\qquad \mathcal{K}_M$ & & $|y|^2_{L^2}$ & $|u|^2_{L^2}$ &  $J(y,u)$  \\
\midrule
$\mathcal{K}^\lambda_M$,~$\lambda=100$ & &0.11 & 17.06 & 8.57\\
$\mathcal{K}^{(\theta)}_M$,~$\gamma=0$ & & 0.18  &  10.34   &  5.26\\
$\mathcal{K}^{(\theta)}_M$,~$\gamma=500$ & & 0.16 & 13.08  &  6.62\\
\bottomrule
\end{tabular}
\end{minipage}\hfill

\

\caption{Results for~$ y_0\in\{\bar{y}^1_0,\bar{y}^2_0\}$.}
\label{tab:results}
\end{table}

In Table~\ref{tab:results} the values of objective functional as
well as the space-time norms of the state trajectories, and the controls associated
to the various feedback laws are summarized. In both examples all considered feedbacks
eventually
stabilize the system at (approximately) the same exponential rate. Additionally we
observe that, in each example, the long-term behavior of the feedback controls
is similar. However, the  results significantly differ in the early stages.
First we focus on the case~$y_0=\bar{y}^1_0$.\black As expected, for both,~the orthogonal projection
and the neural network feedback with~$\gamma=500$, the~$H$-norm of~$y$ is a
strictly decreasing function. Note the slightly worse rate of stabilization
for~$\mathcal{K}^{(\theta)}_M$,~$\gamma=500$, in the transient phase. In contrast,
the norm of the state associated to~$\mathcal{K}^{(\theta)}_M$,~$\gamma=0$,
increases on~$[0,1]$. These observations are also reflected in slightly larger
space-time norms of the state trajectories of the neural network controlled systems
 as we can see in Table~\ref{tab:results}, where~$L^2$ stands for~$L^2((0,T),L^2(\Omega))$.\black On the other hand we note that, apart from a short
early phase, the neural network feedback controls are smaller than the ones induced
by the orthogonal projection, resulting in significantly smaller control costs.
As a consequence, the neural network induced feedback
law~$\mathcal{K}^{(\theta)}_M$,~$\gamma=0$, admits the lowest overall objective
functional value among all considered feedback laws followed
by~$\mathcal{K}^{(\theta)}_M$,~$\gamma=500$, with an improvement
of~$35\%$ and~$27\%$, respectively, over the explicit feedback.

We also make the same qualitative and quantitative observations
for the case~$y_0=\bar{y}^2_0$.\black Indeed, the state and control norms for the
various feedbacks show comparable long-time behavior, but they differ in the
early stages. This leads to slightly increased space-time state norms
for the neural network feedback laws at the advantage of significantly
reduced control costs. Overall this results in~$39\%$, for $\gamma=0$, and~$23\%$,
for $\gamma=500$, smaller
objective functional values for the neural network feedback laws in comparison
to the scaled orthogonal projection based feedback law.

\subsubsection{Validation for a large set of initial conditions}
Finally we repeat the previous simulations for a larger number of randomly generated
initial conditions to highlight the semiglobal nature of the considered feedback laws.
For this purpose we uniformly sample functions~$y^j_0$,~$j=1,\dots,100$, from the
validation set
\begin{align*}
\mathbf{Y}_m:= \left \{\,y_0 \in H\;|\;y_0 \in \operatorname{span}\{v_j\}^m_{j=1},~|y_0|_H \leq 1\,\right\}
\end{align*}
where~$m \in \mathbb{N}$ and~$v_j$,~$j=1,\dots,m $, denote the eigenfunctions from
Section~\ref{subsubsec:explicit}. Note that for~$m\leq 5$,~$\mathbf{Y}_m$ is a subset
of the linear span of the training set. For each considered feedback
law~$\mathcal{K}_M$ and initial condition~$y^i_0$ we then compute the associated
states~$y_i$ and feedback controls~$u_i \coloneqq \mathcal{K}_M(P_{\mathcal{U}_M}y_i)$.
The indices of all initial conditions which are successfully stabilized are collected
in the set
\begin{align*}
\mathcal{I}(\mathcal{K}_M):= \left\{\,i \in \{1,\dots,100\}\;|\;J(y_i,u_i)<\infty\,\right\}.
\end{align*}
Moreover we define the numbers of failed and successful stabilizations
by~$\text{N}_{\text{succ}}\coloneqq \# \mathcal{I}(\mathcal{K}_M)$
and~$\text{N}_{\text{fail}} \coloneqq 100-\text{N}_{\text{succ}}$, respectively,
as well as the averaged state/control norms and objective functional values:
\begin{align*}
\mathbb{E}\left\lbrack|y|^2_{L^2} \right\rbrack\coloneqq\frac{1}{\text{N}_{\text{succ}}}\sum_{i\in \mathcal{I}(\mathcal{K}_M)} |y_i|^2_{L^2},~\mathbb{E}\left\lbrack|u|^2_{L^2} \right\rbrack\coloneqq \frac{1}{\text{N}_{\text{succ}}}\sum_{i\in \mathcal{I}(\mathcal{K}_M)} |u_i|^2_{L^2},~\mathbb{E}\left\lbrack J(y,u) \right\rbrack\coloneqq \frac{1}{\text{N}_{\text{succ}}}\sum_{i\in \mathcal{I}(\mathcal{K}_M)} J(y_i,u_i).
\end{align*}
Further $\text{N}_{\text{improv}}$ is the number of initial conditions for which
$\mathcal{K}^\theta_M$ successfully stabilizes and the value of the objective functional
is small than for the $\mathcal{K}^\lambda_M$ feedback law.
In order to assess the performance of the various neural network feedback laws
in comparison to the orthogonal projection we also compute the average as well
as the best and worst improvement of the objective functional values
\begin{align*}
&\text{Avg. Improv.} \coloneqq \left( \sum_{i\in \mathcal{I}(\mathcal{K}_M)} \left( J(y_i,u_i)-J(y^\lambda_i,u^\lambda_i) \right) \right) / \sum_{i\in \mathcal{I}(\mathcal{K}_M)} J(y^\lambda_i,u^\lambda_i), \\
\text{Best} \coloneqq - &\max_{i \in \mathcal{I}(\mathcal{K}_M)} \left\{ \frac{J(y^\lambda_i,u^\lambda_i)-J(y_i,u_i)}{J(y^\lambda_i,u^\lambda_i)}\right\} \ ,~\text{Worst} \coloneqq -\min_{i\in \mathcal{I}(\mathcal{K}_M)} \left\{ \frac{J(y^\lambda_i,u^\lambda_i)-J(y_i,u_i)}{J(y^\lambda_i,u^\lambda_i)}\right\}
\end{align*}
where~$y^\lambda_i$ and~$u^\lambda_i$ denote the state and feedback controls
associated to~$\mathcal{K}^\lambda_M$.
The computed results for~$m=5,7,10$ can be found in Table~\ref{tab:val1}
and~\ref{tab:val2}, respectively.

\begin{table}[!htb]
\centering
    \medskip
\begin{tabular}{  l c c c c c c c c c  }

\toprule
$\qquad \mathcal{K}_M$ & & $\text{N}_{\text{fail}}$ & $\text{N}_{\text{improv}}$ &  $\mathbb{E}\left\lbrack|y|^2_{L^2} \right\rbrack$ & $\mathbb{E}\left\lbrack|u|^2_{L^2}\right \rbrack$ &  $ \mathbb{E}\left\lbrack J(y,u)\right \rbrack$ & Avg. Improv. & Best & Worst \\
\midrule
$\mathcal{K}^\lambda_M$,~$\lambda=100$ &   & 0 & 0 & 0.07 & 12.6 & 6.35 & $+0 \%$ & $+0\%$ & $+0\%$ \\
$\mathcal{K}^{(\theta)}_M$,~$\gamma=0$ &   & 2 & 98 & 0.2 & 6.3 & 3.24 & $-44\%$ & $-96\%$ & $-6\%$ \\
$\mathcal{K}^{(\theta)}_M$,~$\gamma=500$ &   & 0 & 82 & 0.1 & 10.1 & 5.12 & $-19\%$ & $-86\%$ & $+46\%$ \\
\bottomrule
\end{tabular}

\

\caption{Validation results for~$m=5$.}
\label{tab:val1}
\end{table}

\begin{table}[!htb]
\centering
    \medskip
\begin{tabular}{  l c c c c c c c c c  }

\toprule
$\qquad \mathcal{K}_M$ & & $\text{N}_{\text{fail}}$ & $\text{N}_{\text{improv}}$ &  $\mathbb{E}\left\lbrack|y|^2_{L^2} \right\rbrack$ & $\mathbb{E}\left\lbrack|u|^2_{L^2}\right \rbrack$ &  $ \mathbb{E}\left\lbrack J(y,u)\right \rbrack$ & Avg. Improv. & Best & Worst \\
\midrule
$\mathcal{K}^\lambda_M$,~$\lambda=100$ &   & 0 & 0 & 0.06 &  10.6& 5.34 & $+0\%$ & $+0\%$ & $+0\%$ \\
$\mathcal{K}^{(\theta)}_M$,~$\gamma=0$ &   & 2 & 98 & 0.15 & 4.6  & 2.37  & $-50\%$ & $-95\%$ & $-12\%$ \\
$\mathcal{K}^{(\theta)}_M$,~$\gamma=500$ &   & 0 & 53 & 0.08 & 10 & 5.04 & $-6\%$ & $-72\%$ & $+67\%$ \\
\bottomrule
\end{tabular}

\

\caption{Validation results for~$m=7$.}
\label{tab:val2}
\end{table}

\begin{table}[!htb]
\centering
    \medskip
\begin{tabular}{  l c c c c c c c c c  }

\toprule
$\qquad \mathcal{K}_M$ & & $\text{N}_{\text{fail}}$ & $\text{N}_{\text{improv}}$ &  $\mathbb{E}\left\lbrack|y|^2_{L^2} \right\rbrack$ & $\mathbb{E}\left\lbrack|u|^2_{L^2}\right \rbrack$ &  $ \mathbb{E}\left\lbrack J(y,u)\right \rbrack$ & Avg. Improv. & Best & Worst \\
\midrule
$\mathcal{K}^\lambda_M$,~$\lambda=100$ &   & 0 & 0 & 0.05 &  8.3& 4.18 & $+0\%$ & $+0\%$ & $+0\%$ \\
$\mathcal{K}^{(\theta)}_M$,~$\gamma=0$ &   & 1 & 99 & 0.11 & 3.7  & 1.9  & $-50\%$ & $-94\%$ & $-12\%$ \\
$\mathcal{K}^{(\theta)}_M$,~$\gamma=500$ &   & 0 & 16 & 0.06 & 10.3 & 5.12 & $+23\%$ & $-36\%$ & $+164\%$ \\
\bottomrule
\end{tabular}

\

\caption{Validation results for~$m=10$.}
\label{tab:val3}
\end{table}
As a starting point for the discussion of these results, note that the orthogonal
projection~$\mathcal{K}^\lambda_M$ and the monotone neural network
feedback~$\mathcal{K}^{(\theta)}_M$,~$\gamma=500$, stabilize all sampled initial
conditions. In contrast,~$\mathcal{K}^{(\theta)}_M$,~$\gamma=0$, which is
learned~\emph{without} the monotony enhancing penalty term fails in a small
percentage of cases. From a quantitative perspective, we make similar observations
as for the explicit initial conditions in the previous section. Let us first focus
on~$\mathcal{K}^{(\theta)}_M$,~$\gamma=0$. In all tests we
have~$\text{N}_{\text{succ}}=\text{N}_{\text{improv}}$.
Thus, if~$\mathcal{K}^{(\theta)}_M$ successfully stabilizes one of the sampled
initial condition, it admits a smaller objective functional value
than~$\mathcal{K}^\lambda_M$ with an average improvement of about~$50\%$.
As in Section~\ref{subsubsec:explicit}, this traces back to a significant
decrease in the norm of the feedback controls. Finally the  results indicate
that the improvement of the neural network feedback law for~$\gamma=0$ is
independent of~$m\in\mathbb{N}$. This suggests that while we are learning
the feedback on a relatively small set of~$10$ initial conditions it successfully
generalizes to functions outside of the training set without losing its stabilizing
properties and without a loss of performance.

The behavior of the neural network feedback learned \emph{with} the
monotony enhancing penalty term differs in certain key aspects. For example,
while~$\mathcal{K}^{(\theta)}_M$,~$\gamma=500$, stabilizes in all  considered
test cases, that is,~$\text{N}_{\text{succ}}=100$, it does not always improve over
the orthogonal projection and thus~$\text{N}_{\text{improv}}<100$. Additionally
its performance in comparison to~$\mathcal{K}^\lambda_M$ deteriorates for
growing~$m\in\mathbb{N}$. Indeed, while we observe smaller objective
function values for~$\text{N}_{\text{improv}}=86$ initial conditions with an
average improvement of~$-20\%$ over~$\mathcal{K}^\lambda_M$ for~$m=5$, this
decreases to~$\text{N}_{\text{improv}}=53$ at~$-6\%$ for~$m=7$. Finally,
for~$m=10$, the neural network feedback only improves in
$\text{N}_{\text{improv}}=16$ cases and, on average, performs worse than
the orthogonal projection feedback. To explain this loss of performance
we point out that the average state and controls norms as well as the
averaged objective functional values associated~to~$\mathcal{K}^{\lambda}_M$
become smaller as~$m \in \mathbb{N}$ grows. The same holds true
for~$\mathcal{K}^{(\theta)}_M$,~$\gamma=0$. In contrast,~the monotone
neural network feedback law fails to adapt this improved behavior on the
larger validation sets, that is, the average control cost and the objective
functional value is independent of~$m\in\mathbb{N}$. Thus it generalizes
worse than the other feedbacks to initial conditions outside of the training set.

In summary, on the one hand, the results of this section show the great potential
and success of learning feedback laws for the stabilization of unstable, nonlinear,
parabolic systems. On the other hand, they also point out certain limitations of
this approach which stimulate further research. For example, the performance
of the neural network feedback law learned~\emph{without} the penalty term
could be further improved by adaptively enlarging the training set with those
initial conditions from~$\mathbf{Y}_m$ for which~$\mathcal{K}^{(\theta)}_M$ fails.
Moreover, the discussion in the last paragraph suggests that the performance of
the monotone neural network feedback law is more susceptible to the choice of
the training set. Clearly, this raises the question on how to choose the initial
conditions for the training of the network in an optimal way as wells as on the
robustness of the computed results with respect to changes in the training set.

\appendix
\gdef\thesection{\Alph{section}}
\section*{Appendix}\normalsize
\setcounter{section}{1}
\setcounter{theorem}{0} \setcounter{equation}{0}
\numberwithin{equation}{section}

\subsection{Proof of Lemma~\ref{L:NN2}}\label{Apx:proofL:NN2}

Using Lemma~\ref{L:NN1} with~$(y_1,y_2)=(y,0)$, we find
\begin{align}
2(\clN(y),Ay)_H&\le\gamma_2 \norm{y}{\rmD(A)}^{2}
  +D_1
 \sum\limits_{j=1}^n \norm{y}{\rmD(A)}^\frac{2\zeta_{2j}}{1-\delta_{2j}}
\norm{y}{V}^{\frac{2\zeta_{1j}+2\delta_{1j}}{1-\delta_{2j}}},\quad\mbox{for all}\quad \gamma_2>0,\label{est-NNyAy1}
\end{align}
with~$D_1\coloneqq\left(1+\gamma_2^{-\frac{1+\|\delta_2\|}{1-\|\delta_2\|} }\right) \overline C_{\clN1}$. Next, we recall the young inequality in the form
\begin{align*}
ab\le \tfrac{1}{s}\eta^sa^s+\tfrac{s-1}{s}\eta^{-\frac{s}{s-1}}a^\frac{s}{s-1}, \quad\mbox{for all}\quad a\ge0,\;b\ge0,\;\eta>0,\;s>1,
\end{align*}
(cf.~\cite[Appendix~A.1]{Rod20}), which allow us to obtain, with
\[
s=s_j=\tfrac{1-\delta_{2j}}{\zeta_{2j}}>1, \qquad
\tfrac{s_j-1}{s_j}=\tfrac{1-\delta_{2j}-\zeta_{2j}}{1-\delta_{2j}},\qquad\mbox{in case}\quad \zeta_{2j}>0,
\]
the inequality
\begin{align*}
\norm{y}{\rmD(A)}^\frac{2\zeta_{2j}}{1-\delta_{2j}}
\norm{y}{V}^{\frac{2\zeta_{1j}+2\delta_{1j}}{1-\delta_{2j}}}
&\le \tfrac{\zeta_{2j}}{1-\delta_{2j}}\eta_j^\frac{1-\delta_{2j}}{\zeta_{2j}}\norm{y}{\rmD(A)}^2+\tfrac{1-\delta_{2j}-\zeta_{2j}}{1-\delta_{2j}}
 \eta_j^{-\frac{1-\delta_{2j}}{1-\delta_{2j}-\zeta_{2j}}}
\norm{y}{V}^{\frac{2\zeta_{1j}+2\delta_{1j}}{1-\delta_{2j}-\zeta_{2j}}}, \quad\mbox{for all}\quad \eta_j>0.
\end{align*}

Now, for any given~$\gamma_3>0$, we choose
\[
\eta_j= (\tfrac{1-\delta_{2j}}{\zeta_{2j}}\gamma_3)^\frac{\zeta_{2j}}{1-\delta_{2j}},\quad\mbox{in case}\quad \zeta_{2j}>0.
\]
which gives us
\begin{align*}
\norm{y}{\rmD(A)}^\frac{2\zeta_{2j}}{1-\delta_{2j}}
\norm{y}{V}^{\frac{2\zeta_{1j}+2\delta_{1j}}{1-\delta_{2j}}}
&\le \gamma_3\norm{y}{\rmD(A)}^2+\tfrac{1-\delta_{2j}-\zeta_{2j}}{1-\delta_{2j}}
  (\tfrac{1-\delta_{2j}}{\zeta_{2j}}\gamma_3)^{-\frac{\zeta_{2j}}{1-\delta_{2j}-\zeta_{2j}}}
\norm{y}{V}^{\frac{2\zeta_{1j}+2\delta_{1j}}{1-\delta_{2j}-\zeta_{2j}}}, \quad\mbox{in case}\quad \zeta_{2j}>0.
\end{align*}

Recalling~\eqref{est-NNyAy1}, it follows that
\begin{subequations}\label{est-NNyAy2}
\begin{align}
2(\clN(y),Ay)_H
&\le\gamma_2 \norm{y}{\rmD(A)}^{2}
+D_1 \sum\limits_{
\begin{subarray}{c}1\le j\le n\\ \zeta_{2j}=0\end{subarray}} \norm{y}{V}^{\frac{2\zeta_{1j}+2\delta_{1j}}{1-\delta_{2j}}}
  +D_1 \sum\limits_{
\begin{subarray}{c}1\le j\le n\\ \zeta_{2j}>0\end{subarray}}\left(\gamma_3\norm{y}{\rmD(A)}^2+
D_{2}\norm{y}{V}^{\frac{2\zeta_{1j}+2\delta_{1j}}{1-\delta_{2j}-\zeta_{2j}}}\right)
\notag\\
&\le(\gamma_2 +D_1\gamma_3 n)\norm{y}{\rmD(A)}^{2}
+D_1(1+D_2) \sum\limits_{j=1}^n
\norm{y}{V}^{\frac{2\zeta_{1j}+2\delta_{1j}}{1-\delta_{2j}-\zeta_{2j}}},
\intertext{for all~$\gamma_2>0$, $\gamma_3>0$, with}
D_2&=\max_{
\begin{subarray}{c}1\le j\le n\\ \zeta_{2j}\ne0\end{subarray}}
\tfrac{1-\delta_{2j}-\zeta_{2j}}{1-\delta_{2j}}
  (\tfrac{1-\delta_{2j}}{\zeta_{2j}}\gamma_3)^{-\frac{\zeta_{2j}}{1-\delta_{2j}-\zeta_{2j}}}.
\end{align}
\end{subequations}

Now for an arbitrary  given~$\gamma_4>0$, we may set~$\gamma_2=\frac{\gamma_4}2$ and~$\gamma_3=\frac{\gamma_4}{2nD_1}$, leading us to
\begin{align}
2(\clN(y),Ay)_H
\le \gamma_4 \norm{y}{\rmD(A)}^{2}
+D_1(1+D_2) \sum\limits_{j=1}^n
\norm{y}{V}^{\frac{2\zeta_{1j}+2\delta_{1j}}{1-\delta_{2j}-\zeta_{2j}}},
\quad\mbox{for all}\quad \gamma_4>0.
\end{align}

Observe that
\begin{align*}
\tfrac{2\zeta_{1j}+2\delta_{1j}}{1-\delta_{2j}-\zeta_{2j}}\ge2
&\quad\Longleftrightarrow\quad
\zeta_{1j}+\delta_{1j}\ge 1-\delta_{2j}-\zeta_{2j}
\quad\Longleftrightarrow\quad
\delta_{2j}+\delta_{1j}-1\ge -\zeta_{1j}-\zeta_{2j},
\end{align*}
hence, from Assumption~\ref{A:NN}, we have that
\[
p_j\coloneqq 2-\tfrac{2\zeta_{1j}+2\delta_{1j}}{1-\delta_{2j}-\zeta_{2j}}\ge0
\]
and can write
\begin{align*}
2(\clN(y),Ay)_H
&\le \gamma_4 \norm{y}{\rmD(A)}^{2}
+D_1(1+D_2) \norm{y}{V}^2\sum\limits_{j=1}^n
\norm{y}{V}^{p_j}\\
&\le \gamma_4 \norm{y}{\rmD(A)}^{2}
+nD_1(1+D_2) \norm{y}{V}^2(1+\norm{y}{V}^{\|p\|})
\quad\mbox{for all}\quad \gamma_4>0,
\intertext{with}
\|p\|&=\max_{1\le j\le n}p_j,\qquad
D_1=\left(1+(\tfrac{\gamma_4}2)^{-\frac{1+\|\delta_2\|}{1-\|\delta_2\|} }\right) \overline C_{\clN1},\\ \quad\mbox{and}\quad
D_2&=\max_{
\begin{subarray}{c}1\le j\le n\\ \zeta_{2j}\ne0\end{subarray}}
\tfrac{1-\delta_{2j}-\zeta_{2j}}{1-\delta_{2j}}
  (\tfrac{1-\delta_{2j}}{\zeta_{2j}}\tfrac{\gamma_4}{2nD_1})^{-\frac{\zeta_{2j}}{1-\delta_{2j}-\zeta_{2j}}}.
\end{align*}
This ends the proof.\qed


{\small
\bibliographystyle{plainurl}
\bibliography{ParabolicDNN}
}

\bigskip
\paragraph{\bf Acknowledgements.}  Karl Kunisch and S\'ergio Rodrigues were supported in part by the ERC advanced grant 668998 (OCLOC) under the EU’s H2020 research program.
S\'ergio Rodrigues also acknowledges partial support from Austrian Science Fund (FWF): P 33432-NBL.

\end{document}